\documentclass[10pt]{amsart}
\usepackage[marginratio=1:1,height=530pt,width=400pt]{geometry}
\usepackage{url}
\usepackage[english]{babel}
\usepackage{mathrsfs}
\usepackage{amssymb}
\usepackage{latexsym}
\usepackage{amsmath}
\usepackage{mathtools}
\usepackage{graphicx}
\usepackage{color}
\usepackage[all]{xy}
\usepackage{pstricks}
\usepackage{pst-plot}
\usepackage{epsfig}
\usepackage[hidelinks]{hyperref}

\usepackage{comment}

\theoremstyle{plain}
\newtheorem{thm}{Theorem}[section]
\newtheorem*{conj}{Conjecture}
\newtheorem{propo}[thm]{Proposition}
\newtheorem{lem}[thm]{Lemma}
\newtheorem{cor}[thm]{Corollary}

\theoremstyle{definition}
\newtheorem*{def-unnumbered}{Definition}
\newtheorem{def-numbered}[thm]{Definition}
\theoremstyle{remark}
\newtheorem{rem}[thm]{Remark}

\definecolor{grey30}{rgb}{0.3,0.3,0.3}
\definecolor{grey67}{rgb}{0.67,0.67,0.67}
\definecolor{newcyan}{RGB}{105,164,216}
\definecolor{neworange}{RGB}{215,125,57}

\newcommand{\tif}{\text{if }}

\newcommand{\tand}{\text{ and }}

\newcommand{\tfor}{\text{for }}

\newcommand{\ga}{\alpha}
\newcommand{\gb}{\beta}

\newcommand{\gd}{\delta}
\newcommand{\gep}{\varepsilon}

\newcommand{\gga}{\gamma}

\newcommand{\gj}{\theta}

\newcommand{\gs}{\sigma}

\newcommand{\gt}{\tau}

\newcommand{\gw}{\omega}
\newcommand{\gx}{\xi}

\newcommand{\gz}{\zeta}

\newcommand{\gL}{\Lambda}

\newcommand{\gY}{\Psi}
 
\newcommand{\C}[1]{{\mathcal{#1}}} 
\newcommand{\D}[1]{{\mathbb{#1}}} 
\newcommand{\E}[1]{{\mathscr{#1}}} 

\newcommand{\refS}[1]{Section~\ref{#1}} 

\newcommand{\refT}[1]{Theorem~\ref{#1}}
\newcommand{\refL}[1]{Lemma~\ref{#1}}
\newcommand{\refD}[1]{Definition~\ref{#1}}
\newcommand{\refP}[1]{Proposition~\ref{#1}}

\newcommand{\refE}[1]{Equation~\eqref{#1}}
\newcommand{\refF}[1]{Figure~\ref{#1}}

{\par \samepage}%
{\par}

\newcommand{\ol}{\overline}

\newcommand{\q}{\quad}

\renewcommand{\Im}{\operatorname{Im}}
\renewcommand{\Re}{\operatorname{Re}}

\newcommand{\eps}{\varepsilon}

\newcommand{\co}[1]{^{\circ {#1}}}

\DeclareMathOperator{\sign}{sign}

\begin{document}

\vspace*{-2cm}

\title[Model for renormalisation of irrationally indifferent fixed points]%
{Arithmetic geometric model for the renormalisation of irrationally indifferent attractors}
\author[Davoud Cheraghi]{Davoud Cheraghi}
\email[Davoud Cheraghi]{d.cheraghi@imperial.ac.uk}
\address[Davoud Cheraghi]{Department of Mathematics, Imperial College London, 
London SW7 2AZ, UK}
\subjclass[2010]{37F50 (Primary), 37F10, 46T25 (Secondary)}
\date{\today}

\begin{abstract}
In this paper we build a geometric model for the renormalisation of irrationally indifferent fixed points. 
The geometric model incorporates the fine arithmetic properties of the rotation number at the fixed point. 
Using this model for the renormalisation, we build a topological model for the dynamics of a holomorphic map near an irrationally indifferent fixed point.
We also explain the topology of the maximal invariant set for the model, and also explain the dynamics of the map on the maximal invariant set. 
\end{abstract}

\maketitle

\numberwithin{equation}{section}


\renewcommand{\thethm}{\Alph{thm}}
\section{Introduction}\label{S:intro}
Let $f$ be a holomorphic map defined on a neighbourhood of $0$ in the complex plane $\mathbb{C}$, and assume that $0$ is an 
\textbf{irrationally indifferent fixed point} of $f$, that is, $f$ is of the form $e^{2\pi i \ga} z+ O(z^2)$ near $0$ with $\ga \in \D{R} \setminus \D{Q}$. 
It is well-known that such systems exhibit highly non-trivial dynamical behaviour which depends on the 
arithmetic properties of $\alpha$, see for instance \cite{Sie42,Brj71,Yoc95,McM98,GrSw03,PZ04,Zha11}. 
When the system is unstable near $0$, the local dynamics remains mysterious, even for simple looking non-linear maps such as the quadratic polynomials 
$e^{2\pi i \alpha}z+ z^2$.

A powerful tool for the study of irrationally indifferent fixed points is renormalisation. 
A \textbf{renormalisation scheme} consists of a class of maps (dynamical systems), and an operator which preserves that class of maps. 
The operator assigns a new dynamical system to a given dynamical system, using suitable iterates of the 
original system. 
A fundamental renormalisation scheme for the systems with an irrationally indifferent fixed point is the 
sector renormalisation of Yoccoz, illustrated in \refF{F:renormalisation}.
In this renormalisation, the new system is obtained from considering the return map to a sector landing at $0$, and the sector is 
identified with a neighbourhood of $0$ using  a change of coordinate. 
A remarkable semi-local version of the local Yoccoz renormalisation is built by Inou and Shishikura in \cite{IS06}, 
which is defined on a sector uniformly large in size, and hence captures more dynamical information about the original system.

\begin{figure}[h]
\begin{pspicture}(13,4) 
\psccurve[linecolor=cyan,fillcolor=cyan,fillstyle=solid](.2,2)(2,.2)(3.7,2)(2,3.5)



\pscustom[linewidth=.5pt,linecolor=black,linestyle=dashed,dash=2pt 1pt,fillcolor=green,fillstyle=solid]
{\pscurve(1.5,2)(2,1.79)(2.4,1.535)(2.53,1.4)(2.61,1.25)(2.625,1.1)(2.6,1.02)(2.55,.944)
\pscurve[liftpen=1](2.74,.735)(2.825,.85)(2.85,1.0)(2.76,1.25)(2.64,1.4)(2.43,1.57)(2,1.81)(1.5,2)
}

\rput(1.4,1.9){\small $0$}

\pscurve[origin={1.5,2},linewidth=.5pt]{->}(1.13;332)(1.19;336)(1.12;340)

\pscurve[origin={1.5,2},linewidth=.5pt]{->}(1.14;341)(1.19;346)(1.1;351)

\pscurve[origin={1.5,2},linewidth=.5pt]{->}(1.1;352)(1.15;358)(1.06;365)

\pscurve[origin={1.5,2},linewidth=.5pt]{->}(1.06;7)(1.11;14)(1.02;21) 

\pscurve[origin={1.5,2},linewidth=.5pt]{->}(1.02;21)(1.05;29)(.94;38) 

\pscurve[origin={1.5,2},linewidth=.5pt]{->}(.94;38)(.96;46)(.92;54) 

\pscurve[origin={1.5,2},linewidth=.5pt]{->}(.92;54)(.94;62)(.90;70) 

\pscurve[origin={1.5,2},linewidth=.5pt]{->}(.90;70)(.92;78)(.88;86) 

\pscurve[origin={1.5,2},linewidth=.5pt]{->}(.88;86)(.90;94)(.86;102) 

\pscurve[origin={1.5,2},linewidth=.5pt]{->}(.86;102)(.88;110)(.84;118) 

\pscurve[origin={1.5,2},linewidth=.5pt]{->}(.84;118)(.86;126)(.82;134)

\pscurve[origin={1.5,2},linewidth=.5pt]{->}(.82;134)(.84;142)(.80;150)

\pscurve[origin={1.5,2},linewidth=.5pt]{->}(.80;150)(.82;158)(.78;166)

\pscurve[origin={1.5,2},linewidth=.5pt]{->}(.78;166)(.8;174)(.76;182)

\pscurve[origin={1.5,2},linewidth=.5pt]{->}(.76;182)(.78;190)(.74;198)

\pscurve[origin={1.5,2},linewidth=.5pt]{->}(.74;198)(.76;206)(.72;214)

\pscurve[origin={1.5,2},linewidth=.5pt]{->}(.75;214)(.81;224)(.79;232)

\pscurve[origin={1.5,2},linewidth=.5pt]{->}(.79;232)(.84;240)(.82;248)

\pscurve[origin={1.5,2},linewidth=.5pt]{->}(.83;248)(.88;256)(.86;264)

\pscurve[origin={1.5,2},linewidth=.5pt]{->}(.86;264)(.91;272)(.89;280)

\pscurve[origin={1.5,2},linewidth=.5pt]{->}(.89;280)(.94;288)(.92;296)

\pscurve[origin={1.5,2},linewidth=.5pt]{->}(.92;296)(.97;304)(.95;312)

\pscurve[origin={1.5,2},linewidth=.5pt]{->}(.95;312)(1.01;318)(.98;325)

\pscurve[origin={1.5,2},linewidth=.5pt]{->}(.98;325)(1.05;330)(1.03;335)

\rput(2.5,3){\small $f$}

\psset{xunit=1.4,yunit=1.4}

\pscustom[origin={2.8,-.4},linewidth=.5pt,linecolor=black,fillcolor=green,fillstyle=solid]
{\pscurve(1.5,2)(2,1.79)(2.4,1.535)(2.53,1.4)(2.61,1.25)(2.625,1.1)(2.6,1.02)(2.55,.944)
\pscurve[liftpen=1](2.74,.735)(2.825,.85)(2.85,1.0)(2.76,1.25)(2.64,1.4)(2.43,1.57)(2,1.81)(1.5,2)
}








\pscurve[origin={5.45,.45},linewidth=.5pt,linecolor=gray]{->}(0.19;135)(.35;90)(.35;350)(.2;310)

\pscurve[origin={5.5,.8},linewidth=.5pt,linecolor=gray]{->}(0.13;220)(.35;180)(.35;80)(.12;40)

\pscurve[origin={5.285,1.1},linewidth=.5pt,linecolor=gray]{->}(0.075;180)(.35;160)(.35;90)(.045;70)

\pscurve[origin={4.91,1.3475},linewidth=.5pt,linecolor=gray]{->}(0.02;230)(.35;190)(.35;90)(.02;45)

\pscurve[linewidth=.5pt]{<-}(5.5,0.5)(5.3,0.1)(5.35,0.45)(5.45,.6)
\rput(5.3,-.1){\small return map}

\psset{xunit=1.0,yunit=1.0}

\psccurve[linecolor=green,fillcolor=green,fillstyle=solid](7.2,1.5)(8.2,.5)(9.2,1.5)(8.2,2.4)

\pscircle[linewidth=.5pt](8.2,1.5){.03}

\pscurve[linewidth=.5pt]{->}(6,2.3)(6.7,2.5)(7.4,2.3)
\rput(6.7,2.7){\small uniformisation ($z \sim f(z)$)}

\pscurve[linewidth=.5pt]{->}(8,.8)(8.2,.3)(8.4,.8)
\rput(8.2,0){\small renormalisation of $f$}
\end{pspicture}
\label{F:renormalisation}
\caption{In the left hand picture, 
the (canonically defined) sector landing at $0$ is bounded by a curve landing at $0$, the image of that curve by $f$, and a line segment connecting the end points of those curves.
By glueing the sides of the sector according to $f$, one obtains a Riemann surface, isomorphic to a 
punctured disk.
}
\end{figure}
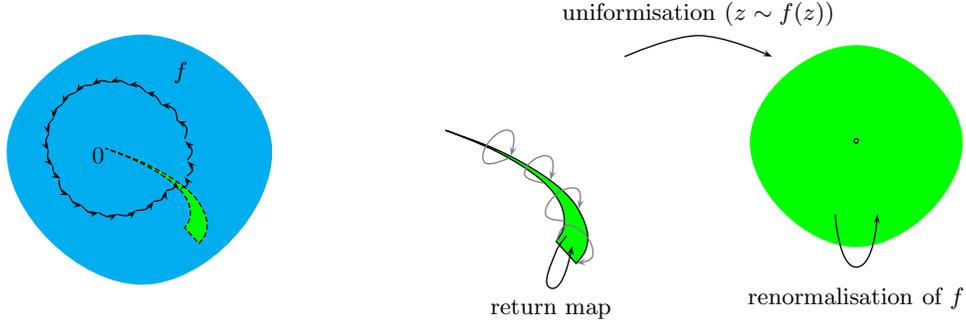

A main technical issue with employing a renormalisation method for irrationally indifferent fixed points is the highly distorting nature of the 
changes of coordinates that appear in successive applications of the renormalisation.
One is lead to analysing the local and global distorting behaviours of the successive changes of coordinates 
in conjunction with the fine arithmetic features of $\alpha$.

In this paper with build a geometric model for the (semi-local version of the) renormalisation of irrationally 
indifferent fixed points. 

\smallskip 

\begin{thm}[Renormalisation model]\label{T:toy-model-renormalisation}
There exists a class of maps  
\[\mathbb{F}=
\big\{\mathbb{T}_\alpha: \mathbb{M}_\alpha \to \mathbb{M}_\alpha \big\}
_{\alpha \in \mathbb{R}\setminus \mathbb{Q}}\] 
and a renormalisation operator $\mathcal{R}_m:\mathbb{F} \to \mathbb{F}$ satisfying the following properties: 
\begin{itemize}
\item[(i)]for every $\alpha \in \mathbb{R} \setminus \mathbb{Q}$, 
$\mathbb{M}_\alpha \subset \mathbb{C}$ is a compact star-like set with $\{0, +1\} \subset \mathbb{M}_\alpha$, 
\item[(ii)]for every $\alpha \in \mathbb{R} \setminus \mathbb{Q}$, 
$\mathbb{T}_\alpha: \mathbb{M}_\alpha \to \mathbb{M}_\alpha$ is a homeomorphism which 
acts as rotation by $2\pi \alpha$ in the tangential direction,
\item[(iii)] for $\ga \in (-1/2, 1/2) \setminus \mathbb{Q}$, 
\[\mathcal{R}_m \left(\mathbb{T}_\alpha: \mathbb{M}_\alpha \to \mathbb{M}_\alpha\right)
= \left(\mathbb{T}_{-1/\alpha}: \mathbb{M}_{-1/\alpha} \to \mathbb{M}_{-1/\alpha}\right),\] 
\item[(iv)] for every $\alpha \in \mathbb{R} \setminus \mathbb{Q}$ and every integer $k$ with 
$0 \leq k \leq 1/|\alpha|$, 
\[\frac{C^{-1}}{1+ \min \{k, |\ga|^{-1} - k\}} \leq |\mathbb{T}_\ga\co{k}(+1)| \leq \frac{C}{1+ \min \{k, |\ga|^{-1}-k\}},\]
for some constant $C$ independent of $k$ and $\alpha$. 
\item[(v)] for every $\alpha \in \mathbb{R} \setminus \mathbb{Q}$, 
$\mathbb{M}_{\alpha+1}= \mathbb{M}_{\alpha}$, $\mathbb{T}_{\alpha+1}=\mathbb{T}_{\alpha}$, 
$\mathbb{M}_{-\alpha}= s(\mathbb{M}_{\alpha})$, and $\mathbb{T}_{-\alpha}=s \circ \mathbb{T}_{\alpha} \circ s$, 
where $s$ denotes the complex conjugation map, 
\item[(vi)] $\mathbb{M}_\alpha$ depends continuously on $\alpha \in \mathbb{R} \setminus \mathbb{Q}$, 
in the Hausdorff topology. 
\end{itemize} 
\end{thm}

The map $\mathbb{T}_\ga$ is a topological model for the map $f(z)=e^{2\pi i \ga}z+ O(z^2)$. 
The set $\mathbb{M}_\ga$ is a topological model for the maximal invariant set of $f$ at $0$ on which $f$ is injective. 
The model for the renormalisation, $\mathcal{R}_m$, is defined using the return map of $\mathbb{T}_\ga$ to a cone of angle $2\pi \ga$ landing at $0$;  
with a change of coordinate which preserves rays landing at 0, while exhibiting the non-linear behaviour of the actual change of coordinate for the sector 
renormalisation of $f$. 
As in the sector renormalisation, $\mathcal{R}_m$ induces the Gauss map $\ga \mapsto -1/\ga$ on the asymptotic rotation numbers at $0$.

The point $+1$ is the largest real number which belongs to $\mathbb{M}_\ga$. 
It plays the role of a certain critical point of $f$, which we will explain in a moment.
The geometry of the orbit of $+1$ under $\mathbb{T}_\ga$, for one return of the dynamics, is explained in part (iv) of \refT{T:toy-model-renormalisation}. 
This form of geometry is ubiquitous for maps with an irrationally indifferent fixed point; it requires a non-zero second derivative at $0$. 
For instance, it holds for the quadratic polynomials $e^{2\pi i \ga} z + z^2$.

By employing the renormalisation scheme $(\mathcal{R}_m, \mathbb{F})$, we build a topological model for the dynamics 
near an irrationally indifferent fixed point, as we explain below. 

By the classic work of Fatou \cite{Fat19} and Mane \cite{Man87}, when $f$ is a polynomial or a rational function of
the Riemann sphere, there is a recurrent critical point of $f$ which ``interacts'' with the fixed point at $0$. 
The topological boundary of $\mathbb{M}_\ga$, 
\[\mathbb{A}_\ga= \partial \mathbb{M}_\ga,\]
is equal to the closure of the orbit of $+1$ for the iterates of $\mathbb{T}_\ga$. 
The set $\mathbb{A}_\ga$ is a topological model for the closure of the orbit of that recurrent critical point of $f$, 
which is the measure theoretic attractor of $f$ for the orbits that remain near $0$. 

We explain the topology of the sets $\mathbb{A}_\ga$ in terms of the arithmetic nature of $\ga$. 

\smallskip

\begin{thm}[Trichotomy of the maximal invariant set]\label{T:model-trichotomy-thm}
For every $\ga \in \mathbb{R} \setminus \mathbb{Q}$ one of the following statements hold: 
\begin{itemize}
\item[(i)] $\ga$ is a Herman number and $\mathbb{A}_\alpha$ is a Jordan curve,
\item[(ii)] $\ga$ is a Brjuno but not a Herman number, and $\mathbb{A}_\alpha$ is a one-sided hairy Jordan curve,
\item[(iii)] $\ga$ is not a Brjuno number, and $\mathbb{A}_\alpha$ is a Cantor bouquet. 
\end{itemize}
\end{thm}

\begin{figure}[t]
\begin{pspicture}(14,5) 
\epsfxsize=4cm
\rput(2.5,2.5){\epsfbox{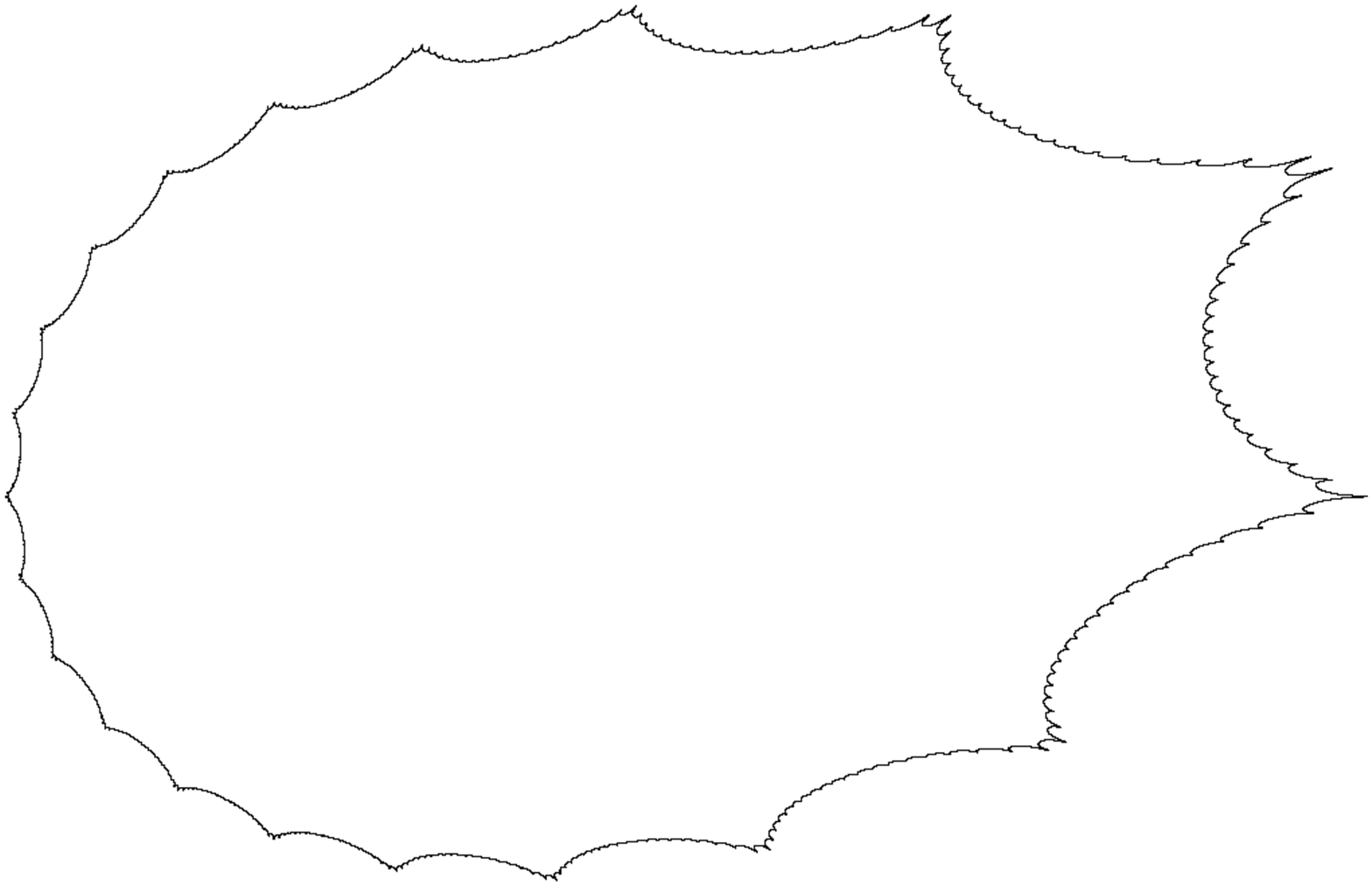}}
\rput(7,2.5){\epsfbox{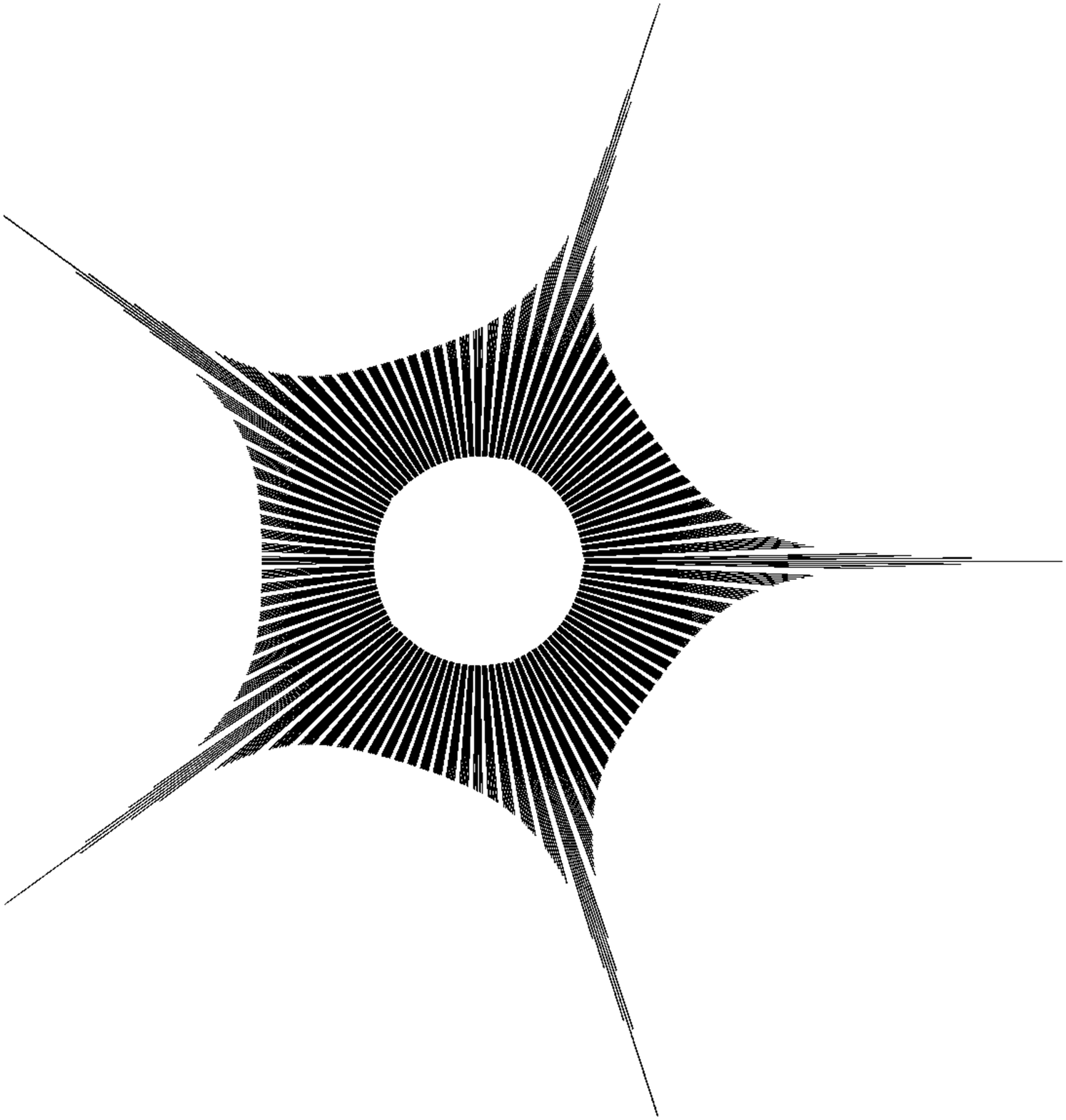}}
\rput(11.5,2.5){\epsfbox{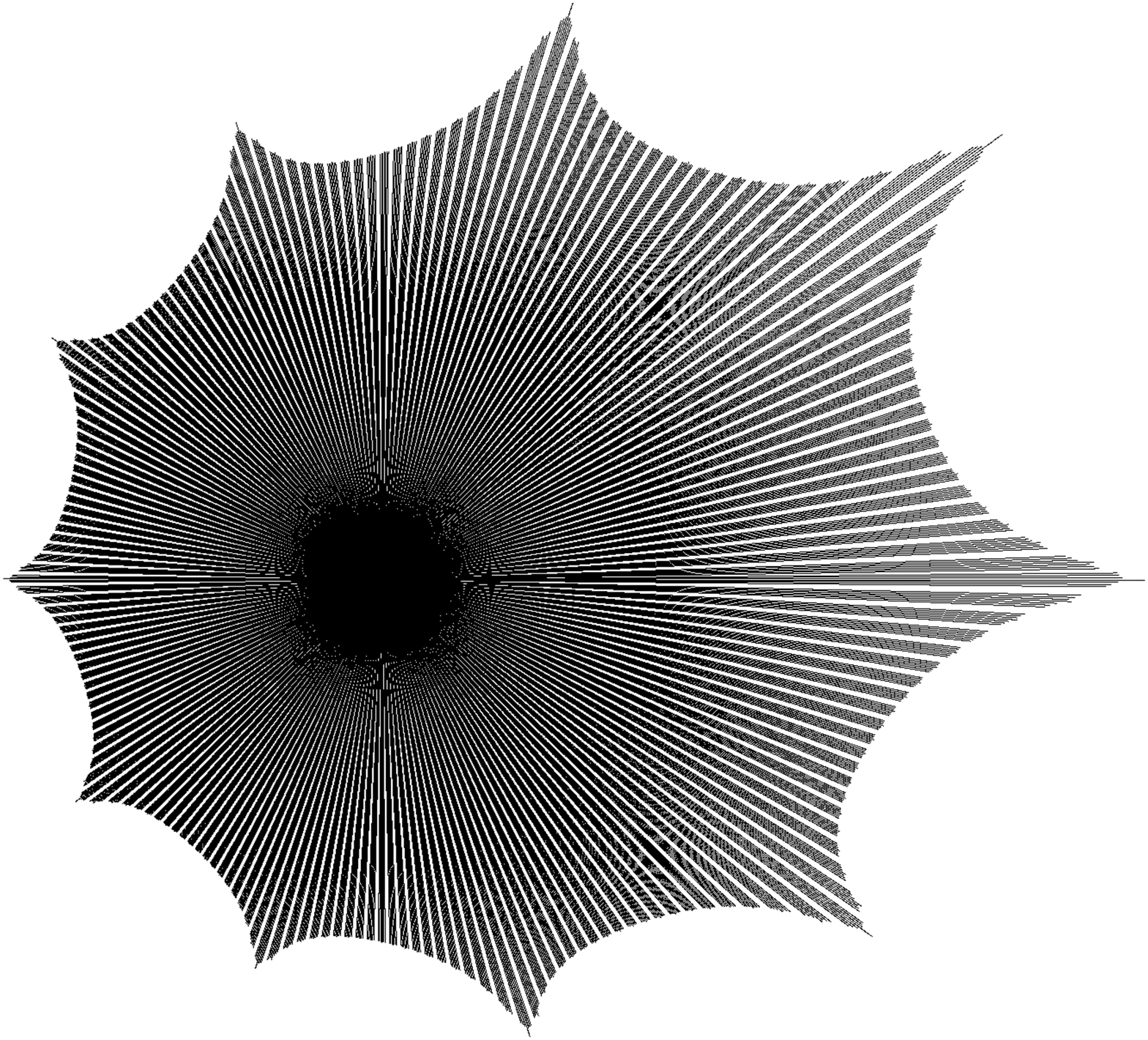}}
\end{pspicture}
\caption{Computer simulations for the three topologies in \refT{T:model-trichotomy-thm}; from left to right, a Jordan curve, a hairy Jordan curve, and a Cantor bouquet.}
\end{figure}

The arithmetic class of Brjuno was discovered by Siegel-Brjuno \cite{Sie42,Brj71} in their now classic work on the study of the linearisation problem for 
irrationally indifferent fixed points. 
It is the set of $\ga$ such that the denominators $q_n$ of the best rational approximants of $\ga$ satisfy
\[\textstyle{\sum}_{n=1}^\infty q_n^{-1} \log q_{n+1} < +\infty.\]

The arithmetic class of Herman was discovered by Herman-Yoccoz \cite{Her79,Yoc02} 
in their landmark study of the analytic linearisation problem for analytic diffeomorphisms of the circle. 
The set of Herman numbers is more complicated to characterise in terms of the arithmetic of $\ga$, 
see \refS{S:arithmetic}. It forms an $F_{\gs,\gd}$ subset of $\D{R}$. 
However, any Herman number is a Brjuno number, the set of Herman numbers has full measure in $\D{R}$, 
while the set of non-Brjuno numbers is topologically generic in $\D{R}$.
The arguments presented in this paper do not rely on the optimality properties of the Brjuno and Herman 
numbers for the linearisation problems. 
These arithmetic conditions naturally emerge in the study of the model $\mathbb{A}_\ga$.

A Cantor bouquet is a compact subset of $\mathbb{C}$ which has empty interior, consists of a collection of Jordan arcs only meeting at $0$, 
and every arc is accumulated from both sides by arcs. 
A (one-sided) hairy Jordan curve is a similar object consisting of a collection of Jordan arcs all attached to (one side of) a Jordan arc. 
See \refS{SS:CB-HJC} for the precise definitions of these objects. 
While one-sided hairy Jordan curves and Cantor bouquets are universal topological objects, 
our construction is geometric, featuring delicate metric properties. 
That is, the models for dissimilar values of non-Brjuno numbers, and also for dissimilar values of Brjuno but non-Herman numbers, 
have very distinct metric properties. 

The maps $\mathbb{T}_\ga: \mathbb{A}_\ga \to \mathbb{A}_\ga$ exhibit peculiar dynamical behaviour, especially 
when $\ga$ is not a Herman number. 
Every point in $\mathbb{A}_\ga$ is topologically recurrent. 
There are points in $\mathbb{A}_\ga$ with dense orbits. 
There are uncountably many distinct closed fully invariant sets. 
In contrast to the linearisable examples where an invariant set consists of disjoint unions of closed analytic curves, here there is only a 
one-parameter family of closed invariant sets. 
The set of accumulation points of the orbit of a point $z\in \mathbb{A}_\ga$, denoted by $\omega(z)$, is a notable example of a closed invariant set. 
It turns out that these are the only ones, as we state below. 

Define $r_\ga \in [0,1]$ according to 
\[[r_\ga, 1]=\{ z \in \mathbb{A}_\ga \mid \Im z=0 , \Re z \geq 0\}.\]

\begin{thm}[Topological dynamics]\label{T:dynamics-on-model}
For every $\ga \in \mathbb{R} \setminus \mathbb{Q}$ the map 
$\mathbb{T}_\alpha: \mathbb{A}_\alpha \to \mathbb{A}_\alpha$ satisfies the following properties:
\begin{itemize}
\item[(i)] the map $\mathbb{T}_\alpha: \mathbb{A}_\alpha \to \mathbb{A}_\alpha$ is topologically recurrent.
\item[(ii)] the map 
\[\omega: [r_\ga, 1] \to \{X \subseteq \mathbb{A}_\ga \mid X \text{ is non-empty, closed and invariant}\}\]
is a homeomorphism with respect to the Hausdorff metric on the range.
\item[(iii)] the map $\omega$ on $[r_\ga, 1]$ is strictly increasing with respect to 
the inclusion. 
\item[(iv)]if $\ga$ is a Brjuno number, for every $t \in (r_\ga, 1]$,  $\omega(t)$ is a hairy Jordan curve. 
\item[(v)] if $\ga$ is not a Brjuno number, for every $t \in (r_\ga, 1]$,  $\omega(t)$ is a Cantor bouquet. 
\end{itemize}
\end{thm}

Based on the above results for the model, we propose the following conjecture. 

\begin{conj}[Trichotomy of the irrationally indifferent attractors]\label{C:trichotomy-main-conj}
Let $f$ be a rational function of the Riemann sphere $\hat{\mathbb{C}}$ of degree at least $2$, with 
$f(0)=0$, $f'(0)=e^{2\pi i \ga}$, and $\ga \in \D{R} \setminus \D{Q}$. 
There exists a critical point of $f$, such that the closure of its orbit, denoted by $\gL(f)$, satisfies the following: 
\begin{itemize}
\item[(i)] if $\ga$ is a Herman number, $\gL(f)$ is a Jordan curve,
\item[(ii)] if $\ga$ is a Brjuno but not a Herman number, $\gL(f)$ is a one-sided hairy Jordan curve, 
\item[(iii)] if $\ga$ is not a Brjuno number, $\gL(f)$ is a Cantor bouquet. 
\end{itemize}
Moreover, in cases (i) and (ii), $f$ preserves the connected component of $\hat{\D{C}} \setminus \gL(f)$ 
containing $0$, and in case (iii), $\gL(f)$ contains $0$. 
\end{conj}

The above conjecture immediately implies a number of important conjectures in the study of irrationally indifferent 
fixed points for rational functions. 
Notably, it implies the optimality of the Brjuno condition for the linearisability of irrationally indifferent fixed 
points (Douady-Yoccoz conjecture 1986), the optimality of the Herman condition for the presence of a critical 
point on the boundary of the Siegel disk (Herman conjecture 1985), 
the Siegel disk are Jordan domains (Douady-Sullivan conjecture 1987). 
On the other hand, for rational functions, the invariant Siegel-compacta and hedgehogs introduced by
Perez-Marco \cite{PM97} are contained in the post-critical set, and hence their topology may be 
completely explained by the above conjecture.


In a counterpart paper \cite{Che17}, we employ the topological model in \refT{T:toy-model-renormalisation}, 
and the near-parabolic renormalisation scheme of Inou and Shishikura \cite{IS06}, 
to prove the above conjecture for a class of maps and rotation numbers.
Indeed, we present some sufficient condition for a renormalisation scheme in order to conclude the above 
conjecture. 
In order to explain this relation, we briefly outline the construction of the model.  
Here we take an upside-down approach to the renormalisation. 
In contrast to defining renormalisation for a given map defined on a domain, we start by building 
the changes of coordinates for the renormalisation, one for each rotation number. 
Repeatedly applying the Gauss map to a fixed $\ga$, one obtains a sequence of rotation numbers, 
and the corresponding changes of coordinates. 
Then, there is a maximal set on which the infinite chain of the changes of coordinates is defined. 
Finally we build the map $\mathbb{T}_\ga$ on the maximal set so that its return map via the change of 
coordinate becomes $\mathbb{T}_{G(\ga)}$. 
The advantage of this approach is that any given renormalisation scheme for irrationally indifferent fixed 
points may be compared to the model $(\mathcal{R}_m, \mathbb{F})$ by comparing the corresponding 
changes of coordinates. 
This provides a streamlined approach to the study of the dynamics of irrationally indifferent fixed points. 
It is also unified, in the sense that one does not need to reconsider the role of the arithmetic properties in 
a given renormalisation scheme. 
In the counterpart paper \cite{Che17}, we show that if a renormalisation scheme consists of a change of 
coordinate which is sufficiently close to the change of coordinate for the renormalisation model, 
then the corresponding maps are topologically conjugate on the corresponding maximal domain of 
renormalisations. 
See also \cite{ShY18} for partial progress towards the above conjecture. 
 
An alternative construction for $\mathbb{M}_\ga$ was suggested by Buff and Ch\'eritat in 2009 \cite{BC09}. 
Our thought in this direction was influenced and motivated by their construction based on employing 
toy models for the changes of coordinates.
In contrast to the conformal changes of coordinates considered by those authors, the models for the changes 
of coordinates presented here are not conformal, but preserve straight rays landing at $0$, while maintaining 
the correct nonlinear behaviour in the radial directions. 
This flexibility allows us to incorporate some remarkable functional relations, which in turn allow us to avoid 
taking Hausdorff limits in the construction. 
The explicit construction presented here allow us to promote the models for $\mathbb{M}_\ga$ to build 
the maps $\mathbb{T}_\ga$, and a renormalisation scheme on those maps. 

\tableofcontents

\renewcommand{\thethm}{\thesection.\arabic{thm}}

\section{The changes of coordinates for the renormalisation \texorpdfstring{$\mathcal{R}_m$}{Rm}}\label{S:change-coordinates}
In this section we introduce a one-parameter ($r\in (0, 1/2]$) family of real analytic diffeomorphisms. 
When we set the parameter as a rotation number, the diffeomorphism becomes the change of coordinate for the 
renormalisation of a map with that asymptotic rotation number at $0$. 
These are the toy models for the change of coordinates in the toy model for the renormalisation scheme. 

\subsection{Explicit formula for the changes of coordinates}\label{SS:coordinate-core}
Consider the set  
\[\D{H}'= \{w\in \D{C} \mid \Im w > -1 \}.\]
For $r \in (0,1/2]$, we define the map $Y_r : \ol{\D{H}'} \to \D{C}$ 
as \footnote{$\ol{X}$ denotes the topological closure of a given set $X$.} 
\[Y_r(w)= r\Re w + 
\frac{ i }{2\pi} \log \Big |\frac{e^{-3\pi r}- e^{-\pi r  i }e^{-2\pi r i  w}}{e^{-3\pi r}- e^{\pi r  i }}\Big|.\]
Evidently, $Y_r$ maps a vertical line in $\D{H}'$ to a vertical line. 
Also, one can see that $Y_r(0)=0$ and that $Y_r(1/(2r)-i)$ is uniformly close to $1/2+ i (\log 1/r)/(2\pi)$. 
\refF{F:model-Y_r} shows the behaviour of $Y_r$ on horizontal and vertical lines.

\begin{figure}[ht]
\includegraphics[width=0.8\textwidth]{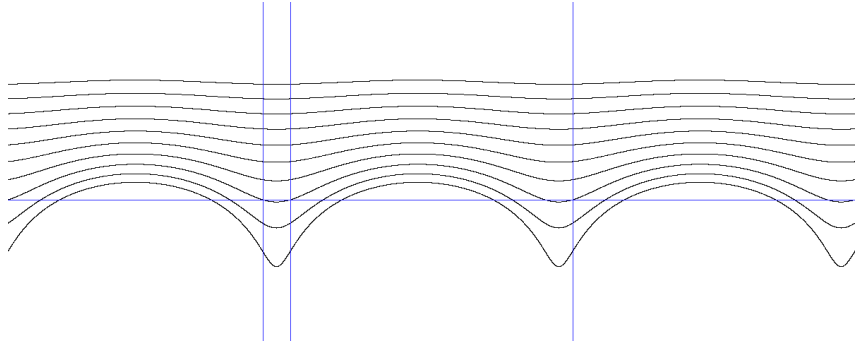}
\caption{The nearly horizontal curves are the images of the horizontal lines $y=-1, 0, 1, \dots, 8$ 
under $Y_r$. 
The vertical lines, from left to right, are the images of the vertical lines $\Re w=-1$, $\Re w=0$, 
and $\Re w=1/\ga$, under $Y_r$. Here, $r= 1/(10+1/(1+1/(1+1/(1+\dots))))$.}
\label{F:model-Y_r}
\end{figure}

\begin{lem}\label{L:Y-domain}
For every $r \in (0, 1/2]$, $Y_r: \ol{\D{H}'} \to \D{H}'$ is well-defined, real-analytic in $\Re w$ and $\Im w$, and is 
injective. 
Moreover, for every $r \in (0, 1/2)$ and every $y \geq 0$, we have 
\[H_r \left (\{w\in \D{C} \mid \Im w \geq y-1 \}\right ) \subset \{w\in \D{C} \mid \Im w \geq \Im H_r(i y)-1\}.\] 
\end{lem}

\begin{proof}
The proof is based on some elementary calculations. Note that 
\[10 \leq 1+ (12/5)+ (12/5)^2/2!+ \dots + (12/5)^5/5! \leq e^{12/5}.\] 
Using $0 \leq r \leq 1/2$, $3\leq \pi \leq 4$, and the above inequality, we obtain  
\[\log (4+ 3\pi r) + 2\pi r \leq    \log(e^{12/5})+ \pi  \leq \log e^{4\pi/5} +\pi \leq 9\pi/5.\]
This implies that for all $r\in [0, 1/2]$, $e^{9\pi/5}  \geq  (4+3\pi r)e^{2\pi r}$, and therefore, 
\[\pi e^{\pi r} (e^{9\pi/5}  - (4+3\pi r)e^{2\pi r})) \geq 0.\]
Fix $r\in [0,1/2]$. Integrating the above inequality from $0$ to $r$, we conclude that 
\[e^{9\pi/5} (e^{\pi r} - 1) -(\pi r+1)e^{3\pi r} +1 \geq 0.\]
This implies that for all $r\in (0,1/2]$, we have 
\begin{equation}\label{E:proof-Y_r-basic-1}
\frac{e^{\pi r} -1}{(\pi r+1)e^{3\pi r} -1} \geq  e^{-9\pi/5}.
\end{equation}
On the other hand, for $w\in \ol{\D{H}'}$, by the triangle inequality, 
\[|e^{-3\pi r} - e^{-\pi r  i }e^{-2\pi r i  w} | \geq |e^{-\pi r  i }e^{-2\pi r i  w} | - |e^{-3\pi r}| 
\geq e^{-2\pi r} - e^{-3\pi r},\]
and  
\[|e^{-3\pi r}- e^{\pi r  i }| \leq |e^{-3\pi r}- 1| + |1-e^{\pi r  i }| \leq (1- e^{-3\pi r}) + \pi r.\]
Combining the above two inequalities with \refE{E:proof-Y_r-basic-1} we obtain 
\[\left| \frac{e^{-3\pi r}- e^{-\pi r  i }e^{-2\pi r i  w}}{e^{-3\pi r}- e^{\pi r  i }} \right| 
\geq \frac{e^{-2\pi r} - e^{-3\pi r}}{\pi r + 1- e^{-3\pi r}}= \frac{e^{\pi r} -1}{(\pi r+1)e^{3\pi r} -1} 
\geq e^{-9\pi/5}.\]
The above inequality implies that for all $w\in \D{H}'$ we have 
\[\Im Y_r(w) \geq  \frac{1}{2\pi} \log e^{-9\pi/5} > -1.\]
In particular, $Y_r$ is well-defined, and maps $\ol{\D{H}'}$ into $\D{H}'$. This also implies that 
$Y_r$ is real-analytic in $\Re w$ and $\Im w$, for $w$ in $\ol{\D{H}'}$. 

To prove that $Y_r$ is injective, assume that $w_1$ and $w_2$ are two distinct points in $\ol{\D{H}'}$. 
If $\Re w_1 \neq \Re w_2$, then $\Re Y_r(w_1) \neq \Re Y_r(w_2)$. 
If $\Re w_1 = \Re w_2$ but $\Im w_1 \neq \Im w_2$, then 
\[|e^{-3\pi r} - e^{-\pi r  i } e^{-2\pi r  i  w_1}| \neq |e^{-3\pi r} - e^{-\pi r  i } e^{-2\pi r  i  w_2}|.\]
This implies that $\Im Y_r(w_1) \neq \Im Y_r(w_2)$. 

One may employ similar elementary calculations to derive the latter part of the lemma. 
\end{proof}

\subsection{Uniform contraction of the changes of coordinates}
A key property of the maps $Y_r$ is stated in the following lemma. 

\begin{lem}\label{L:uniform-contraction-Y_r}
For every $ r\in (0,1/2]$, and every $w_1, w_2$ in $\ol{\D{H}'}$, we have 
\[|Y_r(w_1)- Y_r(w_2)| \leq 0.9 |w_1-w_2|.\]
\end{lem}
The precise contraction factor $0.9$ in the above lemma is not crucial, any constant strictly less than $1$ suffices. 

\begin{proof}
Let $g(w)= (e^{-3\pi r}- e^{-\pi r  i } e^{-2\pi r  i  w})(e^{-3\pi r}- e^{\pi r  i } e^{2\pi r  i  \ol{w}})$. 
Then, $g(w)$ is of the form $\gz \ol{\gz}$, for some $\gz\in \D{C}$, and hence it produces positive real values for 
$w\in \ol{\D{H}'}$. 
We have 
\[\partial g(w)/\partial w
= 2\pi r  i  e^{-\pi r i } e^{-2\pi r  i  w}  (e^{-3\pi r}- e^{\pi r  i } e^{2\pi r  i  \ol{w}}),\]
and 
\[\partial g(w)/\partial \ol{w}
= -2 \pi r  i  e^{\pi r i } e^{2\pi r  i  \ol{w}}  (e^{-3\pi r}- e^{-\pi r  i } e^{-2\pi r  i  w}) .\]
Therefore, by the complex chain rule, 
\[\frac{\partial }{\partial w}\big(\log g(w)\big) = \frac{1}{g(w)} \frac{\partial g}{\partial w} 
= \frac{2\pi r  i  e^{-\pi r i } e^{-2\pi r  i  w}}{e^{-3\pi r}- e^{-\pi r  i } e^{-2\pi r  i  w}}
=\frac{2\pi r  i }{e^{-3\pi r} e^{\pi r i } e^{2\pi r  i  w} - 1},\]
and 
\[\frac{\partial}{\partial \ol{w}}\big(\log g(w)\big) = \frac{1}{g(w)} \frac{\partial g}{\partial \ol{w}} 
= \frac{-2 \pi r  i  e^{\pi r i } e^{2\pi r  i  \ol{w}}}{e^{-3\pi r}- e^{\pi r  i } e^{2\pi r  i  \ol{w}}}
=\frac{-2 \pi r  i }{e^{-3\pi r}  e^{-\pi r  i } e^{-2\pi r  i  \ol{w}}-1}.\]
We rewrite $Y_r$ in the following form 
\[Y_r(w)
=r \cdot \frac{w+\ol{w}}{2} + \frac{ i }{2\pi} \cdot \frac{1}{2} \log g(w) - 
\frac{ i }{2\pi} \log |e^{-3\pi r}-e^{\pi r i }|.\]
Then, by the above calculations,  
\[\frac{\partial Y_r}{\partial w}(w) 
= \frac{r}{2} + \frac{ i }{4\pi} \cdot \frac{2\pi r  i }{e^{-3\pi r} e^{\pi r i } e^{2\pi r  i  w} - 1}
= \frac{r}{2} \Big (1 - \frac{1}{e^{-3\pi r} e^{\pi r i } e^{2\pi r  i  w} - 1}\Big ),\]
and 
\[\frac{\partial Y_r}{\partial \ol{w}}(w) 
=\frac{r}{2} + \frac{ i }{4\pi} \cdot \frac{-2 \pi r  i }{e^{-3\pi r}  e^{-\pi r  i } e^{-2\pi r  i  \ol{w}}-1}
= \frac{r}{2} \Big (1+ \frac{1}{e^{-3\pi r}  e^{-\pi r  i } e^{-2\pi r  i  \ol{w}}-1}\Big).\]
Let $\gx=e^{-3\pi r} e^{\pi r i } e^{2\pi r  i  w}$. For $w\in \ol{\D{H}'}$, $|\gx| \leq e^{-\pi r}$. 
For the maximum size of the directional derivative of $Y_r$ we have  
\begin{align*}
\max_{\theta \in [0,2\pi]}  \big | \operatorname{D} Y_r (w) \cdot e^{ i  \theta}\big| 
&= \Big |\frac{\partial Y_r}{\partial w}(w) \Big |  + \Big |  \frac{\partial Y_r}{\partial \ol{w}}(w) \Big |  \\
&\leq \frac{r}{2} \cdot  \Big |1-\frac{1}{\gx -1}\Big | + \frac{r}{2} \cdot \Big |1+ \frac{1}{\ol{\gx}-1}\Big | \\  
&\leq \frac{r}{2} \cdot \frac{2+ e^{-\pi r}}{1-e^{-\pi r}} + \frac{r}{2} \cdot \frac{e^{-\pi r}}{1-e^{-\pi r}} 
=r \cdot \frac{e^{\pi r}+1 }{e^{\pi r}-1}.
\end{align*}
For $r \geq 0$, $e^{\pi r}-1 \geq \pi r + \pi ^2 r^2/2$, (the first two terms of the Taylor series with positive terms). 
This gives us
\[r \cdot \frac{e^{\pi r}+1 }{e^{\pi r}-1} = r\left (1+ \frac{2}{e^{\pi r}-1}\right ) 
\leq r\left (1+ \frac{2}{\pi r + \pi ^2 r^2/2}\right)
= \frac{2\pi r + \pi^2 r^2+4}{2\pi + \pi^2 r}.\]
The last function in the above equation is increasing on $(0, 1/2)$,  
because it has a non-negative derivative $(4\pi r+\pi^2 r^2)/(2+\pi r)^2$. 
Then, it is bounded by its value at $1/2$, which, using $\pi \geq 3$, gives us
\[\frac{2\pi r + \pi^2 r^2+4}{2\pi + \pi^2 r} 
\leq \frac{\pi +\pi^2/4+4}{2\pi+ \pi^2/2} 
= \frac{1}{2}+ \frac{4}{2\pi+\pi^2/2} 
\leq \frac{1}{2} + \frac{4}{6+4} = \frac{9}{10}.
\qedhere\]
\end{proof}

\subsection{Remarkable functional relations for the changes of coordinates}
The maps $Y_r$ satisfy two crucial functional relations, one at a large scale, and one at a small scale, 
both of which are illustrated in \refF{F:model-Y_r}. We present these properties in the following lemma. 

\begin{lem}\label{L:Y_r-commutation}
For every $r\in (0,1/2]$, we have  
\begin{itemize}
\item[(i)] for every $w\in \ol{\D{H}'}$, 
\[Y_r(w+1/r) = Y_r(w)+1\] 
\item[(ii)] for every $t \geq -1$, 
\[Y_r( i  t+ 1/r-1)= Y_r( i  t)+ 1- r.\]
\end{itemize}
\end{lem}

\begin{proof}
Part (i) of the lemma readily follows from the formula defining $Y_r$. 

To prove part (ii) of the lemma, first note that  
\begin{align*}
\big |e^{-3\pi r}- e^{-\pi r  i } e^{-2\pi r i(it+1/r -1)}\big|   
& =\big |e^{-3\pi r}- e^{-\pi r  i } e^{2\pi r t} e^{2\pi r i} \big|  \\
&= \big| e^{-3\pi r}- e^{\pi r  i } e^{2\pi r t} \big| =\big |e^{-3\pi r}- e^{-\pi r  i } e^{2\pi r t}\big|.
\end{align*}
Above, the first and second ``$=$'' are simple multiplications of complex numbers, while for the third ``$=$'' we have 
used that $|x-z|=|x-\ol{z}|$, for real numbers $x$ and complex numbers $z$. Thus, 
\begin{equation*}
Y_r( i  t+ 1/r-1)= 
r (1/r-1)+ \frac{i}{2\pi}\log \Big |\frac{e^{-3\pi r}-e^{-\pi r  i } e^{2\pi r t}}{e^{-3\pi r}- e^{\pi r  i }}\Big| 
=(1-r) + Y_{r}( i  t). \qedhere
\end{equation*}
\end{proof} 
\section{The sets \texorpdfstring{$\mathbb{M}_\ga$}{M-ga}}\label{S:M-ga}
In this section we build the sets $\mathbb{M}_\ga$ introduced in \refS{S:intro}.
These sets are defined by successively applying the changes of coordinates $Y_r$, for a sequence of parameters $r$. 

\subsection{Successive rotation numbers \texorpdfstring{$\ga_n$}{ga-n} and signs \texorpdfstring{$\gep_n$}{epsilon-n}}\label{SS:modified-fractions-mini}
For $x\in \D{R}$, define $d(x, \D{Z})= \min_{k\in \D{Z}} |x-k|$. Let us fix an irrational number $\ga \in \D{R}$. 
Define the numbers $\ga_n\in (0,1/2)$, for $n\geq 0$, according to 
\begin{equation} \label{E:rotations-rest}
\ga_0=d(\ga, \D{Z}), \quad  \ga_{n+1}=d(1/\ga_n, \D{Z}).
\end{equation} 
Then, there are unique integers $a_n$, for $n\geq -1$, and $\gep_n \in \{+1, -1\}$, for $n\geq 0$, such that 
\begin{equation}\label{E:rotations-relations}
\ga= a_{-1}+ \gep_0 \ga_0, \quad  1/\ga_n= a_n + \gep_{n+1} \ga_{n+1}. 
\end{equation}
Evidently, for all $n\geq 0$,
\begin{equation}\label{E:rotations-signs-epsilon-n}
1/\ga_n \in (a_n-1/2, a_n+1/2), \quad a_n\geq 2,
\end{equation}
and 
\begin{equation}\label{E:signs-introduced}
\gep_{n+1}= 
\begin{cases}
+1 & \text{if } 1/\ga_n \in (a_n, a_n+1/2), \\
-1 & \text{if } 1/\ga_n \in (a_n-1/2, a_n).
\end{cases}
\end{equation}
In order to streamline our notations in this paper, we let $\ga_{-1}=+1$. 
See \refS{S:arithmetic}, in particular \refS{SS:modified-fractions}, for more details about 
this notion of continued fraction algorithm (known as nearest integer continued fraction).

\subsection{Successive changes of coordinates}
Recall that $s(w)=\ol{w}$ denotes the complex conjugation map. 
For $n\geq0$ we define 
\begin{equation}\label{E:Y_n}
\mathbb{Y}_n(w) = 
\begin{cases}
Y_{\ga_n}(w) & \tif \gep_n=-1\\
- s\circ Y_{\ga_n}(w)  & \tif \gep_n=+1. 
\end{cases}
\end{equation}
Each $\mathbb{Y}_n$ is either orientation preserving or reversing, depending on the sing of $\gep_n$. 
For $n\geq 0$, we have\footnote{For a given set $X \subseteq \D{C}$, we define $iX= \{ix \mid x\in X \}$.}
\begin{equation} \label{E:invariant-imaginary-line}
\mathbb{Y}_n( i  [-1, +\infty)) \subset   i  (-1, +\infty), \quad \text{ and } \quad \mathbb{Y}_n(0)=0.
\end{equation}
\refL{L:uniform-contraction-Y_r} implies that for all $n\geq 0$ and all $w_1, w_2$ in $\ol{\D{H}'}$, we have 
\begin{equation}\label{E:uniform-contraction-Y}
|\mathbb{Y}_n(w_1)- \mathbb{Y}_n(w_2)| \leq 0.9 |w_1-w_2|.
\end{equation}
It follows from \refL{L:Y_r-commutation} that for all $n\geq 0$ and all $w\in \ol{\D{H}'}$, 
\begin{equation} \label{E:Y_n-comm-1}
\mathbb{Y}_n(w+1/\ga_n) = 
\begin{cases}
\mathbb{Y}_n(w)+1 & \tif \gep_n=-1,\\
\mathbb{Y}_n(w)-1 & \tif \gep_n=+1. 
\end{cases}
\end{equation}
Also, by the same lemma, for all $n\geq 0$, and all $t\geq -1$, 
\begin{equation}\label{E:Y_n-comm-2}
\mathbb{Y}_n( i  t+ 1/\ga_n-1)= 
\begin{cases}
\mathbb{Y}_n( i  t)+ (1-\ga_n) & \tif \gep_n=-1, \\
\mathbb{Y}_n( i  t) +(\ga_n-1) & \tif \gep_n=+1.
\end{cases}
\end{equation}

\subsection{Equivariant tiling of the tower}\label{SS:tilings-nest}
For $n\geq 0$ let 
\begin{equation}\label{E:I_n-J_n-K_n}
\begin{gathered}
I_n^0 = \{w\in \ol{\D{H}'} \mid \Re w\in [0, 1/\ga_n]\}, \\
J_n^0 =  \{w\in I_n^0 \mid \Re w \in [1/\ga_n-1, 1/\ga_n]\}, \\
K_n^0 =  \{w\in I_n^0 \mid \Re w \in [0, 1/\ga_n-1] \}.
\end{gathered}
\end{equation}
We inductively defined the sets $I_n^j$, $J_n^j$, and $K_n^j$, for $j \geq 1$ and $n\geq 0$. 
Assume that for some $j$ and all $n \geq 0$, $I_n^j$, $J_n^j$ and $K_n^j$ are defined.
We define $I_n^{j+1}$, $J_n^{j+1}$ and $K_n^{j+1}$ for $n\geq 0$ as follows.
Fix an arbitrary $n \geq 0$. 
If $\gep_{n+1}=-1$, let 
\begin{equation}\label{E:I_n^j--1}
I_n^{j+1} 
= \bigcup_{l=0}^{a_n-2} \big( \mathbb{Y}_{n+1} ( I_{n+1}^j)+ l \big)  \bigcup \big( \mathbb{Y}_{n+1}(K_{n+1}^j)+ a_n-1\big).
\end{equation}
If $\gep_{n+1}=+1$, let 
\begin{equation}\label{E:I_n^j-+1}
I_n^{j+1} 
= \bigcup_{l=1}^{a_n} \big( \mathbb{Y}_{n+1} ( I_{n+1}^j)+ l \big)  \bigcup \big(\mathbb{Y}_{n+1}(J_{n+1}^j)+ a_n+1\big ).
\end{equation}
Regardless of the sign of $\gep_{n+1}$, define 
\[J_n^{j+1} =  \{w\in I_n^{j+1} \mid \Re w \in [1/\ga_n-1, 1/\ga_n]\},\]
\[K_n^{j+1} =  \{w\in I_n^{j+1} \mid \Re w \in [0, 1/\ga_n-1] \}.\]
\refF{F:topological-model} presents two generations of these domains.  
We summarise the basic features of these sets in the following lemma. 

\begin{figure}[ht]
\begin{pspicture}(14,11) 
\epsfxsize=6cm
\rput(3,8){\epsfbox{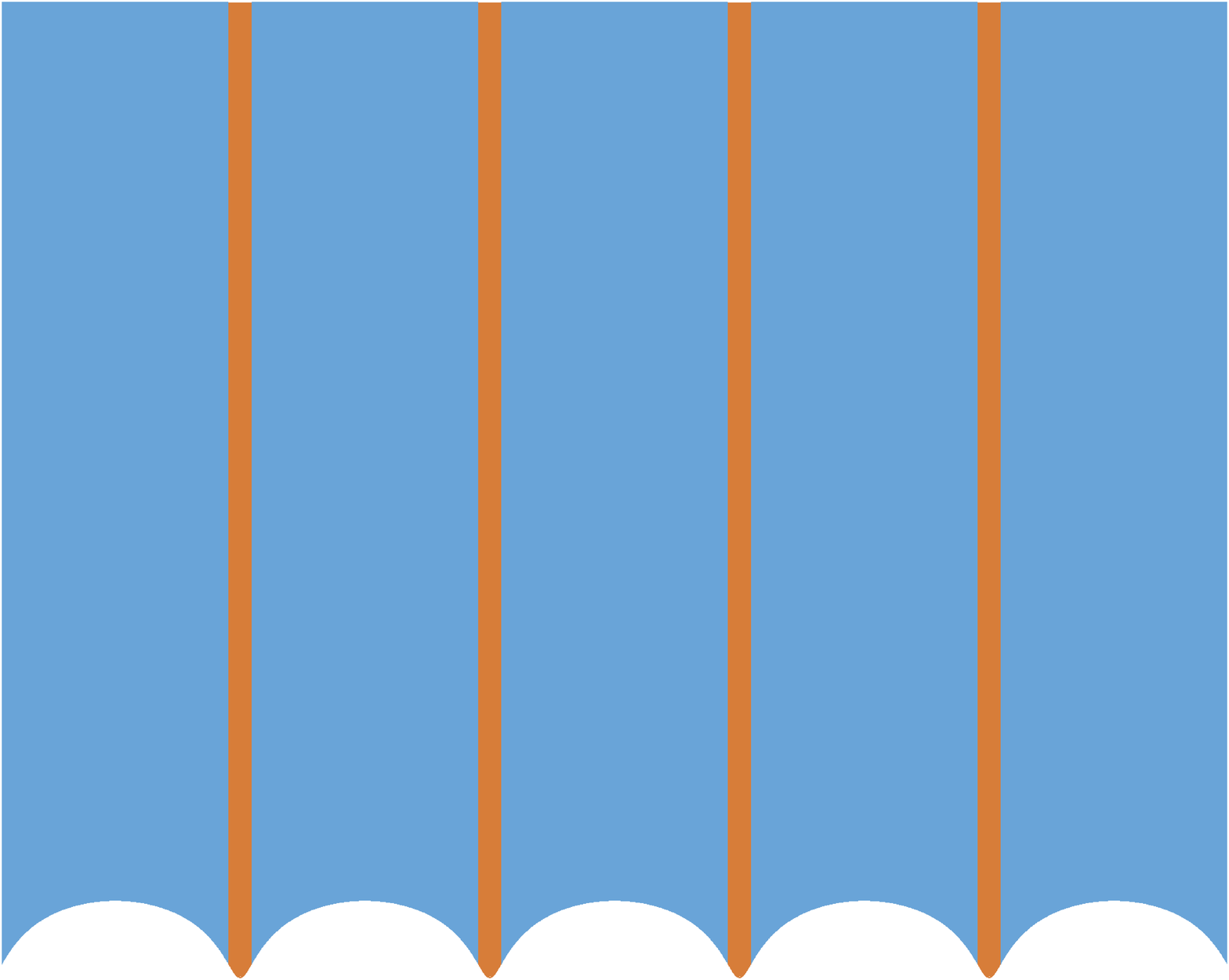}}
\rput(11,8){\epsfbox{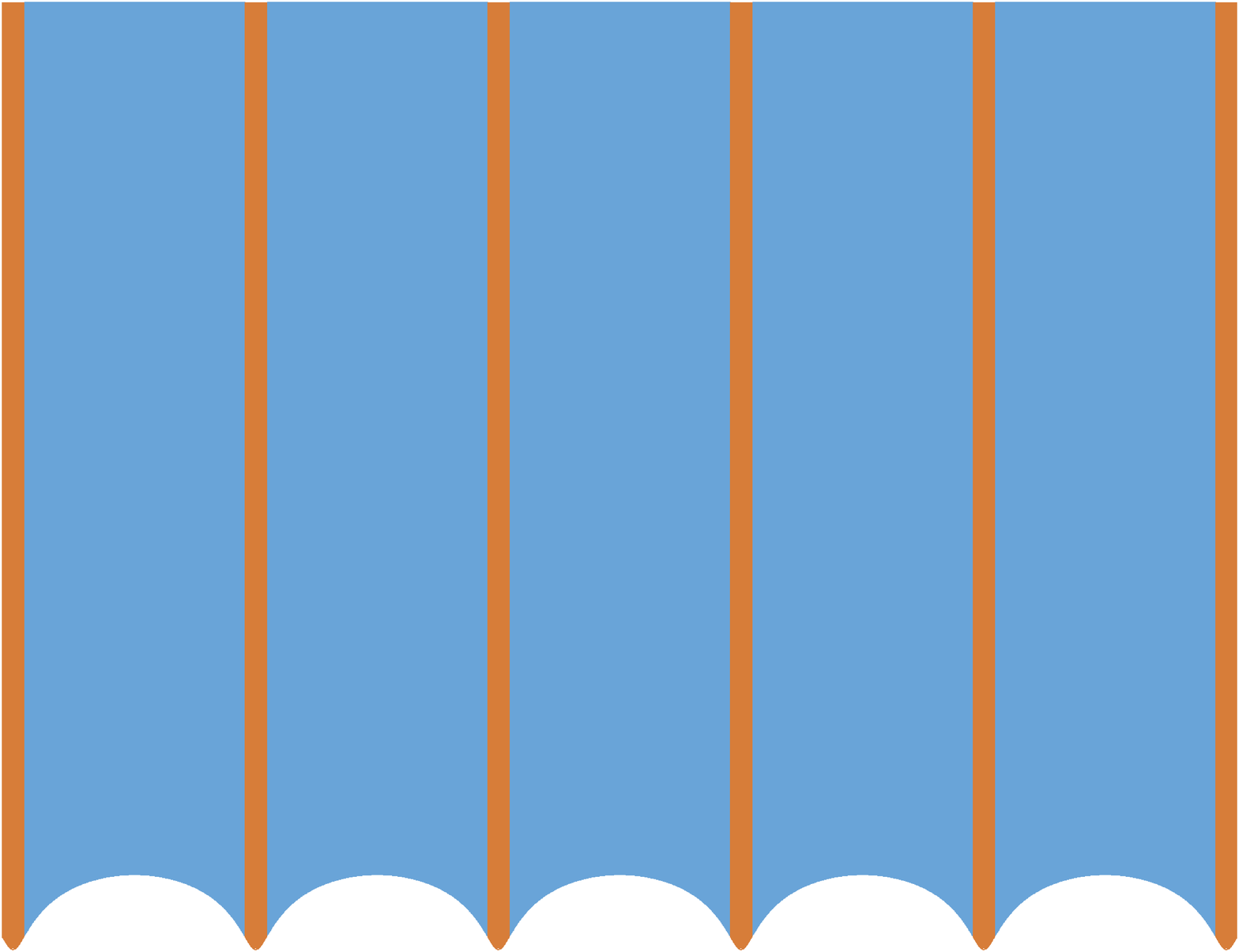}}

\pspolygon[linecolor=newcyan,fillstyle=solid,fillcolor=newcyan](0,0)(5.4545,0)(5.4545,3)(0,3)
\pspolygon[linecolor=neworange,fillstyle=solid,fillcolor=neworange](5.4545,0)(6,0)(6,3)(5.4545,3)
\rput(3,1.5){\small $K_n^0$}
\rput(5.7,1.5){\small $J_n^0$}

\pspolygon[origin={8,0},linecolor=newcyan,fillstyle=solid,fillcolor=newcyan](0,0)(5.4545,0)(5.4545,3)(0,3)
\pspolygon[origin={8,0},linecolor=neworange,fillstyle=solid,fillcolor=neworange](5.4545,0)(6,0)(6,3)(5.4545,3)
\rput(11,1.5){\small $K_n^0$}
\rput(13.7,1.5){\small $J_n^0$}

\psline[linewidth=.5pt,linecolor=darkgray]{->}(2.8,2)(.5,6.5)
\psline[linewidth=.5pt,linecolor=darkgray]{->}(5.7,2)(1.15,6.5)

\psline[linewidth=.5pt,linecolor=darkgray]{->}(3,2)(5.5,6.5)

\rput(1,5){\small $\mathbb{Y}_n$}
\rput(2.3,5){\small $\mathbb{Y}_n$}
\rput(4.7,5){\small $\mathbb{Y}_n+a_{n-1}-1$}

\psline[origin={8,0},linewidth=.4pt]{->}(2.5,2)(.5,6.8)
\psline[origin={8,0},linewidth=.4pt]{->}(5.7,2)(.05,6.8)
\psline[origin={8,0},linewidth=.4pt]{->}(5.8,2)(5.95,6.8)

\rput(9,5){\small $\mathbb{Y}_n$}
\rput(9.8,5){\small $\mathbb{Y}_n$}
\rput(12.8,5){\small $\mathbb{Y}_n+a_{n-1}+1$}
\end{pspicture}
\caption{The left hand picture is for $\gep_n=-1$ and the right hand picture is for $\gep_n=+1$. 
The sets $K_n^0$ and $J_n^0$ are on the lower row,  and the set $I_{n-1}^1$ is on the upper row.}
\label{F:topological-model}
\end{figure}

\begin{lem}\label{L:I_n^j-basic-features}
For all $n \geq 0$ and $j\geq 0$, the sets $I_n^j$, $J_n^j$ and $K_n^j$ are closed and connected subsets of 
$\D{C}$, and each of $\partial I_n^j$, $\partial J_n^j$ and $\partial K_n^j$ is a piece-wise analytic curve. 
Moreover, 
\begin{itemize}
\item[(i)] $\{\Re w \mid w\in I_n^j\}= [0, 1/\ga_n]$;
\item[(ii)] if $\gep_{n+1}=-1$, we have 
\begin{equation}\label{E:lemma-I_n^j--1}
\begin{gathered} 
\{w\in I_n^j  \mid \Re w=0\} \subseteq \mathbb{Y}_{n+1}( i  [-1, \infty)), \\
\{w\in I_n^j  \mid \Re w=1/\ga_n\} \subseteq \mathbb{Y}_{n+1}( i  [-1, \infty)+ 1/\ga_{n+1}-1) +a_n-1;
\end{gathered}
\end{equation}
\item[(iii)] if $\gep_{n+1}=+1$, we have 
\begin{equation}\label{E:lemma-I_n^j-+1}
\begin{gathered}
\{w\in I_n^j  \mid \Re w=0\} \subseteq \mathbb{Y}_{n+1}( i  [-1, \infty)+1/\ga_n)+1, \\
\{w\in I_n^j  \mid \Re w=1/\ga_n\} \subseteq \mathbb{Y}_{n+1}( i  [-1, \infty)+ 1/\ga_{n+1}-1) +a_n+1.
\end{gathered}
\end{equation}
\end{itemize}
\end{lem}

\begin{proof}
The proof is elementary and is left to the reader. One only needs to follow the basic arithmetic relations among 
$\ga_n$, $a_n$ and $\gep_{n+1}$ in \eqref{E:rotations-rest}-\eqref{E:rotations-signs-epsilon-n}, 
and use the functional relations in \eqref{E:Y_n-comm-1} and \eqref{E:Y_n-comm-2}. 
\end{proof}

\begin{lem}\label{L:model-almost-periodic}
For every $n\geq 0$ and $j\geq 0$, the following hold:
\begin{itemize}
\item[(i)] for all $w \in \overline{\mathbb{H}'}$, we have $w \in I_n^j$ if and only if $w+1 \in I_n^j$; 
\item[(ii)] for all $t \in \mathbb{R}$, we have $it \in I_n^j$ if and only if $it +1/\ga_n \in I_n^j$.    
\end{itemize}
\end{lem}

\begin{proof}
We shall prove both parts at once by an inductive argument on $j$. 
Clearly, both statements hold for $j=0$ and all $n\geq 0$. 
Assume that both parts of the lemma hold for some $j \geq 0$ and all $n\geq 0$. 

Part (ii) for $j$ and \refE{E:Y_n-comm-1} together imply that for real values of $t$, $it \in I_n^{j+1}$ if and only 
if $it+1 \in I_n^{j+1}$. Then, by the definition of $I_n^{j+1}$ in \eqref{E:I_n^j--1} and \eqref{E:I_n^j-+1}, 
one concludes part (i) for $j+1$ and all $n\geq 0$.  

To prove part (ii) of the lemma for $j+1$ and all $n\geq 0$ we need to consider two cases. 
First assume that $\gep_{n+1}=-1$. By \eqref{E:Y_n-comm-2} and \eqref{E:rotations-relations}, 
for $t$ and $t'$ in $[-1, +\infty)$ satisfying $\mathbb{Y}_{n+1}(it')=it$, we have 
\[\mathbb{Y}_{n+1}(it'+ 1/\ga_{n+1} - 1)+a_n-1=
\mathbb{Y}_{n+1}(it')+ (1-\ga_{n+ 1}) + a_n - 1=
it + 1/\ga_n .\] 
If $i t \in I_n^{j+1}$, then by \refE{E:I_n^j--1}, there is $it' \in I_{n+1}^j$ with $\mathbb{Y}_{n+1}(it')=it$. 
By the induction hypotheses (part (i) and (ii) for $j$ and all $n\geq 0$), $it'+ 1/\ga_{n+1} - 1 \in I_{n+1}^j$. 
Then, $it'+ 1/\ga_{n+1} - 1 \in K_{n+1}^j$. Therefore, by the above equation and \eqref{E:I_n^j--1}, 
$it+1/\ga_n \in I_n^{j+1}$. 

On the other hand, if $it+1/\ga_n \in I_n^{j+1}$, then by \eqref{E:I_n^j--1} and the induction hypotheses, 
there is $it'\in I_{n+1}^j$ such that $it'+1/\ga_{n+1} -1 \in K_{n+1}^j$ and 
\[\mathbb{Y}_{n+1}(it'+1/\ga_{n+1} -1) + (a_n-1) = it+1/\ga_n.\] 
Again, by \eqref{E:Y_n-comm-2} and \eqref{E:rotations-relations}, this implies that 
$it+1/\ga_n=\mathbb{Y}_{n+1}(it')+1/\ga_n$. 
Hence, $it= \mathbb{Y}_{n+1}(it')$, which implies that $it\in I_n^{j+1}$. 

The proof when $\gep_{n+1}=+1$ is similar, using \eqref{E:I_n^j-+1} and \eqref{E:Y_n-comm-2}.   
\end{proof}

Recall that $\ga_{-1}=+1$. Let $I_{-1}^0=\{w\in \ol{\D{H}'} \mid \Re w \in [0, 1/\ga_{-1}]\}$, and for $j\geq 1$, 
consider the sets 
\begin{equation}\label{E:I--1}
I_{-1}^j= \mathbb{Y}_0(I_0^{j-1}) + (\gep_0+1)/2.
\end{equation}
By \refL{L:Y-domain}, $I_n^1 \subset I_n^0$, for $n\geq -1$. By an inductive argument, this implies that for all 
$n\geq -1$ and all $j\geq 0$, 
\begin{equation}\label{E:I_n^j-forms-nest}
I_n^{j+1} \subset I_n^j.
\end{equation}

\subsection{The sets \texorpdfstring{$\mathbb{M}_\ga$}{M-ga}}\label{SS:M-ga}
For $n\geq -1$, we define  
\[I_{n}= \cap_{j\geq 1} I_{n}^j.\] 
Each $I_n$ consists of closed half-infinite vertical lines with unbounded imaginary parts. 
However, $I_n$ may or may not be connected. 
We note that $\Re I_{-1} \subset  [0,1]$. 
Indeed, by \refL{L:model-almost-periodic}, for real $t$, $it \in I_{-1}$ if and only if $(it +1) \in I_{-1}$. 
Thus, we may define
\begin{equation}\label{E:M_ga}
\mathbb{M}_\ga=   \{s(e^{2\pi i w}) \mid w \in I_{-1}\} \cup \{0\}.
\end{equation}

\begin{rem}
There may be alternative (simpler) approaches to build the set $\mathbb{M}_\ga$ using the maps $\mathbb{Y}_n$. 
For the sake of applications, here we have selected an approach which closely mimics a construction 
in the renormalisation scheme. 
However, it is worth noting that one cannot simply define the domains $I_n^j$, and subsequently 
$\mathbb{M}_\ga$, by first extending the maps $\mathbb{Y}_n$ $1/\ga_n$-periodically onto $\D{H}'$, 
and then iterating them on $\D{H}'$. 
That is because, such a construction would lead to a set in the limit which is periodic under translations 
by $+1$ and some irrational number. 
It would follow that the interior of that set must be the region above a horizontal line, which cannot be the case for 
arbitrary $\ga$. 
\end{rem}

\begin{propo}\label{P:M-ga-relations-1}
For every $\ga \in \mathbb{R} \setminus \mathbb{Q}$ we have the following:\footnote{A set $M \subseteq \mathbb{C}$ is called \textit{star-like about} $0$ if for every $z \in M$ and every $r\in [0,1]$, $rz \in M$.}
\begin{itemize}
\item[(i)] $\mathbb{M}_\ga$ is a compact set which is star-like about $0$, $+1\in \mathbb{M}_\ga$, 
and $\mathbb{M}_\ga \cap (1, \infty)=\emptyset$;  
\item[(ii)] $\mathbb{M}_{\ga+1}= \mathbb{M}_\ga$ and $s(\mathbb{M}_\ga)= \mathbb{M}_{-\ga}$. 
\end{itemize}
\end{propo}

\begin{proof}
Since every map $\mathbb{Y}_n$ sends vertical lines to vertical lines, each of $I_n^j$ is the region above the 
graph of a function. 
This implies that for every $n\geq -1$, $I_n$ consists of some half-infinite vertical lines. 
Thus, $\mathbb{M}_\ga$ is star-like about $0$. 
Also, $\mathbb{M}_\ga$ is bounded and closed, hence a compact set. 
On the other hand, for all $n\geq -1$, $0\in I_n$, which implies that $+1 \in \mathbb{M}_\ga$. 
Indeed, by the uniform contraction of the maps $\mathbb{Y}_n$, \refE{E:uniform-contraction-Y}, $+1$ is an 
end point of a ray in $\mathbb{M}_\ga$, that is, for every $\delta>0$, $1+\delta \notin \mathbb{M}_\ga$. 

Recall that in order to define $\mathbb{M}_\ga$, which only depends on $\alpha$, 
we first define the sequences $(\ga_n)_{n\geq 0}$, $(a_n)_{n\geq 0}$, and $(\eps_n)_{n\geq 0}$. 
These fully determine the sequence of maps $(\mathbb{Y}_n)_{n\geq 0}$, and hence the set $\mathbb{M}_\ga$. 
The irrational numbers $\ga$ and $\ga+1$ produce the same sequences $(\ga_n)_{n\geq 0}$, $(a_n)_{n\geq 0}$ 
and $(\gep_n)_{n\geq 0}$. 
Therefore, the sequence of the maps $\mathbb{Y}_n$ are the same for both $\ga$ and $\ga+1$. 
This implies that $\mathbb{M}_{\ga+1}= \mathbb{M}_\ga$. 

To determine $\mathbb{M}_{-\ga}$, we need to compare the corresponding sequences $(\ga_n)_{n\geq 0}$, 
$(a_n)_{n\geq 0}$ and $(\gep_n)_{n\geq 0}$ for $\ga$ and $-\ga$. 
Let us denote the corresponding objects for $-\ga$ using the same notations as the ones for $\ga$ but with a prime, 
that is, $\ga_n'$, $a_n'$, $\eps'_n$, $\mathbb{Y}'_n$, $I'_n$, $K'_n$, $J'_n$, etc. 
Using $\ga= a_{-1}+ \gep_0 \ga_0$, we note that $a'_{-1}=-a_{-1}$, $\ga_0'=\ga_0$, $\gep'_{0}=-\gep_0$. 
These imply that $\mathbb{Y}_0'=-s \circ \mathbb{Y}_0$. 
However, since $\ga_0'=\ga_0$, we conclude that for all $n \geq 1$ we have $a'_{n-1}=a_{n-1}$, 
$\ga'_n=\ga_n$, $\gep'_n=\gep_n$. 
Thus, for all $n\geq 1$, $\mathbb{Y}_n'=\mathbb{Y}_n$. 
These imply that $I_0= I_0'$. 
In particular, in the definition of $\mathbb{M}_{-\ga}$, the only difference with the definition of 
$\mathbb{M}_\ga$ is that $\mathbb{Y}_0$ changes to $-s \circ \mathbb{Y}_0$. 
Therefore, using \refE{E:I--1}, 
\begin{align*}
I'_{-1} = \mathbb{Y}_0'(I'_0) + \frac{1+\epsilon_0'}{2}
& = -s \circ \mathbb{Y}_0(I_0') + \frac{1-\epsilon_0}{2} \\
&= -s \left( \mathbb{Y}_0(I_0) + \frac{1+\epsilon_0}{2} \right) + 1
= -s (I_{-1}) +1. 
\end{align*}
Comparing to \refE{E:M_ga}, we have 
\begin{align*}
\mathbb{M}_{-\ga} 
& = \{ s(e^{2\pi i w}) \mid w \in I'_{-1}\} \cup \{0\} \\
& = \{ s(e^{2\pi i w}) \mid w \in -s (I_{-1}) \} \cup \{0\} \\
&= \{ s \circ s(e^{2\pi i w}) \mid w \in I_{-1} \} \cup \{0\} \\
&= s \left (\{s(e^{2\pi i w}) \mid w \in I_{-1} \}\right) \cup \{0\}
= s (\mathbb{M}_\ga). \qedhere
\end{align*}
\end{proof}

\begin{lem}\label{L:M-ga-relations-2}
For every $\ga \in \mathbb{R} \setminus \mathbb{Q}$ we have the following:  
\begin{itemize}
\item[(i)] if $\ga \in (0, 1/2)$, $\mathbb{M}_{1/\ga}= \{s(e^{2\pi i w})  \mid w \in I_0 \} \cup \{0\}$, 
\item[(ii)] if $\ga \in (-1/2, 0)$, $\mathbb{M}_{-1/\ga}= \{s(e^{2\pi i w})  \mid w \in I_0 \} \cup \{0\}$.  
\end{itemize}
\end{lem} 

\begin{proof}
For $\ga$, let $(\ga_n)_{n\geq 0}$, $(a_n)_{n \geq 0}$ and $(\gep_n)_{n\geq 0}$ denote the sequences defined in 
\refS{SS:modified-fractions-mini}. 
Define $\ga'= \gep_1 \ga_1$. 
Let us use the notations $(\ga'_n)_{n\geq 0}$, $(a'_n)_{n \geq 0}$ and $(\gep'_n)_{n\geq 0}$ for the sequences 
$(\ga_n)_{n \geq 0}$, $(a_n)_{n \geq 0}$ and $(\gep_n)_{n\geq 0}$ corresponding to $\ga'$. 
We have $\ga'_0=\ga_1$, and then 
\[\gep_1 \ga_1= \ga'= a'_{-1}+\gep'_0 \ga'_0= a'_{-1}+\gep'_0 \ga_1.\] 
As $\ga' \in (-1/2, 1/2)$, we must have $a'_{-1}=0$, and hence $\gep'_0=\gep_1$. 
To determine $\mathbb{Y}_0'$, we consider two cases. 
If $\gep'_0=\gep_1=-1$, we have $\mathbb{Y}'_0= Y_{\ga'_0}=\mathbb{Y}_{\ga_1}= \mathbb{Y}_1$, 
and if $\gep'_0=\gep_1=+1$, we have $\mathbb{Y}'_0= -s \circ Y_{\ga'_0}= -s \circ Y_{\ga_1}= \mathbb{Y}_1$. 

The relation $\ga'_0=\ga_1$ implies that for all $n\geq 1$ we have $\ga'_n= \ga_{n+1}$, $a'_n=a_{n+1}$ and 
$\gep'_n = \gep_{n+1}$. 
Hence, for all $n\geq 1$, $\mathbb{Y}'_n= \mathbb{Y}_{n+1}$. 
These imply that $I'_{-1}=I_0$. 
Therefore, according to \refE{E:M_ga}, we must have 
\begin{equation}\label{E:I_0-projection}
\mathbb{M}_{\gep_1 \ga_1}=\mathbb{M}_{\ga'}= \{ s(e^{2\pi i w}) \mid w \in I'_{-1}\} \cup \{0\}
= \{ s(e^{2\pi i w}) \mid w \in I_0 \} \cup \{0\}.
\end{equation}
If $\ga\in (0, 1/2)$, we have $1/\ga=1/\ga_0 = a_{-1}+ \gep_1 \ga_1$, which using \refL{P:M-ga-relations-1}-(ii), 
we obtain 
\[\mathbb{M}_{1/\ga}=\mathbb{M}_{a_{-1}+\gep_1 \ga_1}= \mathbb{M}_{\gep_1 \ga_1}.\]
If $\ga \in (-1/2, 0)$, we have $-1/\ga= 1/\ga_0= a_{-1}+ \gep_1 \ga_1$, and hence, using \refL{P:M-ga-relations-1}-(ii), 
we get  
\[\mathbb{M}_{-1/\ga}= \mathbb{M}_{a_{-1}+\gep_1 \ga_1} = \mathbb{M}_{\gep_1 \ga_1}.\]
Combining the above equations, we obtain the desired properties in parts (i) and (ii).  
\end{proof} 
\section{The map \texorpdfstring{$\mathbb{T}_\ga$}{T} on \texorpdfstring{$\mathbb{M}_\ga$}{M-ga}}
\label{S:T-on-M}
In this section we define the map  
\begin{equation}
\mathbb{T}_\ga: \mathbb{M}_\ga \to \mathbb{M}_\ga, 
\end{equation}
where $\mathbb{M}_\ga$ is the topological model defined in \refS{S:M-ga}. 

The topological description of $\mathbb{M}_\ga$, which is presented in \refS{S:topology-A}, does not employ 
the map $\mathbb{T}_\ga$ in any ways.
However, the topological description can be used to identify the map $\mathbb{T}_\ga$ as follows. 
When $\mathbb{A}_\alpha$ is a Jordan curve, it is the graph of a function of the argument. 
Thus, there is a unique homeomorphism of $\mathbb{A}_\alpha$ which acts as rotation 
by $2\pi \alpha$ in the tangential direction.
Similarly, when, $\mathbb{A}_\alpha$ is a hairy Jordan curve, there is a unique homeomorphism 
of the base Jordan curve which acts as rotation by $2\pi \alpha$. 
This map can be extended onto the end points of the Jordan arcs attached to the Jordan curve by matching the 
corresponding end points. 
Since the set of end points of those arcs is dense in $\mathbb{A}_\alpha$, there may be a 
unique homeomorphism of $\mathbb{A}_\alpha$ which acts as rotation by $2\pi \alpha$ on the 
base Jordan curve. 
However, it is not clear if this map continuously extents to the whole hairy Cantor set. 
Similarly, there may be a unique homeomorphism of a Cantor bouquet $\mathbb{A}_\alpha$ 
which acts as rotation by $2\pi \ga$ in the tangential direction.
Here we take a different approach to build $\mathbb{T}_\ga$ on $\mathbb{M}_\ga$. 
We give a presentation which is aligned with the action of the map on the renormalisation tower; compare with 
\cite{Che13,AC18}.
This helps us later when describing the dynamics of $\mathbb{T}_\ga$ on $\mathbb{A}_\alpha$, 
and may also be employed to link the toy renormalisation scheme we build here to an actual renormalisation scheme. 

\subsection{Definition of the lift of \texorpdfstring{$\mathbb{T}_\ga$}{T-ga}}\label{SS:T-defn}
Let us fix $\ga \in \D{R}\setminus \D{Q}$, and let $I_n$, for $n\geq -1$, be the sets in \refS{SS:tilings-nest}.  
Given $w_{-1} \in I_{-1}$, we inductively identify the integers $l_i$ and then the points $w_{i+1} \in I_{i+1}$ so that 
\[0 \leq  \Re (w_i -l_i) < 1, \  \tif \gep_{i+1}=-1; \q -1 < \Re (w_i -l_i) \leq 0 , \q \tif \gep_{i+1}=+1;\] 
and 
\[\mathbb{Y}_{i+1}(w_{i+1})+l_{i}= w_i.\]
It follows that for all $n\geq 0$, we have 
\begin{equation}\label{E:trajectory-condition-1}
w_{-1}=(\mathbb{Y}_0+l_{-1}) \circ (\mathbb{Y}_1 + l_0)  \circ \dots \circ (\mathbb{Y}_n+l_{n-1})(w_n).
\end{equation}
Also, by the definition of $I_i$ in \eqref{E:I_n^j--1} and \eqref{E:I_n^j-+1}, for all $i \geq 0$, 
\begin{equation}\label{E:trajectory-real-parts}
0 \leq l_i \leq a_i + \gep_{i+1}, \q \tand \q 0 \leq \Re w_i < 1/\ga_i.
\end{equation}
We refer to the sequence $(w_i ; l_i)_{i \geq -1}$ as the \textbf{trajectory} of $w_{-1}$, with respect to $\ga$, or simply, as 
the trajectory of $w_{-1}$, when it is clear from the context what irrational number is used. 

We define the map 
\begin{equation}
\tilde{T}_\ga:I_{-1} \to I_{-1}, 
\end{equation}
as follows. 
Let $w_{-1}$ be an arbitrary point in $I_{-1}$, and let $(w_i; l_i)_{i\geq -1}$ denote the trajectory of $w_{-1}$. 
Then, 
\begin{itemize}
\item[(i)] if there is $n \geq 0$ such that $w_n \in K_n$, and for all $0 \leq i \leq n-1$, $w_i \in I_i \setminus K_i$, then 
\[\tilde{T}_\ga(w_{-1})= \left(\mathbb{Y}_0+\frac{\gep_0+1}{2}\right ) \circ \left(\mathbb{Y}_1+\frac{\gep_1+1}{2}\right ) 
\circ \cdots \circ \left (\mathbb{Y}_n +\frac{\gep_{n}+1}{2} \right)(w_n+1);\] 
\item[(ii)] if for all $n \geq 0$, $w_n \in I_n \setminus K_n$, then 
\[\tilde{T}_\ga(w_{-1})= \lim_{n \to +\infty} \left(\mathbb{Y}_0+\frac{\gep_0+1}{2}\right ) \circ \left(\mathbb{Y}_1+\frac{\gep_1+1}{2}\right ) 
\circ \cdots \circ \left (\mathbb{Y}_n +\frac{\gep_{n}+1}{2} \right) (w_n+1-1/\ga_n).\]
\end{itemize}
It might not be clear that the limit in case (ii) exists. We look into this within the proof of \refL{L:model-map-lift}. 

\subsection{The continuity}\label{SS:continuity-T-ga}

\begin{lem}\label{L:model-map-lift}
For every $\ga \in \D{R} \setminus \D{Q}$, the map $\tilde{T}_\ga : I_{-1} \to I_{-1}$ induces a well-defined, 
continuous and injective map \footnote{Here, $\D{Z}$ acts on $I_{-1}$ by horizontal translations by integers.} 
\[\tilde{T}_{\ga}: I_{-1}/\D{Z} \to I_{-1}/ \D{Z}.\]
\end{lem}

The main idea of the proof for the above statement is to partition the set $I_{-1}$ into infinitely many pieces, 
where the map is continuous on each piece. Then, we show that the maps on the pieces match at the boundary points.  
Below we introduce the partition pieces. 

For $w_{-1} \in I_{-1}$, let $(w_i; l_i)_{i \geq -1}$ denote the trajectory of $w_{-1}$. 
For each $n\geq 0$, let 
\[W^n= \{ w_{-1} \in I_{-1} \mid \text{ for all } 0 \leq i \leq n-1, w_i \in I_i \setminus K_i, \tand w_n \in K_n\},\]  
\[V^n =\{w_{-1} \in I_{-1} \mid \text{ for all } 0 \leq i \leq n, w_i  \in I_i \setminus K_i\}.\] 
We set $V^\infty= \cap_{n\geq 0} V^n$. 
Evidently, we have 
\begin{equation}\label{E:I_-1-decomposed}
I_{-1}=\cup_{n \geq 0} W^n \cup  V^\infty.
\end{equation}

It is also convenient to use some simplified notations for the compositions of the maps which appear in the definition 
of $\mathbb{T}_\ga$. 
That is, for $m \geq n$, let \footnote{\textrm{id} denotes the identity map.} 
\begin{equation*} 
X_n^{n-1}=\textrm{id},  \qquad 
X_n^m= \left(\mathbb{Y}_n+\frac{\gep_n+1}{2}\right ) \circ \left(\mathbb{Y}_{n+1}+\frac{\gep_{n+1}+1}{2}\right ) 
\circ \cdots \circ \left (\mathbb{Y}_{m} +\frac{\gep_{m}+1}{2} \right).
\end{equation*}

We break the proof of \refL{L:model-map-lift} into several lemmas. 

\begin{lem}\label{L:P:model-map-lift-1}
For all $n\geq 1$, the map $w_{-1} \mapsto w_n$ is continuous and injective on the sets $W^n$ and $V^n$. 
In particular, for all $n \geq 1$, $\tilde{T}_\ga: W^n \to I_{-1}$ is continuous and injective. 
\end{lem}

\begin{proof}
For each $i\geq 0$, if $w_i \in I_i \setminus K_i$ and $w_{i+1} \in K_{i+1}$, then $l_i=a_i+ (\gep_{i+1}-1)/2$. 
Similarly, if $w_i \in I_i \setminus K_i$ and $w_{i+1} \in I_{i+1} \setminus K_{i+1}$, then $l_i=a_i+ (3\gep_{i+1}-1)/2$. 
These imply that for all $n\geq 1$ and all $w_{-1} \in W^n$, the entries $(l_i)_{i=-1}^{n-1}$ in the trajectory of $w_{-1}$ 
is independent of $w_{-1}$. 
Similarly, for all $n\geq 1$ the entries $(l_i)_{i=-1}^{n-1}$ in the trajectory of $w_{-1} \in V^n$ is independent of $w_{-1}$. 
In particular, the map $w_{-1} \mapsto w_n$ is continuous and injective on each of $W^n$ and $V^n$. 

As each $\mathbb{Y}_j$ is continuous and injective, we conclude that $\tilde{T}_\ga$ is continuous and injective on $W^n$.
\end{proof}

\begin{lem} \label{L:P:model-map-lift-2}
The map $\tilde{T}_\ga$ is well-defined, continuous and injective on $V^\infty$.
\end{lem}

\begin{proof}
We know from \refL{L:P:model-map-lift-1} that for every $n \geq 0$, the map $w_{-1} \mapsto w_n$ is continuous and 
injective on $V^n$. 
The image of this map covers $(1/\ga_n-1, 1/\ga_n) \cap I_n$, due to the choice we made in \refE{E:trajectory-real-parts}.
Since the inverse map $w_n \mapsto w_{-1}$ is also continuous and injective on $(1/\ga_n-1, 1/\ga_n) \cap I_n$, it follows that 
$V^n$ is relatively open in $I_{-1}$. It is possible that for some values of $\ga$, the nest $\cap_{n\geq 0} V^n$ is empty 
(for instance when $\gep_i=-1$, for all $i\geq 0$).
Below we assume that $V^\infty$ is not empty. 

For each $n \geq 0$ and $0 \leq i \leq n$, we define the sets $V^n_i$ as the set of $w_i$, for $w_{-1} \in V^n$. 
Then, define $V^\infty_i = \cap_{n \geq i} V^n_i$. It follows that $w_{-1} \mapsto w_i$ is continuous and injective 
from $V^\infty$ to $V^\infty_i$, for all $i\geq 0$. 
Moreover, by the uniform contraction of the map $\mathbb{Y}_j$, each $V^\infty_i$ is a closed half-infinite vertical line.

For $n \geq 0$ and $m \geq n$, we define the map $E_n^m: V_n^m \to I_n$ as follows 
\begin{equation}\label{E:L:model-map-lift-2-1}
E_n^m(w_n) = X_{n+1}^m  (w_m+1-1/\ga_m).
\end{equation}
By the above paragraphs, this is a continuous and injective map on $V_n^m$. 

Note that for $w_n \in I_n \setminus K_n$ we have 
\[|\mathbb{Y}_{n+1}(w_{n+1}) + (\gep_{n+1}+1)/2 - (w_n+1-1/\ga_n)| \leq 1.\]
This is because, $\Im \mathbb{Y}_{n+1}(w_{n+1})= \Im w_n$, $\Re w_n +1 -1/\ga_n \in [0,1]$, and 
$\Re \mathbb{Y}_{n+1} (w_{n+1})+ (\gep_{n+1}+1)/2$ belongs to $[0,+1]$. 

By \refE{E:Y_n-comm-1}, $\mathbb{Y}_{n+1}(w- 1/\ga_{n+1})=\mathbb{Y}_{n+1}(w) + \gep_{n+1}$, and 
by \refL{L:uniform-contraction-Y_r}, 
$|\mathbb{Y}_{n+1}(w +1)-\mathbb{Y}_{n+1}(w)| \leq 0.9$. 
Combining with the above inequality, we conclude that for all $w_n \in V_n^{n+1}$,
\begin{align*}
|E_n^{n+1} (w_n)- E_n^n(w_n)| 
&= \left | \mathbb{Y}_{n+1} (w_{n+1}+1-1/\ga_{n+1})+(\gep_{n+1}+1)/2 - (w_n +1 -1/\ga_n)\right | \\
&\leq \left | \mathbb{Y}_{n+1} (w_{n+1}) + \gep_{n+1}+ (\gep_{n+1}+1)/2 - (w_n +1 -1/\ga_n)\right |+ 0.9\\ 
&= |\gep_{n+1}| + 1 + 0.9 \leq 3. 
\end{align*}
Therefore, using the uniform contraction of $\mathbb{Y}_j$ in \refL{L:uniform-contraction-Y_r}, we conclude that for all 
$m \geq n+2 $, and all $w_n \in V^{m+1}_n$, we have 
\begin{equation}\label{E:L:model-map-lift-2-2}
|E_n^{m+1}(w_n)- E_n^m(w_n)| 
=|X_{n+1}^m \circ E_m^{m+1}(w_m) - X_{n+1}^m \circ E_m^m(w_m)| 
\leq (0.9)^{m-n} \cdot  3. 
\end{equation}
In particular, for $w_n \in V_n^\infty$, the above inequality holds for all $m \geq n$, and hence $E_n^m$ forms a 
uniformly Cauchy sequence. Thus, the map 
\begin{equation}\label{E:E_n-defn}
E_n = \lim_{m \to \infty} E_n^m: V^\infty_n \to I_n
\end{equation}
is well-define, and continuous. 
In particular, since $\tilde{T}_\ga(w_{-1})= (\mathbb{Y}_0+(\gep_0+1)/2) \circ E_0 (w_0)$, we conclude that $\tilde{T}_\ga$ is 
continuous on $V^\infty$. 

In order to show that $\tilde{T}_\ga$ is injective, it is enough to show that $E_0$ is injective. 
To this end, we first note that for all $n\geq 0$ and all $w_n \in V_n^\infty$, we have 
\begin{equation}\label{E:L:model-map-lift-2-3}
|E_n(w_n) - (w_n+1-1/\ga_n) | \leq \sum_{j=n}^\infty  |E_n^{j+1}(w_n) - E_n^j(w_n)| \leq 
3 \sum_{j=n}^\infty 0.9^{j-n} \leq 30.
\end{equation}
Also, for $m \geq n+1$, we may rewrite \refE{E:L:model-map-lift-2-1} as 
\begin{equation}\label{E:L:model-map-lift-2-5}
E_n^m(w_n)= (\mathbb{Y}_{n+1}+ (\gep_{n+1}+1)/2) \circ E_{n+1}^m(w_{n+1}),
\end{equation}
and then take limits as $m \to \infty$ to obtain 
\begin{equation}\label{E:L:model-map-lift-2-4}
E_n(w_n)= (\mathbb{Y}_{n+1}+ (\gep_{n+1}+1)/2) \circ E_{n+1}(w_{n+1}).
\end{equation}
The above relation holds for all $n \geq 0$ and all $w_{n+1} \in V_{n+1}^\infty$. 

Let $w_{-1}$ and  $w'_{-1}$ be distinct elements in $V^\infty$. Let $(w_i ; l_i)_{i\geq -1}$ and $(w_i'; l_i)_{i\geq -1}$ denote the 
trajectories of $w_{-1}$ and $w'_{-1}$, respectively. 
Then, for each $n \geq 0$, both $w_n$ and $w'_n$ belong to $V_n^\infty$. 
By the uniform contraction of $\mathbb{Y}_j$, there is $n \geq 0$ such that $|w_n - w_n'| \geq 61$. 
By virtue of the uniform bound in \refE{E:L:model-map-lift-2-3}, $E_n(w_n) \neq E_n(w_n')$. 
Inductive using \eqref{E:L:model-map-lift-2-4}, and the injectivity of $\mathbb{Y}_j$, we conclude that 
$E_n(w_n)\neq E_n(w_n')$, for all $n\geq 0$.  
In particular, $E_0$ is injective. 
\end{proof}

\begin{lem}\label{L:P:model-map-lift-3}
For all $n \geq 1$, $\tilde{T}_\ga: \ol{W^n} \to I_{-1}/\D{Z}$ is continuous. 
\end{lem}

\begin{proof}
Fix an arbitrary $n \geq 1$. 
By definition, $W^n \subset V^{n-1}$. 
Recall from the proof of \refL{L:P:model-map-lift-2} that the map $w_{-1} \mapsto w_{n-1}$ is a homeomorphism from 
$V^{n-1}$ onto $(1/\ga_{n-1}-1, 1/\ga_{n-1}) \cap I_{n-1}$. Thus, $V^{n-1}$ is relatively open in $I_{-1}$. 
Indeed, for $w_{-1} \in V^{n-1}$, the map $\Re w_{-1} \to \Re w_{n-1}$ is a translation independent of $w_{-1}$. 

We have 
\begin{equation}\label{E:L:P:model-map-lift-3-1}
\begin{aligned}
\{w_{n-1} \mid w_{-1} \in W^n \}=\{ w \in I_{n-1} \mid \Re w \in [a_{n-1}-1, 1/\ga_{n-1})\}, \quad \tif \gep_n=-1; \\
\{w_{n-1} \mid w_{-1} \in W^n\}=\{ w \in I_{n-1} \mid \Re w \in (1/\ga_{n-1}-1, a_{n-1}]\}, \quad \tif \gep_n=+1. 
\end{aligned}
\end{equation}
Combining with the above paragraph, we conclude that there are real numbers $x_n$ and $y_n$ such that 
either $W^n=\{w\in I_{-1} \mid \Re w \in (x_n, y_n]\}$ or $W^n=\{w\in I_{-1} \mid \Re w \in [x_n, y_n)\}$. 

Let $(w^i)_{i\geq 0}$ be a sequence in $W^n$ converging to some $w \in \ol{W^n}$.  
If $w \in W^n$, the continuity at $w$ follows from \refL{L:P:model-map-lift-1}. 
So, below we assume that $w\in \ol{W_n} \setminus W^n$. 
Note that either $\Re w^i < \Re w$ or $\Re w^i > \Re w$ holds for all $i \geq 0$. 

Let $(w_j ; l_j)_{j\geq -1}$ denote the trajectory of $w$, and for each $i \geq 0$, let $(w^i_j, l^i_j)_{j\geq -1}$ denote 
the trajectory of $w^i$. By \eqref{E:L:P:model-map-lift-3-1}, either $\gep_n=+1$ and 
$\Re w^i_{n-1} \searrow 1/\ga_{n-1}-1$, or $\gep_n=-1$ and $\Re w^i_{n-1} \nearrow 1/\ga_{n-1}$.\footnote{We use the 
notation $x_i \searrow x$ to mean that the sequence $x_i $ converges to $x$ and $\Re x_i > x$. Similarly, $x_i \nearrow x$ 
means that the sequence $x_i$ converges to $x$ and $\Re x_i < x$.} 
We consider two cases based on these scenarios. 

(i) $\gep_n=-1$ and $\Re w^i_{n-1} \nearrow 1/\ga_{n-1}$, as $i\to \infty$.   
For integers $j$ with $-1 \leq j \leq n-3$, $\lim_{i \to \infty} \Re w^i_j \notin \mathbb{Z}$. 
That is because, if there is $j$ with $0 \leq j \leq n-3$ and $\lim_{i \to \infty} \Re w^i_j \in \mathbb{Z}$ then we must have 
$\lim_{i \to \infty} \Re w^i_{j+1}= 1/\ga_{j+1}$, then $\lim_{i \to \infty} \Re w^i_{j+2}= 1/\ga_{j+2}-1$, 
and then $\lim_{i \to \infty} \Re w^i_{n-1}= 1/\ga_{n-1}-1$.  
The last property is not possible in this case.
Similarly, if $\lim_{i \to \infty} \Re w^i_{-1} \in \mathbb{Z}$ then we must have 
$\lim_{i \to \infty} \Re w^i_{0}= 1/\ga_{0}$, then $\lim_{i \to \infty} \Re w^i_{1}= 1/\ga_{1}-1$, 
and then $\lim_{i \to \infty} \Re w^i_{n-1}= 1/\ga_{n-1}-1$.  
The same contradiction. 
 
By the previous paragraph, $l^i_j=l_j$, for all $-1 \leq j \leq n-3$. 
This implies that for $-1\leq j \leq n-2$, $\lim_{i \to \infty} w^i_j= w_j$. 

The integer $n-1$ is the smallest integer with $w_{n-1} \in K_{n-1}$.
To see this, first note that either $\gep_{n-1}=-1$ and $\Re w^i_{n-2} \nearrow a_{n-2}-1$, or $\gep_{n-1}=+1$ and 
$\Re w^i_{n-2} \searrow a_{n-2}$. 
Thus, $\Re w_{n-2}= \lim_{i \to \infty} \Re w_{n-2}^i \in \mathbb{Z}$. Hence, $\Re w_{n-1}=0$. 
On the other hand, for $0 \leq j \leq n-3$, $\Re w_j \neq 1/\ga_j -1$, since otherwise $\Re w_{n-1}= 1/\ga_{n-1}-1$. 
With this paragraph, we conclude that 
\[\tilde{T}_\ga (w)=X_0^{n-1} (w_{n-1}+1).\]

Let $w_{n-1}'= \lim_{i\to \infty} w^i_{n-1}$ and $w_{n}'= \lim_{i \to \infty} w^i_{n}$. 
We must have $\Re w'_{n-1}= 1/\ga_{n-1}$ and $\Re w'_n= 1/\ga_{n}-1$.
By \refE{E:Y_n-comm-1}, we have $w'_{n-1}= w_{n-1}+1/\ga_{n-1}$, and then using \refE{E:Y_n-comm-2}, we get 
$w'_{n}+1= w_{n}+1/\ga_n$.
Therefore, 
\begin{align*}
\lim_{i \to \infty} \tilde{T}_\ga(w^i) = \lim_{i \to \infty} X_0^{n} (w^i_n+1) = X_0^n (w'_n+1) 
&= X_0^{n-1} \circ \mathbb{Y}_n (w_n'+1) \\
&= X_0^{n-1} \circ \mathbb{Y}_n (w_n+1/\ga_n) \\
&=X_0^{n-1} (\mathbb{Y}_n (w_n)+1)=X_0^{n-1} (w_{n-1}+1).
\end{align*}
This completes the proof in this case. 

(ii) $\gep_n=+1$ and $\Re w^i_{n-1} \searrow 1/\ga_{n-1}-1$, as $i\to \infty$. 
Let us first assume that for all $1 \leq j \leq n-1$ we have $\gep_j=-1$. 
This implies that $\Re w^i_j \searrow 1/\ga_j-1$, for all $0 \leq j \leq n-1$.  
In particular, $\Re w_j=1/\ga_j-1$, for all $j \geq 0$.

For $0 \leq j \leq n$, choose $t_j \geq -1$ such that $w_j=1/\ga_j-1+it_j$. 
Then, by \refE{E:Y_n-comm-1}, 
\begin{align*}
\tilde{T}_\ga(w)= \left(\mathbb{Y}_0+\frac{\gep_0+1}{2}\right)(w_0+1) = \mathbb{Y}_0 (1/\ga_0 + i t_0) + \frac{\gep_0+1}{2} 
=  \mathbb{Y}_0(it_0) - \gep_0 + \frac{\gep_0+1}{2}.
\end{align*}
On the other hand, by \eqref{E:Y_n-comm-2}, 
\[\Im \mathbb{Y}_j(i t_j)= \Im \mathbb{Y}_j(i t_j +1/\ga_j-1)= \Im \mathbb{Y}_j(w_j)= \Im w_{j-1}= t_{j-1},\]
and hence $\mathbb{Y}_j(it_j)= i t_{j-1}$. 
Therefore,  
\begin{align*}
\lim_{i \to \infty} \tilde{T}_\ga (w^i) = \lim_{i \to \infty} X_0^n(w^i_n+1) 
&= X_0^n(w_n+1) \\
&= \mathbb{Y}_0 \circ \dots \circ \mathbb{Y}_{n-1} \circ (\mathbb{Y}_n + 1)(1/\ga_n + i t_n) + \frac{\gep_0+1}{2} \\
&= \mathbb{Y}_0 \circ \dots \circ \mathbb{Y}_{n-1} \circ (\mathbb{Y}_n( i t_{n})) + \frac{\gep_0+1}{2}\\
&= \mathbb{Y}_0(it_0) + \frac{\gep_0+1}{2}. 
\end{align*}
By the above equations, $\lim_{i \to \infty} \tilde{T}_\ga (w^i) /\mathbb{Z}= \tilde{T}_\ga(w)/\mathbb{Z}$. 

Now assume that there is $1 \leq m \leq n-1$ with $\gep_m=+1$. 
Assume that $m$ is the largest integer with $1 \leq m \leq n-1$ and $\gep_m=+1$. 
As in the above paragraph, we note that for all $j$ with $m \leq j \leq n-1$ we have $\Re w^i_j \searrow 1/\ga_j-1$. 
Since $\gep_m=+1$, we must have $\Re w^i_{m-1} \nearrow 1/\ga_{m-1}$. 
As in case (i), this implies that for $0 \leq j\leq m-3$, $\lim_{i \to \infty} \Re w^i_{j} \notin \mathbb{Z}$. 
Hence, for $-1 \leq j \leq m-2$, $w_j= \lim_{i\to \infty} w^i_j$.   
Moreover, $\lim_{i \to \infty} \Re w^i_{m-2} \in \mathbb{Z}$. 

By the above paragraph, $\Re w_{m-1}=0$, and $m-1$ is the smallest positive integer with $w_{m-1} \in K_{m-1}$.  
Then, 
\[\tilde{T}_\ga (w)=X_0^{m-1} (w_{m-1}+1).\]

For $m-1 \leq j \leq n $, let $w'_j= \lim_{i \to \infty} w^i_j$. 
We have $w_j=w_j'+1-1/\ga_j$. 
Then, as in the previous cases, we note that 
\begin{align*}
\lim_{i \to \infty} \tilde{T}_\ga(w^i) & = \lim_{i \to \infty} X_0^{n} (w^i_n+1) \\
&= X_0^n (w'_n+1) \\
&= X_0^{m-1} \circ (\mathbb{Y}_m+1) \circ \mathbb{Y}_{m+1}  \dots    \circ \mathbb{Y}_{n-1} \circ  (\mathbb{Y}_n+1) (w_n'+1) \\
&= X_0^{m-1} \circ (\mathbb{Y}_m+1) \circ \mathbb{Y}_{m+1}  \dots    \circ \mathbb{Y}_{n-1} \circ  (\mathbb{Y}_n+1) (w_n+1/\ga_n) \\
&= X_0^{m-1} \circ (\mathbb{Y}_m+1) \circ \mathbb{Y}_{m+1}  \dots    \circ \mathbb{Y}_{n-1} (w_{n-1}) \\
&= X_0^{m-1} (w_{m-1}+1).
\end{align*}
This completes the proof in this case. 
\end{proof} 

\begin{lem}\label{L:P:model-map-lift-4}
The set $W^0$ is closed, and $\tilde{T}: W^0/\D{Z} \to I_{-1}/\D{Z}$ is well-defined, continuous, and injective. 
\end{lem}

\begin{proof}
This is straightforward and is left to the reader. 
\end{proof}

\begin{proof}[Proof of \refL{L:model-map-lift}]
First we prove the injectivity. 
For $n\geq 1$, each equivalence class in $W^n/\D{Z}$ consists of a single element, because the horizontal width of $W^n$ 
is at most $\ga_0 \ga_1 \dots \ga_{n-1} \leq 1/2$. 
Also, the width of $\tilde{T}_\ga (W^n)$ is at most $1/2$, which implies that each equivalence class in $\tilde{T}_\ga(W^n)/\D{Z}$
consists of a single element. 
Then, by \refL{L:P:model-map-lift-1}, $\tilde{T}_\ga: W^n/\D{Z} \to I_{-1}/\D{Z}$ is injective. 
By \refL{L:P:model-map-lift-4}, $\tilde{T}_\ga: W^0/\D{Z} \to I_{-1}/\D{Z}$ is injective.
Similarly, $V^\infty$ consists of a single half-infinite vertical line which is mapped to a single half-infinite vertical 
line by $\tilde{T}_\ga$. 
Thus, by \refL{L:P:model-map-lift-2}, $\tilde{T}_\ga : V^\infty/\D{Z} \to V^\infty/\D{Z}$ is injective.  
on the other hand, for $0 \leq n < m$, the sets $\tilde{T}_\ga (W^n)$, $\tilde{T}_\ga (W^m)$ and 
$\tilde{T}_\ga (V^\infty)$, are pairwise disjoint. Therefore, $\tilde{T}_\ga: I_{-1}/\D{Z} \to I_{-1}/\D{Z}$ is injective. 
 
Now we prove the continuity. 
Let $(w^i)_{i\geq 0}$ be a convergent sequence in $I_{-1}/\D{Z}$. Without loss of generality, we may assume that 
$(w^i)_{i\geq 0}$ converges in $I_{-1}$. 
By Lemmas \ref{L:P:model-map-lift-2} and \ref{L:P:model-map-lift-3},
if a subsequence of this sequence lies in some $W^n$ or in $V^\infty$, then $\tilde{T}_\ga$ is continuous along that subsequence. 
Therefore, it is enough to deal with the case of $w^i \in W^{j_i}$, with $j_i \to +\infty$ as $i \to +\infty$.
Without loss of generality (in order to simplify the notations) we may assume that $w^i \in W^{i+1}$, for $i\geq 0$. 

Let $w=\lim_{i \to \infty} w^i$. Below, we consider three cases, based on the location of $w$.  

{\em Case 1: $w \in V^\infty$.} 
Let $(w_j ; l_j)_{j \geq -1}$ denote the trajectory of $w$, and for $i\geq 0$, let $(w^i_j ; l^i_j)_{j\geq -1}$ 
denote the trajectory of $w^i$. 
Recall from the proof of \refL{L:P:model-map-lift-2} that 
\[\tilde{T}_\ga(w_{-1})=(\mathbb{Y}_0+(\gep_0+1)/2) \circ E_0(w_0),\] 
and $E_0^m \to E_0$, uniformly on $V^\infty_0$. 

Fix an arbitrary $\gep>0$. 
Let us choose an integer $m \geq 5$ such that for all $i \geq m$ we have 
\[|E_0(w_0) - E_0^i(w_0)| \leq \gep/2, \qquad 63 \cdot (0.9)^m \leq \gep/2.\]

For $i\geq m$, $w^i \in W^{i+1} \subset V^i \subset V^m$, and $w \in V^\infty \subset V^m$. 
Then, by \refL{L:P:model-map-lift-1}, $w^i_m \to w_m$, as $i \to \infty$. 
In particular, there is $N \geq 0$ such that for all $i \geq N$ we have $|w_m - w^i_m| \leq 1$. 
Below we assume that $i \geq \max\{N, m\}$, and we aim to show \refE{E:P:model-map-lift-3}. 

Since both $w^i_m$ and $w_m$ belong to $V^i_m$, we may employ the estimate in \refE{E:L:model-map-lift-2-2}, to obtain 
\begin{equation}\label{E:P:model-map-lift-2}
\begin{aligned}
|E_m^{i}(w_m)  & - E_m^{i}(w_m^i) | \\
&\leq \sum_{l=m}^{i-1}|E_m^{l+1}(w_m) - E_m^{l}(w_m)| 
+ |w_m-w_m^i| + \sum_{l=m}^{i-1}|E_m^{l+1}(w^i_m) - E_m^{l}(w^i_m)| \\
& \leq 2 \cdot 3 \sum_{l=m}^{\infty} (0.9)^{l-m} + |w_m-w^i_m| \leq 60 +1=61. 
\end{aligned}
\end{equation}
Note that 
\[\left | (w^i_{i}+1-1/\ga_{i}) - \left (\mathbb{Y}_{i+1}+\frac{\gep_{i+1}+1}{2}\right ) (w^i_{i+1})\right | \leq 1.\]
Combining the above inequality with the uniform contraction in \refL{L:uniform-contraction-Y_r}, we obtain
\begin{align*}
\bigg | E_{i}^{i}&(w^i_{i})  - \left (\mathbb{Y}_{i+1}+\frac{\gep_{i+1}+1}{2}\right ) (w^i_{i+1}+1) \bigg | \\
& = \left | (w^i_{i}+1-1/\ga_{i}) - \left (\mathbb{Y}_{i+1}+\frac{\gep_{i+1}+1}{2}\right ) (w^i_{i+1}) \right. \\
& \left. \quad \quad + \left (\mathbb{Y}_{i+1}+\frac{\gep_{i+1}+1}{2}\right ) (w^i_{i+1}) 
- \left (\mathbb{Y}_{i+1}+\frac{\gep_{i+1}+1}{2}\right ) (w^i_{i+1}+1) \right |
 \leq 1 + 0.9 \cdot 1 \leq 2. 
\end{align*}
Using the uniform contraction one more time, this give us 
\begin{align*}
\left| E_m^{i}(w_m^i) - X_{m+1}^{i+1} (w^i_{i+1}\!+\!1) \right |  
= \left| X_{m+1}^{i} \circ E_{i}^{i}(w^i_{i}) -\!  X_{m+1}^{i} \circ \left (\mathbb{Y}_{i+1} +\frac{\gep_{i}+1}{2} \right) (w^i_{i+1}+1) \right |
\leq 2. 
\end{align*}
Combining the above bound with the bound in \refE{E:P:model-map-lift-2}, we conclude that 
\[\left| E_m^{i}(w_m) -  X_{m+1}^{i+1} (w^i_{i+1}+1) \right | \leq 61+2=63.\]

Using the relation in \refE{E:L:model-map-lift-2-5} several times, we note that 
$(\mathbb{Y}_0+(\gep_0+1)/2) \circ E_0^i(w_0) =X_0^{m-1} \circ E_m^i(w_m)$. 
Thus, using the above equation with $m=0$, we obtain 
\begin{equation}\label{E:P:model-map-lift-3}
\begin{aligned}
|\tilde{T}_\ga(w) - \tilde{T}_\ga(w^i)|
& = \left | (\mathbb{Y}_0+(\gep_0+1)/2) \circ E_0(w_0) - X_0^{m-1} \circ X_m^{i+1}(w^i_{i+1}+1) \right |  \\
& \leq | (\mathbb{Y}_0+(\gep_0+1)/2) \circ E_0(w_0)- (\mathbb{Y}_0+(\gep_0+1)/2) \circ E_0^i(w_0)| \\
& \qquad \qquad \qquad + | X_0^{m-1} \circ E_m^i(w_m) - X_0^{m-1} \circ X_m^{i+1}(w^i_{i+1}+1)| \\
& \leq (0.9) \cdot \gep/2 + (0.9)^m \cdot (63) \leq \gep. 
\end{aligned}
\end{equation} 
This completes the proof in Case 1. 

{\em Case 2: $w \in W^n$ for some $n\geq 1$.}
Recall from the proof of \refL{L:P:model-map-lift-2}, that there are real numbers $x_n$ and $y_n$ such that either 
$W^n= \{w\in I_{-1} \mid \Re w \in [x_n, y_n)\}$ or $W^n= \{w\in I_{-1} \mid \Re w \in (y_n, x_n]\}$. 
Since, $W^n$ and $W^{i+1}$ are pairwise disjoint, for $i+1 > n$, and $w_i \in W^{i+1}$, we must have $\Re w= x_n$. 

Note that for integers $j$ with $0 \leq j \leq n-2$, $\Re w_j \notin \D{Z}$. 
This implies that $w^i_{n-1} \to w_{n-1}$, as $i \to \infty$.
It follows from \refE{E:L:P:model-map-lift-3-1} that $\Re w_{n-1} \in \D{Z}$. 
Hence, $\Re w_n=0$.
Below we consider two scenarios.

{\em Case 2-i: $\gep_n=-1$.}
Since $w^i_n \notin K_n$, for large $i$, we must have $\Re w^i_{n-1} \nearrow \Re w_{n-1}$. 
It follows that $\Re w^i_n \nearrow 1/\ga_n + w_n$. 
Let us define $w_n'= 1/\ga_n + w_n$. 
For $i \geq n$, let $(w'_i; l_i')_{i\geq n}$ denote the trajectory of $w_n'$, defined in the same fashion according to 
\refE{E:trajectory-real-parts}.    
We must have $\Re w'_{i}=1/\ga_{i}-1$, for all $i \geq n+1$, and by \refE{E:Y_n-comm-2}, $w_i=w_i' + 1-1/\ga_i$, for all $i\geq n$.

If $\gep_{n+1}=-1$, then we must have $\Re w_{n+1}^i \nearrow 1/\ga_{n+1}-1$. 
This is a contradiction, since this implies that $w^i \in W^{n+1}$, for sufficiently large $i$. 
Therefore, we have $\gep_{n+1}=+1$ and $\Re w^i_{n+1} \searrow 1/\ga_{n+1}-1$. 
Indeed, for the same reason, we must have $\gep_i=-1$, for all $i\geq n+2$. 

Assume that  $i \geq n$. We have 
\[\left| (w_{i} - \left (\mathbb{Y}_{i+1}+\frac{\gep_{i+1}+1}{2}\right ) (w^i_{i+1}) \right| \leq 1.\]
Therefore, applying $X_{n+1}^i$, we obtain 
\begin{align*}
\left | w_n -   X_{n+1}^{i+1}(w^i_{i+1}+1) \right | \leq 1 \cdot (0.9) ^{i-n}. 
\end{align*}
Then, 
\begin{align*}
\left |\tilde{T}_\ga (w) - \tilde{T}_\ga (w^i) \right | 
&=\left | X_0^{n} (w_n)  - X_0^{i+1}(w^i_{i+1}+1) \right | \\
& \leq \left | X_0^{n} (w_n) - X_0^n \circ X_{n+1}^{i+1}(w^i_{i+1}+1) \right | 
\leq 1 \cdot (0.9)^{i+1}. 
\end{align*}
This completes the proof in this case.  

{\em Case 2-ii: $\gep_n=+1$.}
Here, we must have $\Re w^i_{n-1} \searrow \Re w_{n-1}$.  
It follows that $w^i_n \nearrow  1/\ga_n + w_n$. Let $w_n'= 1/\ga_n + w_n$. 
Let $(w'_i; l'_i)_{i\geq n}$ denote the trajectory of $w'_n$. 
We must have $\gep_{n+1}=+1$, otherwise, $\Re w^i_{n+1} \nearrow 1/\ga_{n+1}-1$ and hence $w^i \in W^n$ 
for sufficiently large $i$. For the same reason, we must have $\gep_i=-1$, for all $i \geq n+2$. 
As in the previous case, we obtain $\tilde{T}_\ga (w^i) \to \tilde{T}_\ga (w)$. 

{\em Case 3: $w \in W_0$.}
If $\gep_0=-1$, either $\Re w^i \nearrow 1$ and hence $\Re w^i_0 \nearrow 1/\ga_0$, or $\Re w^i \searrow 1-\ga_0$
and hence $\Re w^i_0 \searrow 1/\ga_0-1$.
Similarly, if $\gep_0=+1$, either $\Re w^i \nearrow \ga_0$ and hence $\Re w^i_0 \searrow 1/\ga_0-1$, 
or $\Re w^i \searrow 0$ and hence $\Re w^i_0 \nearrow 1/\ga_0$. 
All these scenarios may be dealt with as in Case 2.
\end{proof}

\subsection{The map \texorpdfstring{$\mathbb{T}_\ga$}{T-ga}, and its properties}\label{SS:T-ga-properties}

\begin{propo}\label{P:T-ga-tangential}
For every $\ga \in \D{R}\setminus \D{Q}$, the map $\tilde{T}_{\ga}: I_{-1} \to I_{-1}$ induces 
a homeomorphism
\[\mathbb{T}_\ga:\mathbb{M}_\ga \to \mathbb{M}_\ga,\]
via the projection $w \mapsto s(e^{2\pi i w})$, that is, 
$s(e^{2\pi i \tilde{T}_\ga(w)})= \mathbb{T}_\ga (s(e^{2\pi i w}))$ for all $w\in I_{-1}$. 
Moreover, $\mathbb{T}_\ga$ acts as rotation by $2\pi \ga$ in the tangential direction, 
that is, there is a function $g_\ga$ such that 
\[\mathbb{T}_\ga(r e^{2\pi i \theta})= g_\ga(r, \theta) e^{2\pi i (\theta+\alpha)}\]
for every $r e^{2\pi i\theta} \in \mathbb{M}_\ga$. 
\end{propo}

\begin{proof}
By \refL{L:model-map-lift}, $\tilde{T}_{\ga}: I_{-1}/\D{Z} \to I_{-1}/ \D{Z}$ induces a continuous and injective map 
$\mathbb{T}_\ga$ of the set $\{s(e^{2\pi i w}) \mid w\in I_{-1}\}$. 
From the construction, we note that as $\Im z \to \infty$ in $I_{-1}/\mathbb{Z}$, 
$\Im \tilde{T}_\ga(z) \to +\infty$. 
This implies that we may continuously extend $\mathbb{T}_\ga$ onto $0$ by setting $\mathbb{T}_\ga(0)=0$. 
Since $\mathbb{M}_\ga$ is compact, and $\mathbb{T}_{\ga}: \mathbb{M}_\ga \to \mathbb{M}_\ga$ is continuous 
and injective, it must be a homeomorphism. 

In order to show that $\mathbb{T}_\ga$ acts as rotation by $2\pi \ga$ in the tangential direction, it is enough to show 
that $\tilde{T}_\ga$ acts as translation by $-\alpha$ on $\cup_{n\geq 0}W^n \cup V^\infty$. 
However, because $V^\infty$ has empty interior, by the continuity of $\tilde{T}_\ga$, it is enough to prove this 
on the sets $W^n$. 
The latter property follows from the definition of $\tilde{T}_\ga$ and the functional relation in \refE{E:Y_n-comm-1}. 
We present the details below. 

For $w_{-1} \in W^0$, we have 
\[w_{-1}= (\mathbb{Y}_0+(1+\epsilon_0)/2)(w_0), \qquad \tilde{T}_\ga(w_1)= (\mathbb{Y}_0+(1+\epsilon_0)/2)(w_0+1).\] 
Thus, by the definition of $\mathbb{Y}_0$, 
\[\Re (\mathbb{Y}_0+(1+\epsilon_0)/2)(w_0+1) = \Re w_{-1} -\epsilon_0 \ga_0.\] 
Therefore, 
\[\arg \left(\mathbb{T}_{\ga}(s(e^{2\pi i w_{-1}}))\right) 
= \arg (s(e^{2\pi i w_{-1}})) + \epsilon_0\ga_0
=\arg (s(e^{2\pi i w_{-1}})) + \ga.\]

Now, fix an arbitrary $n\geq 1$, and let $w_{-1} \in W^n$ be an arbitrary point with trajectory $(w_i, l_i)_{i\geq -1}$. 
For $j=n, n-1, n-2, \dots, -1$, let us define the points $\xi_j= X_{j+1}^n(w_n+1)$.
We note that in this case we must have $\gep_n=-1$, and $\gep_i=+1$ for all $i=n-1, n-2, \dots, 1$. 
Then, using $1/\ga_{n-1}= a_{n-1} -\ga_n$ and \refE{E:I_n^j--1}, we obtain 
\begin{align*}
\Re w_{n-1} - \Re \xi_{n-1} 
&= \left( \Re (\mathbb{Y}_n(w_n)) + a_{n-1}-1 \right) - \Re (\mathbb{Y}_n(w_n+1))   \\
&= (a_{n-1} -1) - \ga_n \\
&= \left(\frac{1}{\ga_{n-1}} + \ga_n -1 \right) - \ga_n = \frac{1}{\ga_{n-1}}-1.
\end{align*}
For the next step, we have $1/\ga_{n-2}= a_{n-2} + \ga_{n-1}$, and use \refE{E:I_n^j-+1}, to obtain 
\[\Re w_{n-2} - \Re \xi_{n-2}= -\ga_{n-1}\left(\frac{1}{\ga_{n-1}}-1\right) + a_{n-2}=  \frac{1}{\ga_{n-2}} -1.\]
Repeating the above process for levels $n-3, n-4, \dots, 0$, we end up with 
\[\Re w_{0} - \Re \xi_{0}= \frac{1}{\ga_0}-1.\]
In the last stage we apply the map $\mathbb{Y}_0+(1+\gep_0)/2$, to obtain,
\[\Re w_{-1} - \Re \xi_{-1}= -\gep_0 \ga_0 \left( \frac{1}{\ga_0}-1\right) =- \gep_0 + \gep_0\ga_0.\]
Thus, $\tilde{T}_\ga$ acts as translation by $-\gep_0\ga_0$ which is the same as $-\ga$ modulo $\mathbb{Z}$. 
This completes the argument for the set $W^n$. 
\end{proof}

\begin{propo}\label{P:T-ga-relations}
For every $\ga \in \mathbb{R} \setminus \mathbb{Q}$, we have $s \circ \mathbb{T}_\ga \circ s = \mathbb{T}_{-\ga}$ 
on $\mathbb{M}_{-\ga}$.
\end{propo}

\begin{proof}
Recall that by \refP{P:M-ga-relations-1}, we have $s(\mathbb{M}_{\ga})= \mathbb{M}_{-\ga}$. 
Thus both of the maps $\mathbb{T}_{-\ga}$ and $s\circ \mathbb{T}_\ga \circ s$ are defined on $\mathbb{M}_{-\ga}$
and map into $\mathbb{M}_{-\ga}$. 

When $\ga$ changes to $-\ga$, $\gep_0$ changes to $-\gep_0$, but all the remaining numbers $\ga_i$
and $\gep_{i+1}$, for $i \geq 0$, remain the same. 
Let $I_{-1}$ and $I'_{-1}$ denote the corresponding sets for $\ga$ and $-\ga$, respectively. 
We have $-s(I'_{-1})= I_{-1}$; see proof of \refP{P:M-ga-relations-1}. 
From the definition of the maps $\tilde{T}_\ga$ and $\tilde{T}_{-\ga}$, we can see that 
$(-s) \circ \tilde{T}_{\ga} \circ (- s) = \tilde{T}_{-\ga}$ on $I'_{-1}$.
Projecting onto $\mathbb{M}_{-\ga}$ via $w \mapsto s(e^{2\pi i w})$ we obtain the desired relation. 
\end{proof}

\begin{propo} \label{P:orbit-size}
There is a constant $C$ such that for every $\alpha \in (-1/2, 1/2) \setminus \mathbb{Q}$ and every integer 
$k$ satisfying $0 \leq k < 1/|\ga|$, we have 
\[\frac{C^{-1}}{1+ \min \{k, |\ga|^{-1}-k\}} \leq |\mathbb{T}_\ga\co{k}(+1)| \leq \frac{C}{1+ \min \{k,|\ga|^{-1}-k\}}.\]
\end{propo}

\begin{proof}
By virtue of \refP{P:T-ga-relations}, it is enough to show this for $\ga \in (0, 1/2) \setminus \mathbb{Q}$. 
In that case, we have $\ga_0=\ga$, $\gep_0=+1$, and $\mathbb{Y}_0=-s \circ Y_{\ga_0}$. 
The point $+1$ may be lifted under the projection $w \mapsto s(e^{2\pi i w})$ to the point $+1$ in $I_{-1}$. 
In the trajectory of $+1$, we have $l_{-1}=+1$ and $w_0=0$. 
Therefore, by the definition of $\tilde{T}_{\ga}$, for all integers $k$ satisfying $0 \leq k < 1/\ga$, we have 
$\tilde{T}_\ga\co{k} (+1)= \mathbb{Y}_0(k)+(1+\gep_0)/2=\mathbb{Y}_0(k)+1$. 
This implies that 
\[|\mathbb{T}_\ga\co{k}(+1)| =\left  | s\left (e^{2\pi i \mathbb{Y}_0(k)} \right) \right | 
= \Big |\frac{e^{-3\pi\ga}- e^{\pi\ga i}}{e^{-3\pi\ga}- e^{-\pi\ga i}e^{-2\pi\ga i k }}\Big|.\]
We need to estimate the right hand side of the above equation. 
The numerator of that formula is proportional to $\ga$; 
see Equations \eqref{E:P:Y_r-vs-h_r-0} and \eqref{E:P:Y_r-vs-h_r^-1} for a more precise estimate on this.
For the denominator we have, 
\begin{align*}
\left |e^{-3\pi\ga}- e^{-\pi\ga i}e^{-2\pi\ga i k }\right | 
&\leq \left|e^{-3\pi \ga}-1 \right| + \left|1 - e^{-\pi \ga i}\right| + \left| e^{-\pi \ga i} - e^{-\pi\ga i}e^{-2\pi\ga ik}\right| \\
& \leq 3 \pi \ga + \pi \ga + 2\pi \ga \min \{k, \ga^{-1} -k\} \\
& \leq 4 \pi \ga \left ( 1+ \min \{k, \ga^{-1} -k\}\right ). 
\end{align*}
To give a lower bound on the left hand side of the above equation, let us consider two cases. 
If $k \leq 5$, or $\ga^{-1} -k \leq 5$, we have 
\begin{align*}
\left |e^{-3\pi\ga}- e^{-\pi\ga i}e^{-2\pi\ga i k }\right | 
&  \geq \left| e^{-3\pi \ga}  -1 \right| \\
&  \geq \ga \\
&  \geq \frac{\ga}{6} \left (1+ \min \{k, \ga^{-1} -k\}\right ). 
\end{align*}
If $k$ satisfies $5 \leq k \leq \ga^{-1} -5$, by the triangle inequality, we have 
\begin{align*}
\left |e^{-3\pi\ga}- e^{-\pi\ga i}e^{-2\pi\ga i k }\right | 
&\geq \left|e^{-\pi\ga i}e^{-2\pi\ga ik} - e^{-\pi \ga i} \right| - \left|e^{-\pi \ga i} - e^{-3\pi \ga} \right| \\
& \geq 2\pi \ga \min \{k, \ga^{-1} -k\} - (\pi \ga + 3\pi \ga) \\
& \geq 2 \pi \ga \left (\min \{k, \ga^{-1} -k\} -2 \right) \\
& \geq \pi \ga \left (\min \{k, \ga^{-1} -k\} + 1 \right).
\end{align*}
Combining the above inequalities, we obtain the desired estimate in the proposition for a suitable constant $C$. 
\end{proof}

The estimate of the sizes of the orbit presented in \refP{P:orbit-size} widely holds for holomorphic map of the 
from $e^{2\pi i \ga} z+ a_2 z^2+ \dots$, with $a_2 \neq 0$. 
More precisely, for a fixed nonlinearity (higher order terms), and a choice of a point close enough to the fixed point 
at $0$, there is a constant $C$ such that the upper and lower bounds in the estimate
hold for those integers. 
The constant $C$ is independent of the rotation number $\ga$ and the number of iterates $k$. 
For large values of $\ga$, this mostly follows from continuity of the map. 
For smaller values of $\ga$, one employs the existence of perturbed Fatou-coordinates, and elementary estimates on 
their behaviour. See \cite{Sh00} for more details on this. 
\section{The renormalisation operator \texorpdfstring{$\mathcal{R}_m$}{R-m}}\label{S:renormalisation-model}
In this section we define the renormalisation of $\mathbb{T}_\ga: \mathbb{M}_\ga \to \mathbb{M}_\ga$
using a suitable return map to a fundamental set in $\mathbb{M}_\ga$. 
This construction is in the spirit of the sector renormalisation, qualitatively presented by Douady and 
Ghys \cite{Do86}, and quantitatively employed by Yoccoz in the study of the small divisors problem in 
complex dimension one \cite{Yoc95}.
Our construction of renormalisation relies on the model for the changes of coordinates we defined 
in \refS{S:change-coordinates}. 

\subsection{Definition of \texorpdfstring{$\mathcal{R}_m$}{R-m}}\label{SS:definition-Renorm}
Fix an arbitrary $\ga \in \mathbb{R} \setminus \mathbb{Q}$. 
Since $\mathbb{M}_{\ga+1}= \mathbb{M}_\ga$ and $\mathbb{T}_{\ga+1}= \mathbb{T}_\ga$, we may assume that 
$\ga \in (-1/2,1/2)$. 
Let us first assume that $\ga \in (0, 1/2)$. 

Consider the set 
\[S_\ga= \{ z \in \mathbb{M}_\ga \setminus \{0\}  \mid  \arg z \in ([0, 2\pi \ga)+2 \pi \mathbb{Z})\}.\]
This consists of all the non-zero points in $\mathbb{M}_\ga$ which lie in a ``sector'' of angle $2\pi \alpha$ at $0$. 
Because $\mathbb{T}_\ga$ acts as rotation by $2\pi \alpha$ in the tangential direction, \refP{P:T-ga-tangential}, 
for every $z \in S_\ga$, there is the smallest integer $k_z \geq 2$ such that $\mathbb{T}_\ga^{\circ k_z}(z) \in S_\ga$. 
Indeed, for every $z \in S$, either $k_z$ is equal to the integer part of $1/\ga$ or $k_z$ is equal to the 
integer part of $1/\ga$ plus $1$. 
The map $\mathbb{T}_\ga^{\circ k_z}$ is the first return map to $S_\ga$ in the dynamical system 
$\mathbb{T}_\ga: \mathbb{M}_\ga \to \mathbb{M}_\ga$.   

Recall that $\mathbb{Y}_0(w+ 1/\ga_0)= \mathbb{Y}_0(w) + (1+\gep_0)/2=\mathbb{Y}_0+1$. 
Let us consider the map 
\[\psi_\ga: \mathbb{H}' \to \mathbb{C} \setminus \{0\}\] 
defined as 
\[\psi_\ga(w)= s \left( e^{2\pi i \mathbb{Y}_0(w)} \right)
=  s \left( e^{2\pi i (-s\circ Y_{\ga_0}(w))} \right)
=\Big |\frac{e^{-3\pi\ga}- e^{\pi\ga i}}{e^{-3\pi\ga}- e^{-\pi\ga i}e^{-2\pi\ga iw}}\Big| e^{2\pi i \ga \Re w}.\]
We note that 
\begin{equation}\label{E:renorm-def-1}
S_\ga \subset \psi_\ga \left( \{w \in \mathbb{H}'  \mid \Re w \in [0, 1)\}\right).  
\end{equation}
The map $\psi_\ga$ is continuous and injective on $ \{w \in \mathbb{H}'  \mid \Re w \in [0, 1)\}$, and 
it sends half-infinite vertical lines in that set to straight rays landing at $0$. 

There is a continuous inverse branch of $\psi_\ga$ defined on $S_\ga$, that is, 
\begin{equation}\label{E:renorm-def-2}
\phi_\ga: S_\ga \to \{w \in \mathbb{H}'  \mid \Re w\in [0, 1)\}.
\end{equation}
The above map is continuous and injective. 
(This map is the analogue of the perturbed Fatou coordinate for $\mathbb{T}_\ga$) 

The return map $\mathbb{T}_\ga\co{k_z}: S_\ga \to S_\ga$ induces a map  
\[h_\ga: \phi_\ga(S_\ga) \to \phi_\ga(S_\ga)\] 
via $\psi_\ga$ and $\phi_\ga$, that is,  
\[h_\ga(w)= \phi_\ga \circ \mathbb{T}_\ga^{\circ k_{\psi_\ga(w)}} \circ \psi_\ga(w).\]
Using the projection $w \mapsto e^{2\pi i w}$, the map $h_\ga$ projects to a map $E_\ga$, defined on 
$e^{2\pi i  \phi_\ga(S_\ga)} \subset \mathbb{C}\setminus \{0\}$. 
That is, $e^{2\pi i h_\ga(w)}= E_\ga (e^{2\pi i w})$ for all $w\in \phi_\ga(S_\ga)$. 
We may extend $E_\ga$ onto $0$ by letting it map $0$ to $0$. 
We call the map $E_\ga$, and its domain of definition, the renormalisation of 
$\mathbb{T}_\ga: \mathbb{M}_\ga \to \mathbb{M}_\ga$, and denote it by 
$\mathcal{R}_m(\mathbb{T}_\ga: \mathbb{M}_\ga \to \mathbb{M}_\ga)$. 

\begin{rem}
In the above definition, one may replace the set $S_\ga$ by any set of the form 
\[S_\ga^\eta 
=\{z \in \mathbb{M}_\ga \setminus \{0\} \mid \arg z \in \left([2\pi \eta, 2\pi\eta+2\pi\alpha) + 2\pi\mathbb{Z}\right)\}\]
for a fixed $\eta \in \mathbb{R}$, and use the return map of $\mathbb{T}_\ga$ on $S_\ga^\eta$ to
define the renormalisation of $\mathbb{T}_\ga: \mathbb{M}_\ga \to \mathbb{M}_\ga$ in the same fashion. 
The definition of renormalisation is independent of the choice of $\eta$. 
That is, one obtains the same map defined on the same set of points. 
We avoid considering this generality for the sake of simplicity, although, the independence from $\eta$ will 
become evident from the proof of \refP{P:renormalisation-relations}.

Also, when projecting the map $h_\ga$ to obtain $E_\ga$ we have used the projection map $w \mapsto e^{2\pi i w}$ 
for the sake of consistency with how the sector renormalisation is defined.
This may appear in contrast to the projection map $w\mapsto s\left(e^{2\pi i w} \right)$ used in the definition of the 
sets $\mathbb{M}_\ga$. 
Due to the properties in \refP{P:M-ga-relations-1} and \refL{L:M-ga-relations-2}, this does not cause any problems. 
See the proof of \ref{P:renormalisation-relations} for more details. 
\end{rem} 

The above process defines the renormalisation of $\mathbb{T}_\ga: \mathbb{M}_\ga \to \mathbb{M}_\ga$ 
when $\ga \in (0, 1/2)$. 
For $\ga \in (-1/2, 0)$, we first use dynamical conjugation via the complex conjugation map so that the 
rotation number of the map becomes $-\ga \in (0, 1/2)$ and then repeat the above process. 
By \refP{P:M-ga-relations-1}, we have $s(\mathbb{M}_\ga)= \mathbb{M}_{-\ga}$ and by \refP{P:T-ga-relations} 
we have $s \circ \mathbb{T}_\ga \circ s= \mathbb{T}_{-\ga}$. 
Therefore, for $\ga \in (-1/2, 0)$, we define 
\begin{equation}\label{E:def-renorm-neg-alpha}
\begin{aligned}
\mathcal{R}_m(\mathbb{T}_\ga: \mathbb{M}_\ga \to \mathbb{M}_\ga) 
&=s \Big ( \mathcal{R}_m \big(s \circ \mathbb{T}_\ga \circ s: s\left(\mathbb{M}_\ga\right) \to s\left(\mathbb{M}_\ga\right)\big)\Big )  \\
&= s\big(\mathcal{R}_m(\mathbb{T}_{-\ga}: \mathbb{M}_{-\ga} \to \mathbb{M}_{-\ga}) \big),
\end{aligned}
\end{equation}
where the complex conjugation $s$ on the right hand side of the above equation means that we consider the
dynamical conjugate of the map and its domain of definition using the complex conjugation map. 
More precisely, in terms of our notations, that means 
\[s(f: X \to Y)= \left( s\circ f \circ s:  s(X) \to s(Y)\right).\] 

\subsection{Invariance of the class of maps under the renormalisation operator}\label{SS:invariance-renormalisation}
\begin{propo}\label{P:renormalisation-relations}
For every $\ga\in (-1/2, 1/2) \setminus \mathbb{Q}$ we have 
\[\mathcal{R}_m(\mathbb{T}_\ga: \mathbb{M}_\ga \to \mathbb{M}_\ga)
= (\mathbb{T}_{-1/\ga}: \mathbb{M}_{-1/\ga} \to \mathbb{M}_{-1/\ga}).\]
\end{propo}

\begin{proof}
Let us first assume that $\ga \in (0, 1/2) \setminus \mathbb{Q}$. 
Recall the numbers $(\ga_n)_{n\geq 0}$ and $(\gep_n)_{n\geq 0}$ defined in \refS{SS:modified-fractions-mini}. 
We shall also use the notations $I_n$, $K_n$ and $J_n$ introduced in \refS{SS:tilings-nest}.  

Let us consider the set 
\[\hat{S}_\ga= \{w\in I_0 \mid \Re w \in [0, 1)\}.\]
By the definition of $\mathbb{M}_\ga$ in \refS{S:M-ga}, in particular Equations \eqref{E:I--1} and \eqref{E:M_ga}, 
we have $\psi_\ga(I_0) \cup\{0\}= \mathbb{M}_\ga$, and 
\[\psi_\ga (\hat{S}_\ga)=S_\ga.\] 
Indeed, $\psi_\ga:   \hat{S}_\ga \to S_\ga$ is a homeomorphism. 
This implies that 
\[\phi_\ga: S_\ga \to \hat{S}_\ga\]
is a homeomorphism as well. 
Therefore, we have 
\[h_\ga: \hat{S}_\ga \to \hat{S}_\ga.\] 
Because $\ga \in (0, 1/2)$ and $I_0$ is periodic of period $+1$, by \refP{P:M-ga-relations-1} 
and \refL{L:M-ga-relations-2}, we have 
\[\{e^{2\pi i w} \mid w \in \hat{S}_\ga\} \cup \{0\} = s (\mathbb{M}_{1/\ga})= \mathbb{M}_{-1/\ga}.\]
This shows that the map $h_\ga$ projects to the map $E_\ga$ defined on $\mathbb{M}_{-1/\ga}$, via the projection 
$w \mapsto e^{2\pi i w}$. 
In other words, the domain of definition of 
$\mathcal{R}_m(\mathbb{T}_\ga: \mathbb{M}_\ga \to \mathbb{M}_\ga)$ is $\mathbb{M}_{-1/\ga}$. 
Now, we need to show that $E_\ga= \mathbb{T}_{-1/\ga}$. 

We continue to assume that $\ga \in (0, 1/2)$.
Let us consider the set 
\[\tilde{S}_\ga= \{w \in I_{-1} \mid \Re w \in (1-\ga, 1].\]
The map $w \mapsto s(e^{2\pi i w})$ is a homeomorphism from $\tilde{S}_\ga$ to $S_\ga$. 
We also have the homeomorphism 
\[\mathbb{Y}_0+1: \hat{S}_\ga \to \tilde{S}_\ga.\]
Recall from \refS{S:T-on-M}, that $\mathbb{T}_\ga$ on $\mathbb{M}_\ga$ 
is induced from $\tilde{T}_\ga$ on $I_{-1}$, via the projection $w \mapsto s(e^{2\pi i w})$. 
Therefore, the return map of $\mathbb{T}_\ga: \mathbb{M}_\ga \to \mathbb{M}_\ga$ on $S_\ga$ 
corresponds to the return map of $\tilde{T}_\ga: I_{-1} \to I_{-1}$ on $\tilde{S}_\ga$.  

The map $\mathbb{Y}_0+ (1+\gep_0)/2= \mathbb{Y}_0+1: I_0 \to I_{-1}$ is a homeomorphism. 
Let us consider the map 
\[\hat{T}_\ga= (\mathbb{Y}_0+1)^{-1} \circ \tilde{T}_\ga \circ (\mathbb{Y}_0+1): I_0 \to I_0.\]
Then, through the conjugacy $\mathbb{Y}_0+1$, the return map of $\tilde{T}_\ga: I_{-1} \to I_{-1}$ on $\tilde{S}_\ga$ 
corresponds to the return map of $\hat{T}_\ga: I_0 \to I_0$ on $\hat{S}_\ga$.
Because $\hat{S}_\ga = \phi_\ga(S_\ga)$, by the definition of renormalisation, the return map of 
$\hat{T}_\ga: I_0 \to I_0$ on $\hat{S}_\ga$ projects via $w \mapsto e^{2\pi i w}$ to the map $E_\ga$. 
Below, we investigate this return map in more details. 

Let $w\in I_{-1}$ be an arbitrary point, and let $(w_i; l_i)_{i\geq -1}$ denote the trajectory of $w$ defined in 
\refS{SS:T-defn}. 
By the definition of $\tilde{T}_\ga$, if $w_0 \in K_0$ we have 
$\tilde{T}_\ga(w_{-1})= (\mathbb{Y}_0+1)((\mathbb{Y}_0+1)^{-1}(w_{-1})+1)$. 
This implies that for $w\in K_0$, $\hat{T}_\ga (w)=w+1$. 
On the other hand, if $w_0 \in I_0\setminus K_0$, $\hat{T}_\ga(w_0)$ is defined as follows:  
\begin{itemize}
\item[(i)] if there is $n \geq 1$ such that $w_n \in K_n$, and for all $1 \leq i \leq n-1$, $w_i \in I_i \setminus K_i$, then 
\[\hat{T}_\ga (w_0)= \left(\mathbb{Y}_1+\frac{\gep_1+1}{2}\right ) \circ \left(\mathbb{Y}_2+\frac{\gep_2+1}{2}\right ) 
\circ \cdots \circ \left (\mathbb{Y}_n +\frac{\gep_{n}+1}{2} \right)(w_n+1);\] 
\item[(ii)] if for all $i \geq 1$, $w_i \in I_i \setminus K_i$, then 
\[\tilde{T}_\ga(w_0)= \lim_{n \to +\infty} \left(\mathbb{Y}_1+\frac{\gep_1+1}{2}\right ) \circ \left(\mathbb{Y}_2+\frac{\gep_2+1}{2}\right ) 
\circ \cdots \circ \left (\mathbb{Y}_n +\frac{\gep_{n}+1}{2} \right) (w_n+1-1/\ga_n).\]
\end{itemize}
Combining the above together, we conclude that the return map of $\hat{T}_\ga: I_0 \to I_0$ on the set  
$\hat{S}_\ga$ consists of a finite number of translations by $+1$ which take a point in 
$\hat{S}_\ga$ to a point in $I_0 \setminus K_0$, and then one iterate of either the map in 
item (i) or the the map in item (ii), depending on which scenario takes place. 

Since any point in $I_0$ and its integer translations are sent by $w \mapsto s(e^{2\pi i w})$ to the same point, 
each translation by +1 on $K_0$ induce the identity map via the projection $w \mapsto s(e^{2\pi i w})$. 
Thus, by the above paragraph, the rerun map of $\hat{T}_\ga: I_0 \to I_0$ on the set $\hat{S}_\ga$ 
and the map $\hat{T}_\ga: I_0\setminus K_0 \to \hat{S}_\ga$ induce the same map via the 
projection $w \mapsto s(e^{2\pi i w})$. 
On the other hand, by the definition in \refS{SS:T-defn}, the map specified in items (i) and (ii) is 
$\tilde{T}_{\gep_1\ga_1}: I_0/\mathbb{Z} \to I_0/\mathbb{Z}$. 
Therefore, the rerun map of $\hat{T}_\ga: I_0 \to I_0$ on the set $\hat{S}_\ga$, and the map 
$\tilde{T}_{\gep_1\ga_1}: I_0/\mathbb{Z} \to I_0/\mathbb{Z}$ induce the same map via the projection 
$w \mapsto s(e^{2\pi i w})$. 

By the definition in \refS{SS:T-defn}, $\tilde{T}_{\gep_1 \ga_1}: I_0/\mathbb{Z} \to I_0/\mathbb{Z}$ 
induces $\mathbb{T}_{\gep_1 \ga_1}: \mathbb{M}_{\gep_1\ga_1} \to \mathbb{M}_{\gep_1\ga_1}$ 
via the projection $w \mapsto s(e^{2\pi i w})$. 
If we project via $w \mapsto e^{2\pi i w}$, $\tilde{T}_{\gep_1 \ga_1}$ induces 
$s \circ \mathbb{T}_{\gep_1 \ga_1} \circ s$ on $s(\mathbb{M}_{\gep_1\ga_1})$. 
By \refP{P:M-ga-relations-1}, 
$s(\mathbb{M}_{\gep_1 \ga_1})= \mathbb{M}_{-\gep_1\ga_1} = \mathbb{M}_{-1/\ga_0}= \mathbb{M}_{-1/\ga}$ 
and by \refP{P:T-ga-relations}, $s \circ \mathbb{T}_{\gep_1 \ga_1} \circ s= \mathbb{T}_{-\gep_1 \ga_1}
= \mathbb{T}_{-1/\ga}$. 
Note that here we have used $\ga \in (0, 1/2)$, which implies that $\ga_0=\ga$ and 
$1/\ga_0= -\gep_1 \ga_1, \mod \mathbb{Z}$.

Now assume that $\ga \in (-1/2,0)$. 
By the definition of renormalisation for $\ga \in (-1/2, 0)$, we have 
\begin{align*}
\mathcal{R}_m (\mathbb{T}_\ga: \mathbb{M}_\ga \to \mathbb{M}_\ga)
& = s \left ( \mathcal{R}_m (\mathbb{T}_{-\ga}: \mathbb{M}_{-\ga} \to \mathbb{M}_{-\ga})\right) \\
& = s \left ( \mathbb{T}_{1/\ga}: \mathbb{M}_{1/\ga} \to \mathbb{M}_{1/\ga}\right) \\
& = s \circ \mathbb{T}_{1/\ga} \circ : s (\mathbb{M}_{1/\ga}) \to s(\mathbb{M}_{1/\ga}) \\
& = \mathbb{T}_{-1/\ga}: \mathbb{M}_{-1/\ga} \to \mathbb{M}_{-1/\ga}. \qedhere
\end{align*}
\end{proof}

\section{Arithmetic classes of Brjuno and Herman}\label{S:arithmetic}
In this section we define the arithmetic classes of Brjuno and Herman. 
This requires the action of the modular group $\mathrm{PGL}(2, \D{Z})$ on the real line, which produces 
continued fraction type representation of irrational numbers. 
To study the action of this group, one may choose a fundamental interval for the action of $z \mapsto z+1$
and study the action of $z \mapsto 1/z$ on that interval. 
When the interval $(0,1)$ is chosen, one obtains the standard representation (continued fraction). 

Because of the nature of the renormalisation, we work with the fundamental interval $(-1/2,1/2)$ for the translation. 
This is partly due to the symmetry of the renormalisation scheme 
$\{\mathbb{T}_\ga, \mathbb{M}_\ga\}$ 
with respect to the rotation stated in part (v) of \refT{T:toy-model-renormalisation}.
This choice of the fundamental interval leads to a modified representation (continued fraction) of irrationals, 
also known as nearest integer continued fraction.

The modified continued fraction was used by Yoccoz in \cite{Yoc1988,Yoc95} to 
characterise the Brjuno numbers. 
He also studied the relation between the Brjuno functions (see below) in terms of the standard and the 
modified continued fractions. 
A systematic in-depth study of the Brjuno condition, its properties, and its dependence on the choice of the continued 
fraction is carried out by Marmi, Moussa and Yoccoz in \cite{MMY97,MMY01,MMY06}. 
A key point in those papers is that the Brjuno function is a cocycle under the action of $\mathrm{PGL}(2, \D{Z})$.

In \cite{Yoc02}, Yoccoz uses the standard continued fraction to identify the arithmetic class $\E{H}$ 
(see \refD{D:Herman-Yoccoz-criterion}). 
The main aim of this section is to identify the equivalent form of the Herman condition in terms of the modified 
(nearest integer) continued fraction. 
Some of our technical arguments may be found in, or may follow from, \cite{MMY97,MMY01,MMY06}.
We present a quick route to the equivalent form of the Herman condition in terms of the modified 
continued fractions.

We shall only use the modified representation beyond this section. 
For basic properties of continued fractions one may consult \cite{Khin64}. 

\subsection{Standard continued fraction}\label{SS:standard-fraction} 
For $x\in \D{R}$, let $\langle x \rangle \in (0,1)$ denote the fractional part of $x$, that is, $x\in \D{Z}+\langle x \rangle$. 
For $\ga \in \D{R}\setminus \D{Q}$, we may define the numbers $\tilde{\ga}_n \in (0,1)$ as $\tilde{\ga}_0=\langle \ga \rangle$ and 
$\tilde{\ga}_{n+1}=\langle 1/\tilde{\ga}_n \rangle$, for $n\geq 0$.
Then we identify the unique integers $\tilde{a}_n$, for $n\geq -1$, according to 
\begin{equation}\label{E:recursive-factions-standard}
\ga = \tilde{a}_{-1}+ \tilde{\ga}_0, \quad 1/\tilde{\ga}_{n}=\tilde{a}_{n} + \tilde{\ga}_{n+1}.
\end{equation}
These may be combined to obtain  
\begin{equation*}\label{E:modified-expansion-alpha-standard}
\ga= \tilde{a}_{-1}+\cfrac{1}{\tilde{a}_0+\cfrac{1}{\ddots + \cfrac{1}{\tilde{a}_n + \tilde{\ga}_{n+1}}}}, \; \tfor n\geq -1.
\end{equation*}
The $n$-th convergent of $\ga$ is define as 
\[\frac{\tilde{p}_n}{\tilde{q}_n}=\tilde{a}_{-1}+\cfrac{1}{\tilde{a}_0+\cfrac{1}{\ddots + \cfrac{1}{\tilde{a}_n}}}, \;  \tfor n\geq -1.\] 
Then,  
\begin{equation*}\label{E:ga-ga_n-formula}
\ga=(\tilde{p}_n+\tilde{p}_{n-1}\tilde{\ga}_{n+1})/(\tilde{q}_n+ \tilde{q}_{n-1}\tilde{\ga}_{n+1}), \; \tfor n\geq -1,
\end{equation*}
which implies 
\begin{equation}\label{E:ga_n-ga-formula}
\ga_{n+1}=-(\ga \tilde{q}_n- \tilde{p}_n)/(\ga \tilde{q}_{n-1}-\tilde{p}_{n-1}), \; \tfor n\geq -1.
\end{equation}
In \cite{Brj71}, see also \cite{Cherry64}, Brjuno introduced the important series  $\sum_{n=-1}^{+\infty} \tilde{q}_n^{-1} \log \tilde{q}_{n+1}$.
In \cite{Yoc95}, Yoccoz defines a closely related series which enjoys remarkable equivariant properties 
with respect to the action of $\mathrm{PGL}(2, \D{Z})$.  
To define that, we need to introduce the numbers
\[\tilde{\gb}_{-2}=\ga, \quad \tilde{\gb}_{-1}=+1, \quad  \tilde{\gb}_n= \prod_{i=0}^n \tilde{\ga}_i, \; \tfor n \geq 0.\] 
In terms of the convergents, \refE{E:ga_n-ga-formula} gives us 
\begin{equation}\label{E:beta_n-p_n-q_n}
\tilde{\gb}_n= (-1)^n (\ga \tilde{q}_{n-1} - \tilde{p}_{n-1}), \;  \tfor n\geq -1.
\end{equation}
Define the (standard Brjuno) function $\tilde{\C{B}}: \D{R}\setminus \D{Q} \to (0, +\infty) \cup \{+\infty\}$ as 
\begin{equation}\label{E:Brjuno-condition-standard}
\tilde{\C{B}}(\ga)= \sum_{n=-1}^{+\infty} \tilde{\gb}_{n} \log \frac{1}{\tilde{\ga}_{n+1}}
=  \sum_{n=-1}^{+\infty}   (-1)^n (\ga \tilde{q}_{n-1} -\tilde{p}_{n-1}) \log \frac{\ga \tilde{q}_{n-1}-\tilde{p}_{n-1}}{\tilde{p}_n-\ga \tilde{q}_n}.
\end{equation}
This is a highly irregular function; $\tilde{\C{B}}(\ga)=+\infty$ for generic choice of $\ga \in \D{R}$. 
One may refer to \cite{MMY97,MMY01,JM18}, and the extensive list of references therein, 
for detailed analysis of the regularity properties of this function. 
In this paper we are not concerned with the regularity, but only exploit the equivariant properties of the Brjuno 
function with respect to the action of $PGL(2,\mathbb{Z})$.

\begin{def-numbered}\label{D:brjuno-numbers-beta-version}
An irrational number $\ga$ is called a \textbf{Brjuno number} if $\tilde{\C{B}}(\ga)<+\infty$.  
\end{def-numbered}

The function $\tilde{\C{B}}$ satisfies the remarkable relations 
\begin{equation}\label{E:Brjuno-functional-equations-standard}
\tilde{\C{B}}(\ga)= \tilde{\C{B}}(\ga+n), \; 
\tilde{\C{B}}(\ga)= \ga \tilde{\C{B}}(1/\ga) + \log (1/\ga), 
\end{equation}
for all $\ga \in (0,1)$ and all $n\in \D{Z}$. These show that the set of Brjuno numbers is 
$\mathrm{PGL}(2, \D{Z})$-invariant. 

Let $\E{B}$ denote the set of Brjuno numbers. 
The definition of the Brjuno numbers given in the introduction is consistent with the one given in 
\refD{D:brjuno-numbers-beta-version}. 
That is because, $| \sum_{n=-1}^{+\infty} \tilde{q}_n^{-1} \log \tilde{q}_{n+1} - \tilde{\C{B}}(\ga) |$ is uniformly 
bounded from above.

The set of Herman numbers is defined in a different fashion. 
To that end, we need to consider the diffeomorphisms  $h_r: \D{R} \to (0, +\infty)$, for $r\in (0,1)$:
\[h_r(y)= 
\begin{cases}
r^{-1} (y- \log r^{-1} +1)  & \tif y \geq \log r^{-1}, \\
e^{y} & \tif y \leq \log r^{-1}.
\end{cases}
 \]
Each $h_r$ satisfies
\begin{equation}\label{E:h_r-properties}
\begin{gathered}
h_r (\log r^{-1}) = h_r'(\log r^{-1}) = r^{-1},  \\
e^y \geq h_r(y) \geq y+1, \forall y\in \D{R}, \\
h_r'(y)\geq 1, \forall y\geq 0.
\end{gathered}
\end{equation} 

Following Yoccoz \cite{Yoc02}, we give the following definition.

\begin{def-numbered}\label{D:Herman-Yoccoz-criterion}
An irrational number $\ga$ is of \textbf{Herman type}, if for all $n\geq 0$ there is $m\geq n$ such that 
\[h_{\tilde{\ga}_{m-1}} \circ \dots \circ h_{\tilde{\ga}_n} (0)\geq \tilde{\C{B}}(\tilde{\ga}_{m}).\]
\end{def-numbered}
 
In the above definition, the composition $h_{\tilde{\ga}_{m-1}} \circ \dots \circ h_{\tilde{\ga}_n}$ is understood as the
identity map when $m=n$, and as $h_{\tilde{\ga}_n}$ when $m=n+1$. 
The set of Herman numbers is denoted by $\E{H}$. 
It follows from the definition that\footnote{In this paper we do not make a distinction between 
``$\subset$'' and ``$\subseteq$''. If strict inclusion is meant, we use ``$\subsetneq$''.}
\[\E{H} \subset \E{B}.\] 
That is because, if $\ga \notin \E{B}$, then $\tilde{\C{B}}(\tilde{\ga})=\tilde{\C{B}}(\tilde{\ga}_{0})= +\infty$. 
Repeatedly using the functional equations in \refE{E:Brjuno-functional-equations-standard}, one concludes that 
$\tilde{\C{B}}(\tilde{\ga}_{m})=+\infty$ for all $m\geq 0$. 
In particular, the inequality in the above definition never holds. 


In \refD{D:Herman-Yoccoz-criterion}, one may only require that for large $n$ there is $m$ such that the inequality 
holds. That is because, if $m'$ works for some $n'$, then the same $m'$ works for all $n \leq n'$. 
This shows that $\tilde{\ga}_0 \in \E{H}$ if and only if $\tilde{\ga}_1 \in \E{H}$. 
On the other hand, since $\ga$ and $\ga+1$ produce the same sequence of $\tilde{\ga}_i$, we see that $\E{H}$ is 
invariant under $z\mapsto z+1$. These show that  $\E{H}$ is invariant under the action of 
$\mathrm{PGL}(2,\D{Z})$. 

Recall that $\ga$ is a Diophantine number, if there are $\gt \geq 0$ and $c>0$ such that 
for all $p/q \in \D{Q}$ with $q\geq 1$ we have $|\ga-p/q|\geq c/q^{2+\gt}$. 
Any Diophantine number is of Herman type. Since the set of Diophantine numbers has full Lebesgue measure in 
$\D{R}$, the sets $\E{H}$ and $\E{B}$ have full Lebesgue measure in $\D{R}$. 

\begin{lem}
The set $\E{B} \setminus \E{H}$ is dense in $\D{R}$. 
\end{lem}

\begin{proof}
Let $\ga$ be an irrational number such that there is an integer $k \geq 0$ such that 
$\tilde{a}_{k}=1$, and for all $i\geq k$ we have $e^{\tilde{a}_i} \leq  \tilde{a}_{i+1} \leq e^{2 \tilde{a}_i}-1$. 
Evidently, the set of such irrational numbers is dense in $\D{R}$. Below we show that any such $\ga$ belongs to 
$\E{B} \setminus \E{H}$. 
 
For integers $i \geq k$,
\[\tilde{\gb}_i \log (1/\tilde{\ga}_{i+1}) 
\leq \tilde{\gb}_i \log (\tilde{a}_{i+1}+1) 
\leq  \tilde{\gb}_i 2 \tilde{a}_{i} 
\leq 2 \tilde{\gb}_{i-1}.\]
Then, 
\begin{align*}
\tilde{\C{B}}(\ga)
&= \textstyle{ \sum_{i=-1}^{k-1}\tilde{\gb}_i\log (1/\tilde{\ga}_{i+1}) + \sum_{i=k}^{+\infty}\tilde{\gb}_i\log(1/\tilde{\ga}_{i+1}) } \\
& \leq \textstyle{ \sum_{i=-1}^{k-1}\tilde{\gb}_i\log (1/\tilde{\ga}_{i+1}) + 2 \sum_{i=k}^{+\infty} \tilde{\gb}_{i-1} }
\leq \textstyle{ \sum_{i=-1}^{k-1}\tilde{\gb}_i \log (1/\tilde{\ga}_{i+1}) + 8.} 
\end{align*}
This proves that $\ga$ belongs to $\E{B}$. 

On the other hand, for all $i \geq k+1$, 
\[\tilde{a}_{i-1} \leq \log \tilde{a}_{i} \leq \log (1/\tilde{\ga}_i) \leq \tilde{\C{B}}(\tilde{\ga}_i).\]
Then, by an inductive argument, for all 
$i \geq k+1$, $h_{\tilde{\ga}_{i-1}} \circ \dots \circ h_{\tilde{\ga}_k}(0) \leq \tilde{a}_{i-1}$.
Therefore, for all integers $m\geq n \geq k+1$, 
\[h_{\tilde{\ga}_{m-1}} \circ \dots \circ h_{\tilde{\ga}_n}(0) 
\leq h_{\tilde{\ga}_{m-1}} \circ \dots \circ h_{\tilde{\ga}_n} (h_{\tilde{\ga}_{n-1}} \circ \dots \circ h_{\tilde{\ga}_k}(0))  
\leq \tilde{a}_{m-1} \leq \tilde{\C{B}}(\tilde{\ga}_m).\]
This shows that $\ga \notin \E{H}$. 
\end{proof}

By classical results of Siegel \cite{Sie42} and Brjuno \cite{Brj71}, if $\ga \in \E{B}$, then every germ of 
holomorphic map $f(z)=e^{2\pi i \ga} z+ O(z^2)$ is locally conformally conjugate to the rotation by $2\pi\ga$
near $0$. 
On the other hand, Yoccoz in \cite{Yoc1988,Yoc95} proved that this condition is optimal in the quadratic family 
$e^{2\pi i \ga} z+ z^2$, i.e.\ if $\ga \notin \E{B}$ then $e^{2\pi i \ga} z+ z^2$ is not linearisable near $0$. 
His approach is geometric, and avoids formidable calculations one encounters in the study of 
small-divisors. 
The optimality of this condition has been (re)confirmed for several classes of maps 
\cite{PM93,Ge01,BC04,Ok04,Ok2005,FMS2018,Che19}, but in its general form for rational functions remains a 
significant challenge in the field of holomorphic dynamics. 
In this paper we do not rely on the optimality of this condition in any class of maps. 

In \cite{Her79}, Herman carried out a comprehensive study of the problem of linearisation of orientation-preserving 
diffeomorphisms of the circle $\D{R}/\D{Z}$ with irrational rotation number. 
In particular, he presented a rather technical arithmetic condition which guaranteed the linearisation of such analytic 
diffeomorphisms. 
Although the linearisation problem for analytic circle diffeomorphisms close to rotations was successfully studied 
earlier by Arnold \cite{Ar61}, no progress had been made in between. 
Enhancing the work of Herman, Yoccoz identified the optimal arithmetic condition $\E{H}$ for the analytic 
linearisation of analytic diffeomorphisms of the circle, \cite{Yoc95-ICM,Yoc02}.
The name, Herman numbers, for the class $\E{H}$ was suggested by Yoccoz in honour of the work of Herman on this problem.
Similarly, in this paper we do not use this form of the optimality of the condition $\E{H}$.

\subsection{Modified continued fraction, and the equivalent form of Herman numbers}
\label{SS:modified-fractions}
Let us recall the modified continued fraction algorithm we mentioned in \refS{SS:modified-fractions-mini}. 
For $x\in \D{R}$, define $d(x, \D{Z})= \min_{k\in \D{Z}} |x-k|$. Let us fix an irrational number $\ga \in \D{R}$. 
Define the numbers $\ga_n\in (0,1/2)$, for $n\geq 0$, according to 
\begin{equation}
\ga_0=d(\ga, \D{Z}), \quad  \ga_{n+1}=d(1/\ga_n, \D{Z}),
\end{equation} 
Then, there are unique integers $a_n$, for $n\geq -1$, and $\gep_n \in \{+1, -1\}$, for $n\geq 0$, such that 
\begin{equation}
\ga= a_{-1}+ \gep_0 \ga_0, \quad  1/\ga_n= a_n + \gep_{n+1} \ga_{n+1}. 
\end{equation}
Evidently, for all $n\geq 0$,
\begin{equation}
1/\ga_n \in (a_n-1/2, a_n+1/2), \quad a_n\geq 2,
\end{equation}
and 
\begin{equation}
\gep_{n+1}= 
\begin{cases}
+1 & \text{if } 1/\ga_n \in (a_n, a_n+1/2), \\
-1 & \text{if } 1/\ga_n \in (a_n-1/2, a_n).
\end{cases}
\end{equation}
We also defined $\ga_{-1}=+1$. 

The sequences $\{a_n\}$ and $\{\gep_n\}$ provide us with the infinite continued fraction 
\begin{equation*}\label{E:modified-expansion-alpha}
\ga=a_{-1}+\cfrac{\gep_0}{a_0+\cfrac{\gep_1}{a_1+\cfrac{\gep_2}{a_2+\dots}}}.
\end{equation*}
Consider the numbers 
\[\gb_{-2}=\ga, \; \gb_{-1}=+1, \; \gb_n= \gb_n(\ga)= \textstyle{ \prod_{i=0}^n \ga_i, \; \tfor n \geq 0}.\] 
In \cite{Yoc95}, Yoccoz defines the arithmetic series 
\begin{equation}\label{E:Brjuno-condition}
\textstyle{ \C{B}(\ga)= \sum_{n=0}^\infty \gb_{n-1} \log \ga_n^{-1},}
\end{equation}
and calls it the Brjuno function. This function is defined on the set of irrational numbers, and takes values in 
$(0, +\infty]$. One may extend $\C{B}$ onto $\D{Q}$, by setting $\C{B}(p/q)=+\infty$, for all $p/q\in \D{Q}$. 

The Brjuno function satisfies the remarkable relations 
\begin{equation}\label{E:Brjuno-functional-equations}
\begin{gathered}
\C{B}(\ga)= \C{B}(\ga+1)= \C{B}(-\ga), \; \tfor \ga \in \D{R}, \\
\C{B}(\ga)= \ga \C{B}(1/\ga)+ \log (1/\ga), \; \tfor \ga \in (0,1/2).
\end{gathered}
\end{equation}
These show that one may think of the Brjuno function as a $\mathrm{PGL}(2,\D{Z})$-cocycle. 
This point of view drives some of the technical arguments we present later in the paper, 
notably in \refS{SS:Herman-tower}. 
See \cite{MMY97, MMY01} for a systematic approach to employing this mechanism.

One may formally define the arithmetic classes of $\E{B}$ and $\E{H}$ using the modified continued fraction and the 
modified function $\C{B}$ in the same fashion. The following two propositions guarantee that through this 
we identify the same classes of irrational numbers.  

\begin{propo}\label{P:brjuno-standard-vs-modified}
For all $\ga \in \D{R}\setminus \D{Q}$ we have $|\C{B}(\ga)- \tilde{\C{B}}(\ga)| \leq 29$.
In particular, $\ga$ is a Brjuno number if and only if $\C{B}(\ga)< +\infty$. 
\end{propo}

In \cite{MMY97}, the authors go very far in this direction, by showing the remarkable property that 
the difference $\C{B}(\ga)- \tilde{\C{B}}(\ga)$ extends to a $1/2$-holder continuous function over all of $\mathbb{R}$.

\begin{propo}\label{P:Herman-Yoccoz-criterion}
An irrational number $\ga$ is a Herman number if and only if for all $n\geq 0$ there is $m\geq n$ such that 
\[h_{\ga_{m-1}} \circ \dots \circ h_{\ga_n} (0)\geq \C{B}(\ga_{m}).\]
\end{propo}
Although the criterion in the above proposition appears identical to the one given in 
\refD{D:Herman-Yoccoz-criterion}, the value of $m$ for a given $n$ may be different. 
As before, here $h_{\ga_{m-1}} \circ \dots \circ h_{\ga_n}$ is understood as the identity map when $m=n$, 
and as $h_{\ga_n}$ when $m=n+1$.

The remaining of this section is devoted to the proof of the above two propositions. 
The main reason here is that the sequences $\{\tilde{\ga}_n\}$ and $\{\ga_n\}$ are closely related, and it is possible to 
identify one from the other using an algorithm. 

Let us define the sequence 
\[ \textstyle{ c(-1)=-1,\;  c(n)=-1+ \sum_{i=0}^{n} (3-\gep_i)/2,\; \tfor n\geq 0}.\]
That is, $c(n)$ is obtained from $c(n-1)$ by adding $+2$ if $\gep_n=-1$ or adding $+1$ if $\gep_n=+1$.  
Clearly, $n \mapsto c(n)$ is strictly monotone with $c(n) \to +\infty$ as $n \to +\infty$. 
For more general properties of $c(n)$, and its dependence on the choice of the continued fraction, one may refer to \cite{MMY97} (where $c(n)$ is denoted as $k^{1/2}(n)$).

\begin{lem}\label{L:fractions-related}
For all $n\geq -1$, the following hold: 
\begin{itemize}
\item[(i)] if $\gep_{n+1}=-1$, $\ga_{n+1}= 1- \tilde{\ga}_{c(n)+1}$, 
and if $\gep_{n+1}=+1$, $\ga_{n+1}=\tilde{\ga}_{c(n)+1}$;  
\item[(ii)] $\ga_{n+1}= \prod_{c(n)+1}^{c(n+1)} \tilde{\ga}_i$;
\item[(iii)] $\tilde{\gb}_{c(n)}=\gb_n$;
\item[(iv)] $\C{B}(\tilde{\ga}_{c(n)+1})= \C{B}(\ga_{n+1})$.
\end{itemize}
\end{lem}

\begin{proof}
We prove (i) by induction on $n$. 
We start with $n=-1$. If $\gep_0=+1$, $\ga_0=\tilde{\ga}_0= \tilde{\ga}_{c(-1)+1}$. 
If $\gep_0=-1$, $\ga_0 =1-\tilde{\ga}_0=1-\tilde{\ga}_{c(-1)+1}$.  
Now assume that the assertion in (i) is true for $n-1$. To prove it for $n$, we consider two cases:  

First assume that $\gep_n=+1$. By the induction hypothesis for $n-1$, $\ga_n= \tilde{\ga}_{c(n-1)+1}=\tilde{\ga}_{c(n)}$. 
Hence, $1/\ga_n=1/\tilde{\ga}_{c(n)}$. 
Now, if $\gep_{n+1}=+1$, $1/\ga_n=1/\tilde{\ga}_{c(n)}$ leads to $\ga_{n+1}=\tilde{\ga}_{c(n)+1}$. 
If $\gep_{n+1}=-1$, $1/\ga_n=1/\tilde{\ga}_{c(n)}$ leads to $\ga_{n+1}=1-\tilde{\ga}_{c(n)+1}$.

Now assume that $\gep_n=-1$. By the induction hypothesis, $\ga_n=1-\tilde{\ga}_{c(n-1)+1}=1-\tilde{\ga}_{c(n)-1}$. 
As $\ga_n \in (0,1/2)$, $\tilde{\ga}_{c(n)-1} \in (1/2,1)$. 
Then, $\tilde{\ga}_{c(n)}= 1/\tilde{\ga}_{c(n)-1}-1= 1/(1-\ga_n)-1= \ga_n /(1-\ga_n)$. 
Hence, $1/\tilde{\ga}_{c(n)}= 1/\ga_n-1$. 
Now, if $\gep_{n+1}=+1$, $1/\tilde{\ga}_{c(n)}= 1/\ga_n-1$ leads to $\tilde{\ga}_{c(n)+1}= \ga_{n+1}$. 
If $\gep_{n+1}=-1$, $1/\tilde{\ga}_{c(n)}= 1/\ga_n-1$ leads to $1-\tilde{\ga}_{c(n)+1}=\ga_{n+1}$. 

Part (ii): If $\gep_{n+1}=+1$, by Part (i), $\ga_{n+1}= \tilde{\ga}_{c(n)+1}= \tilde{\ga}_{c(n+1)}$. 
If $\gep_{n+1}=-1$, by Part (i), $\ga_{n+1}=1-\tilde{\ga}_{c(n)+1}$. As 
$\ga_{n+1}\in (0,1/2)$ we conclude that $\tilde{\ga}_{c(n)+1}\in (1/2,1)$, which implies that 
$\tilde{\ga}_{c(n)+2}= 1/\tilde{\ga}_{c(n)+1}-1$. 
Therefore, 
\[\tilde{\ga}_{c(n+1)} \tilde{\ga}_{c(n)+1}=\tilde{\ga}_{c(n)+2} \tilde{\ga}_{c(n)+1}=1-\tilde{\ga}_{c(n)+1}= \ga_{n+1}.\] 

Part (iii): By the formula in Part (ii), and the definition of $c(n)$, 
\[\textstyle{ \gb_n=\prod_{m=0}^n \ga_m =\prod_{m=0}^n \Big (\prod_{i=c(m-1)+1}^{c(m)} \tilde{\ga}_i\Big) 
=\prod_{i=0}^{c(n)} \tilde{\ga}_i= \tilde{\gb}_{c(n)}. }\] 

Part (iv): If $\gep_{n+1}=+1$, by Part (i), $\ga_{n+1}= \tilde{\ga}_{c(n)+1}$, and hence 
$\C{B}(\tilde{\ga}_{c(n)+1})= \C{B}(\ga_{n+1})$.  
If $\gep_{n+1}=-1$, by Part (i), $\ga_{n+1}= 1-\tilde{\ga}_{c(n)+1}$. Using \refE{E:Brjuno-functional-equations}, 
\[\C{B}(\ga_{n+1})=\C{B}(1-\tilde{\ga}_{c(n)+1})=\C{B}(-\tilde{\ga}_{c(n)+1})= \C{B}(\tilde{\ga}_{c(n)+1}). \qedhere\]
\end{proof}

\begin{proof}[Proof of \refP{P:brjuno-standard-vs-modified}]
Fix $n \geq 0$. 
If $\gep_n=+1$ then $c(n)=c(n-1)+1$, and by \refL{L:fractions-related}-(iii), 
\begin{equation}\label{E:P:brjuno-standard-vs-modified-1}
\gb_{n-1}\log \gb_n = \tilde{\gb}_{c(n-1)}\log \tilde{\gb}_{c(n-1)+1}= \tilde{\gb}_{c(n)-1}\log \tilde{\gb}_{c(n)}.
\end{equation}
If $\gep_n=-1$ then $c(n)= c(n-1)+2$, and by \refL{L:fractions-related}-(iii), 
\[\tilde{\gb}_{c(n)-1}
=\tilde{\gb}_{c(n-1)+1} = \tilde{\gb}_{c(n-1)} \tilde{\ga}_{c(n-1)+1} = \gb_{n-1} (1-\ga_n)= \gb_{n-1}- \gb_{n}.\] 
and therefore 
\begin{equation}\label{E:P:brjuno-standard-vs-modified-2}
\begin{aligned}
\big(\tilde{\gb}_{c(n-1)}\log\tilde{\gb}_{c(n-1)+1}&+\tilde{\gb}_{c(n)-1}\log\tilde{\gb}_{c(n)} \big) -\gb_{n-1}\log\gb_n\\
&=\big( \gb_{n-1}\log (\gb_{n-1}- \gb_n) + (\gb_{n-1} - \gb_n) \log \gb_{n}\big) - \gb_{n-1}\log\gb_n\\
&= \gb_{n-1}\log (\gb_{n-1}- \gb_n) - \gb_n \log \gb_{n}
\end{aligned}
\end{equation}
Combing \eqref{E:P:brjuno-standard-vs-modified-1} and \eqref{E:P:brjuno-standard-vs-modified-2}, 
and using $\gb_{n-1}- \gb_n = \gb_{n-1}(1-\ga_n)$, we conclude that for all $m\geq 0$ we have 
\begin{align*}
\textstyle{ \sum_{i=0}^{c(m)} \tilde{\gb}_{i-1} \log \tilde{\gb}_i}  & - \textstyle{\sum_{n=0}^m \gb_{n-1} \log \gb_n} \\
&=\textstyle{  \sum_{n=0}^m \left (\sum_{i=c(n-1)+1}^{c(n)} (\tilde{\gb}_{i-1} \log \tilde{\gb}_i) -  \gb_{n-1}\log \gb_n \right) }\\
&= \textstyle{ \sum_{n=0\, ;\, \gep_n=-1}^m  \left (\gb_{n-1}\log \gb_{n-1} + \gb_{n-1} \log (1-\ga_n) - \gb_n \log \gb_n\right).}  
\end{align*} 
On the other hand, 
\begin{equation*}
\textstyle{  \sum_{n=0}^m \gb_{n-1} \log (1/\ga_n) + \sum_{n=0}^m \gb_{n-1} \log \gb_n
= \sum_{n=0}^m \gb_{n-1} \log \gb_{n-1}, }
\end{equation*}
and similarly, 
\[\textstyle{  - \sum_{i=0}^{c(m)} \tilde{\gb}_{i-1} \log \tilde{\gb}_i  - \sum_{i=0}^{c(m)} \tilde{\gb}_{i-1} \log (1/\tilde{\ga}_i)
= - \sum_{i=0}^{c(m)} \tilde{\gb}_{i-1} \log \tilde{\gb}_{i-1}. }\]
Recall that $\tilde{\gb}_{-1}=\gb_{-1}=1$. Adding the above three equations, we conclude that 
\begin{equation*}
\begin{aligned}
\textstyle{ \left |\sum_{n=0}^m \gb_{n-1} \log (1/\ga_n) \right. } &  \textstyle{ \left .-  \sum_{i=0}^{c(m)} \tilde{\gb}_{i-1} \log (1/\tilde{\ga}_i) \right |} \\
& \textstyle{ \leq 3 \sum_{n=0}^{+\infty} |\gb_{n} \log \gb_{n}| + \sum_{n=0}^{+\infty} |\gb_{n-1} \log (1-\ga_n)| }
+ \textstyle{ \sum_{i=0}^{+\infty} |\tilde{\gb}_{i} \log \tilde{\gb}_{i}|} 
\end{aligned}
\end{equation*}
Since $|x \log x| \leq 2 \sqrt{x}$ for $x\in (0,1)$, 
\begin{equation}\label{E:P:Brjuno-Yoccoz-equivalent}
\begin{aligned}
\textstyle{ \sum_{n=0}^{+\infty} |\tilde{\gb}_{n}\log \tilde{\gb}_{n}| }
\textstyle{ \leq  2 \sum_{n=0}^{+\infty} (\tilde{\gb}_{n})^{1/2} }
& \textstyle{ \leq 2 \sum_{n=0}^{+\infty} (\tilde{\gb}_{2n})^{1/2} + 2 \sum_{n=0}^{+\infty} (\tilde{\gb}_{2n+1})^{1/2} } \\ 
& \textstyle{ \leq 2 (\tilde{\gb}_0)^{1/2} \sum_{n=0}^{+\infty} 2^{-n/2} + 2 \sum_{n=1}^{+\infty} 2^{-n/2} } \\
& \leq 6+ 4 \cdot 2^{1/2}. 
\end{aligned}
\end{equation}
On the other hand, $\gb_n\leq 2^{-n-1}$, for all $n\geq 0$. 
Using $|x \log x| \geq 2 \sqrt{x}$, for $x \in (0,1)$, $\ga_j \in (0, 1/2)$, we obtain 
\begin{equation*}
\textstyle{ \sum_{n=0}^k |\gb_{n} \log \gb_{n}| \leq 2 \sum_{n=0}^{+\infty} \sqrt{\gb_{n}} \leq
2 \sum_{n=0}^{+\infty} 1/2^{(n+1)/2} = 2+ 2 \sqrt{2}, }
\end{equation*}
and
\[ \textstyle{\sum_{n=0}^{+\infty} |\gb_{n-1} \log (1-\ga_n)| \leq \log 2 \sum_{n=0}^{+\infty} \gb_{n-1} = 2 \log 2. }\]
This completes the proof of the proposition. 
\end{proof} 

\begin{lem}\label{L:h-vs-cocycle}
We have, 
\begin{itemize}
\item[(i)] for all $r \in (0,1)$,
\[h_r (\C{B}(r)) \geq \C{B}(1/r)+1, \quad h_r (\tilde{\C{B}}(r)) \geq \tilde{\C{B}}(1/r)+1.\]
\item[(ii)] if there are $m\geq n \geq 0$ satisfying 
$h_{\tilde{\ga}_{m-1}} \circ \dots \circ h_{\tilde{\ga}_n} (0)\geq \tilde{\C{B}}(\tilde{\ga}_{m})$, 
then 
\[\lim_{m \to +\infty} h_{\tilde{\ga}_{m-1}} \circ \dots \circ h_{\tilde{\ga}_n} (0) - \tilde{\C{B}}(\tilde{\ga}_{m})=+\infty.\]
\item[(iii)] if there are $m\geq n \geq 0$ satisfying 
$h_{\ga_{m-1}} \circ \dots \circ h_{\ga_n} (0)\geq \C{B}(\ga_{m})$, then
\[\lim_{m\to +\infty} h_{\ga_{m-1}} \circ \dots \circ h_{\ga_n} (0) - \C{B}(\ga_{m})=+\infty.\]
\end{itemize} 
\end{lem}

\begin{proof}
Note that for all $r\in (0,1)$ and all $y \in \D{R}$, 
\[h_r(y) \geq r^{-1} y + r^{-1} \log r+ 1.\]
If $y\geq \log r^{-1}$, $h_r(y)=r^{-1}y + r^{-1} \log r + r^{-1} \geq r^{-1}y + r^{-1} \log r +1$. 
Using the inequality $x \geq 1+ \log x$, for $x>0$, we note that 
$r e^y \geq 1+ \log (re^y) = y + \log r +1 \geq y + \log r + r$. 
This implies the above inequality for $y < \log r^{-1}$. 

By the above inequality, as well as \eqref{E:Brjuno-functional-equations-standard} and \eqref{E:Brjuno-functional-equations}, we obtain 
\[h_r (\C{B}(r)) \geq r^{-1} \C{B}(r)+ \log r +1 = \C{B}(1/r)+1, \quad 
h_r (\tilde{\C{B}}(r)) \geq r^{-1} \tilde{\C{B}}(r)+ \log r +1 = \tilde{\C{B}}(1/r)+1.\]

If the inequalities in (ii) and (iii) hold, we may use the inequality in (i) and $h_r(y+1)\geq h_r(y)+1$ in 
\refE{E:h_r-properties}, to obtain
\[h_{\tilde{\ga}_{m+j-1}} \circ \dots \circ h_{\tilde{\ga}_n} (0)\geq \tilde{\C{B}}(\tilde{\ga}_{m+j})+j, \quad 
h_{\ga_{m+j-1}} \circ \dots \circ h_{\ga_n} (0)\geq \C{B}(\ga_{m+j})+j.\qedhere\] 
\end{proof}

\begin{lem}\label{L:herman-blocks}
Let $r_1 \in (1/2,1)$, $r_2=1/r_1-1\in (0,1)$, and $r=r_1 r_2 \in (0, 1/2)$. Then, for all $y \geq e^2$ we have 
\[\big| h^{-1}_r(y)-  h^{-1}_{r_1} \circ h^{-1}_{r_2}(y)\big|  \leq 1+ e^{-1}.\]
\end{lem}

\begin{proof}
The inverse map $h_r^{-1}: (0, +\infty) \to \D{R}$ is given by the formula 
\[h_r^{-1}(y)=   
\begin{cases}
r y + \log r^{-1} -1 & \tif  y \geq 1/r,  \\
\log y  & \tif 0 < y \leq 1/r.
\end{cases}\]
Since $y\geq e^2$, one can see that $h_{r_2}^{-1}(y)\geq  1+ r_2 = 1/r_1$. Thus, 
$h^{-1}_{r_1} \circ h^{-1}_{r_2}(y)= r_1 h_{r_2}^{-1}(y)+ \log r_1^{-1}-1$. 
Using $r_2=1/r_1-1$ and the elementary inequality $|x \log x| \leq 1/e$, for $x\in (0,1)$, we have 
\[(1-r_1) \log (1/r_2) = (r_2/(1+ r_2)) \log r^{-1}_2 \leq e^{-1}/(1+r_2) \leq e^{-1}.\] 
We consider three cases:

\noindent (1) $y \leq 1/r_2$: 
Since $1/r_2 \leq 1/r$, we get 
\begin{align*}
\big| h^{-1}_r(y)-  h^{-1}_{r_1} \circ h^{-1}_{r_2}(y)\big| 
&= \big |  \log y  - (r_1 \log y + \log r_1^{-1} -1)\big| \\
& \leq (1-r_1) \log y + \big |1- \log r^{-1}_1 \big| \\
& \leq (1-r_1) \log r_2^{-1} + \big |1- \log r^{-1}_1 \big| \leq e^{-1}+ 1. 
\end{align*}

\noindent (2) $1/r_2 \leq y \leq 1/r$: 
Then, 
\begin{align*}
h^{-1}_r(y)-  h^{-1}_{r_1} \circ h^{-1}_{r_2}(y)
& = \log y  - (r_1 (r_2 y+\log r_2^{-1}-1) + \log r_1^{-1} -1) \\
&= \log (yr_1) - r_1 \log r_2^{-1} + r_1 (1-r_2 y) + 1. 
\end{align*}
On the other hand, we have 
\[-1/2 \leq r_1 -1 \leq r_1 - ry = r_1 - r_1 r_2 y = r_1(1-r_2y) \leq 0,\] 
and, using $y r_1 r_2= y r \leq 1$, we get 
\[\log yr_1- r_1 \log r_2^{-1} \leq  \log r^{-1}_2 - r_1 \log r_2^{-1} = (1-r_1) \log r_2^{-1} \leq e^{-1},\]
and, using $y\geq 1/r_2$,  
\[\log yr_1- r_1 \log r_2^{-1} \geq \log (r_1/r_2) - r_1 \log r_2^{-1} = \log r_1 + (1-r_1) \log r_2 ^{-1} 
\geq \log r_1 \geq - \log 2.\]
Combining the above inequalities we get 
\[-1 - e^{-1} \leq 1-1/2 -\log 2  \leq h^{-1}_r(y)-  h^{-1}_{r_1} \circ h^{-1}_{r_2}(y) \leq 1+ e^{-1}.\]

\noindent (3) $y \geq 1/r$: 
Using $r_1 r_2=r$, we get 
\begin{align*}
\big| h^{-1}_r(y)-  h^{-1}_{r_1} \circ h^{-1}_{r_2}(y)\big| 
& = \big | r y + \log r^{-1} -1  - (r_1 (r_2 y+\log r_2^{-1}-1) + \log r_1^{-1} -1)\big| \\
&= (1-r_1) \log r_2^{-1} + r_1 \leq e^{-1}+ 1. 
\end{align*}
This completes the proof of the lemma. 
\end{proof} 

\begin{lem}\label{L:herman-chains}
Let $m > n \geq 0$ and $y \geq e^2$. Assume that at least one of the following holds: 
\begin{itemize}
\item[(i)] $h^{-1}_{\ga_n} \circ h^{-1}_{\ga_{n+1}} \circ \dots \circ h^{-1}_{\ga_m}(y)$ is defined and is at least $e^2$, 
\item[(ii)]
$h^{-1}_{\tilde{\ga}_{c(n-1)+1}} \circ h^{-1}_{\tilde{\ga}_{c(n-1)+2}} \circ \dots \circ h^{-1}_{\tilde{\ga}_{c(m)}}(y)$ 
is defined and is at least $e^2$. 
\end{itemize}
Then, 
\[\big |h^{-1}_{\ga_n} \circ h^{-1}_{\ga_{n+1}} \circ \dots \circ h^{-1}_{\ga_m}(y) 
- h^{-1}_{\tilde{\ga}_{c(n-1)+1}} \circ h^{-1}_{\tilde{\ga}_{c(n-1)+2}} \circ \dots \circ h^{-1}_{\tilde{\ga}_{c(m)}}(y)
\big |\leq 2(1+e^{-1}).\] 
\end{lem}
In the above lemma, it is part of the conclusion that if one of the compositions in items (i) and (ii) is defined and 
is at least $e^2$, then the composition in the other item is defined as well. 

\begin{proof}
First assume that item (i) holds. 
By \refE{E:h_r-properties}, $h_r(t) \geq t$ for all $t \geq 0$ and $r\in (0,1)$. 
This implies that for all $j$ with $n \leq j \leq m$, 
$h^{-1}_{\ga_j} \circ h^{-1}_{\ga_{j+1}} \circ \dots \circ h^{-1}_{\ga_m}(y)\geq e^2$. 
Let 
\[y_j=h^{-1}_{\ga_j} \circ h^{-1}_{\ga_{j+1}} \circ \dots \circ h^{-1}_{\ga_m}(y), \; \tfor n \leq j \leq m.\] 
By an inductive argument, we show that for all $j$ with $n \leq j \leq m$, 
\[\tilde{y}_j
=h^{-1}_{\tilde{\ga}_{c(j-1)+1}} \circ h^{-1}_{\tilde{\ga}_{c(n-1)+2}} \circ \dots \circ h^{-1}_{\tilde{\ga}_{c(m)}}(y),\]
 is defined, and satisfies 
\[|y_j -\tilde{y}_j| \leq 2 (1+ e^{-1}).\]

We start with $j=m$. 
If $\gep_m=+1$, $c(m-1)+1=c(m)$, and since $y \geq e^2 > 0$, $\tilde{y}_m=h^{-1}_{\tilde{\ga}_{c(m)}}(y)$ is 
defined. Moreover, by \refL{L:fractions-related}, $\ga_m=\tilde{\ga}_{c(m-1)+1}= \tilde{\ga}_{c(m)}$, and hence, 
$y_m=\tilde{y}_m$. 

If $\gep_m=-1$, $c(m-1)+1=c(m)-1$. Since $y\geq e^2 >1$,
$h^{-1}_{\tilde{\ga}_{c(m)}}(y) > h^{-1}_{\tilde{\ga}_{c(m)}}(1)=0$, and hence 
$\tilde{y}_m=h^{-1}_{\tilde{\ga}_{c(m-1)+1}} \circ h^{-1}_{\tilde{\ga}_{c(m)}}(y)$ is defined. 
Moreover, by \refL{L:fractions-related}, we have $\ga_m=1- \tilde{\ga}_{c(m-1)+1}$ and 
$\ga_m=\tilde{\ga}_{c(m-1)+1} \tilde{\ga}_{c(m)}$. 
As $\ga_m \in (0,1/2)$, $\tilde{\ga}_{c(m-1)+1}\in (1/2,1)$, and hence 
$\tilde{\ga}_{c(m)}=\tilde{\ga}_{c(m-1)+2}=1/\tilde{\ga}_{c(m-1)+1}-1$. 
We use \refL{L:herman-blocks} with $r=\ga_m$, $r_1=\tilde{\ga}_{c(m-1)+1}$ and $r_2=\tilde{\ga}_{c(m)}$ 
to conclude that 
$|y_m-\tilde{y}_m|= |h_{\ga_m}^{-1}(y)-  h^{-1}_{\tilde{\ga}_c(m-1)+1} \circ h^{-1}_{\tilde{\ga}_{c(m)}}(y)| 
\leq 1+e^{-1}$.

Now assume that the assertion holds for $j+1$, with $n < j+1 \leq m$. To prove it for $j$ we consider two cases. 

If $\gep_j=+1$, $c(j-1)+1=c(j)$. By the induction hypothesis for $j+1$, $|y_{j+1}-\tilde{y}_{j+1}| \leq 2 (1+e^{-1})$.
Combining with $y_{j+1}\geq  e^2$, we obtain $\tilde{y}_{j+1}\geq e^2 - 2 (1+e^{-1})>4$. This implies that 
$\tilde{y}_j=h^{-1}_{\tilde{\ga}_{c(j)}}(\tilde{y}_{j+1})$ is defined. 
On the other hand, by \refL{L:fractions-related}, $\ga_j=\tilde{\ga}_{c(j-1)+1}= \tilde{\ga}_{c(j)}$. 
Moreover, for all $t \geq 4$ and all $r \in (0,1)$, $(h_r^{-1})'(t)\leq 1/2$. 
Therefore,  
\begin{align*}
|y_j-\tilde{y}_j|= |h^{-1}_{\ga_j}(y_{j+1}) - h^{-1}_{\tilde{\ga}_{c(j)}}(\tilde{y}_{j+1})|
& = |h^{-1}_{\ga_j}(y_{j+1}) - h^{-1}_{\ga_j}(\tilde{y}_{j+1})|   \\
& \leq (1/2) |y_{j+1}- \tilde{y}_{j+1}|  \leq (1/2) 2 (1+e^{-1}) \leq 2 (1+e^{-1}).
\end{align*}

If $\gep_j=-1$, $c(j-1)+1=c(j)-1$. As in the previous case, the induction hypothesis for $j+1$ gives us 
$|y_{j+1}-\tilde{y}_{j+1}| \leq 2 (1+e^{-1})$. 
Combining with $y_{j+1}\geq e^2$ we obtain $\tilde{y}_{j+1}\geq e^2 - 2 (1+e^{-1})>4$. 
This implies that $\tilde{y}_j=h^{-1}_{\tilde{\ga}_{c(j)-1}}\circ h^{-1}_{\tilde{\ga}_{c(j)}}(\tilde{y}_{j+1})$ is defined. 

On the other hand, by \refL{L:fractions-related}, $\ga_j=1- \tilde{\ga}_{c(j-1)+1}$, and 
$\tilde{\ga}_{c(j)}=\tilde{\ga}_{c(j-1)+2}=1/\tilde{\ga}_{c(j-1)+1}-1$. 
Using \refL{L:herman-blocks} with $r=\ga_j$, $r_1=\tilde{\ga}_{c(j-1)+1}$ and $r_2=\tilde{\ga}_{c(j)}$ 
we get
\[|h_{\ga_j}(y_j) - h^{-1}_{\tilde{\ga}_{c(j)-1}}\circ h^{-1}_{\tilde{\ga}_{c(j)}}(y_j) |\leq 1+e^{-1}.\]
By elementary calculations one may see that for all $t\geq 4$, 
$\big(h^{-1}_{\tilde{\ga}_{c(j)-1}}\circ h^{-1}_{\tilde{\ga}_{c(j)}}\big)'(t)\leq 1/2$. 
Therefore,  
\begin{align*}
|y_j-\tilde{y}_j|
&=|h^{-1}_{\ga_j}(y_{j+1})-h^{-1}_{\tilde{\ga}_{c(j)-1}}\circ h^{-1}_{\tilde{\ga}_{c(j)}}(\tilde{y}_{j+1})|\\
& \leq |h^{-1}_{\ga_j}(y_{j+1}) -  h^{-1}_{\tilde{\ga}_{c(j)-1}}\circ h^{-1}_{\tilde{\ga}_{c(j)}}(y_{j+1})|  \\
& \quad \quad + |h^{-1}_{\tilde{\ga}_{c(j)-1}}\circ h^{-1}_{\tilde{\ga}_{c(j)}}(y_{j+1})
 - h^{-1}_{\tilde{\ga}_{c(j)-1}}\circ h^{-1}_{\tilde{\ga}_{c(j)}}(\tilde{y}_{j+1})|  \\               
& \leq  (1+ e^{-1})+ (1/2) \cdot  |y_{j+1}- \tilde{y}_{j+1}| \\
& \leq (1+ e^{-1})+ (1/2) \cdot 2 (1+e^{-1}) = 2 (1+e^{-1}).
\end{align*}
This completes the proof of the induction. 

The proof of the lemma when item (ii) holds is similar, indeed easier, and is left to the reader. 
\end{proof}

\begin{proof}[Proof of \refP{P:Herman-Yoccoz-criterion}]
Let $\exp^{\circ i}$ denote the $i$-fold composition of the exponential map $x\mapsto e^x$. 

Assume that $\ga$ satisfies the criterion in \refP{P:Herman-Yoccoz-criterion}. 
Fix an arbitrary $n \geq 0$. 
Let $n'$ be an integer with 
\begin{equation}\label{E:P:Herman-Yoccoz-criterion-2}
n' \geq n + \exp^{\circ 2}(2) + 2 (1+e^{-1}).
\end{equation}
Applying the criterion in \refP{P:Herman-Yoccoz-criterion} with $n'$, there is $m' \geq n'$ such that 
\[h_{\ga_{m'-1}} \circ \dots \circ h_{\ga_{n'}} (0) \geq \C{B}(\ga_{m'}).\]
By \refL{L:h-vs-cocycle}-(iii), there is $m \geq n'$ such that 
\[h_{\ga_{m}} \circ h_{\ga_{m-1}} \circ \dots \circ h_{\ga_{n'}} (0) 
\geq \C{B}(\ga_{m+1})+29.\]
We are going to show that the pair $c(m)+1, n$ satisfies the inequality in \refD{D:Herman-Yoccoz-criterion}.

By \refP{P:brjuno-standard-vs-modified} and \refL{L:fractions-related}-(iv),  
\[\C{B}(\ga_{m+1})+ 29= \C{B}(\tilde{\ga}_{c(m)+1})+ 29 \geq \tilde{\C{B}}(\tilde{\ga}_{c(m)+1}).\]
Combining with the previous inequality, we obtain
\begin{equation*}
h_{\ga_{m}} \circ h_{\ga_{m-1}} \circ \dots \circ h_{\ga_{n'}} (0) \geq \tilde{\C{B}}(\tilde{\ga}_{c(m)+1}).
\end{equation*}
This implies that there is an integer $j$, with $n' \leq j \leq m$ such that 
\begin{equation}\label{E:P:Herman-Yoccoz-criterion-1} 
h^{-1}_{\ga_j} \circ \dots \circ h^{-1}_{\ga_{m}} (\tilde{\C{B}}(\tilde{\ga}_{c(m)+1})) \leq 0.
\end{equation}
Now we consider two cases based on the size of $\tilde{\C{B}}(\tilde{\ga}_{c(m)+1})$. 

First assume that $\tilde{\C{B}}(\tilde{\ga}_{c(m)+1}) < e^2$. 
Using $h_r(y)\geq y+1$ several times, we obtain 
\[e^2 \leq n'-n \leq  m-n \leq c(m)-n \leq h_{\tilde{\ga}_{c(m)}} \circ \dots \circ h_{\tilde{\ga}_n} (0).\]
Thus,  we have the desired inequality 
$\tilde{\C{B}}(\tilde{\ga}_{c(m)+1}) \leq h_{\tilde{\ga}_{c(m)}} \circ \dots \circ h_{\tilde{\ga}_n} (0)$. 

Now assume that $\tilde{\C{B}}(\tilde{\ga}_{c(m)+1}) \geq e^2$. 
Combining with \refE{E:P:Herman-Yoccoz-criterion-1}, there is an integer $j'$, with $j \leq j' \leq m$, such that 
\[h^{-1}_{\ga_{j'}} \circ \dots \circ h^{-1}_{\ga_{m}} (\tilde{\C{B}}(\tilde{\ga}_{c(m)+1})) 
\in [e^2,\exp^{\circ 1}(e^2)].\]
This is because any such orbit must pass through the interval $[e^2,\exp^{\circ 1}(e^2)]$.
We may now apply \refL{L:herman-chains}, to conclude that 
$h^{-1}_{\tilde{\ga}_{c(j'-1)+1}} \circ \dots \circ h^{-1}_{\tilde{\ga}_{c(m)}}(\tilde{\C{B}}(\tilde{\ga}_{c(m)+1}))$ 
is defined and 
\begin{align*}
h^{-1}_{\tilde{\ga}_{c(j'-1)+1}} \circ \dots \circ & h^{-1}_{\tilde{\ga}_{c(m)}}(\tilde{\C{B}}(\tilde{\ga}_{c(m)+1})) \\
& \leq 
h^{-1}_{\tilde{\ga}_{c(j'-1)+1}} \circ \dots \circ h^{-1}_{\tilde{\ga}_{c(m)}}(\tilde{\C{B}}(\tilde{\ga}_{c(m)+1}))
- h^{-1}_{\ga_j'} \circ \dots \circ h^{-1}_{\ga_{m}}(\tilde{\C{B}}(\tilde{\ga}_{c(m)+1}))  \\
& \qquad \qquad + h^{-1}_{\ga_{j'}} \circ \dots \circ h^{-1}_{\ga_{m}}(\tilde{\C{B}}(\tilde{\ga}_{c(m)+1})) \\
& \leq 2(1+e^{-1}) + \exp^{\circ 1}(e^2) = 2 (1+ e^{-1}) + \exp^{\circ 2}(2).
\end{align*}
On the other hand, for all $j'\geq 0$, $j' \leq c(j')$. 
Then, combining with \refE{E:P:Herman-Yoccoz-criterion-2}, we have 
\[c(j'-1)-n+1 \geq j'-1-n+1 \geq j - n \geq n' -n \geq 2 (1+e^{-1})+\exp^{\circ 2}(2).\]
Now, using $h_r(y)\geq y+1$ several times, we obtain 
\[2(1+e^{-1}) +\exp^{\circ 2}(2) \leq c(j'-1)-n+1 \leq h_{\tilde{\ga}_{c(j'-1)}} \circ \dots \circ h_{\tilde{\ga}_n}(0).\]
Combining the above inequalities we obtain 
\[h^{-1}_{\tilde{\ga}_{c(j'-1)+1}} \circ \dots \circ h^{-1}_{\tilde{\ga}_{c(m)}}(\tilde{\C{B}}(\tilde{\ga}_{c(m)+1}))
\leq h_{\tilde{\ga}_{c(j'-1)}} \circ \dots \circ h_{\tilde{\ga}_n}(0),\] 
which gives the desired relation 
\[\tilde{\C{B}}(\tilde{\ga}_{c(m)+1})  \leq h_{\tilde{\ga}_{c(m)}} \circ \dots \circ h_{\tilde{\ga}_n}(0).\]
This completes the proof of one side of the proposition. 

The other side of the proposition may be proved along the same lines, but there are some technical 
differences. 

Assume that $\ga$ satisfies \refD{D:Herman-Yoccoz-criterion}. 
Fix an arbitrary $n \geq 0$. 
Let $n'$ be an integer with 
\begin{equation}\label{E:P:Herman-Yoccoz-criterion-3}
n' \geq 2n + 2 \exp^{\circ 3}(2) + 4 (1+e^{-1}).
\end{equation}
Applying \refD{D:Herman-Yoccoz-criterion} with $n'$, there is $m' \geq n'$ such that 
\[h_{\tilde{\ga}_{m'-1}} \circ \dots \circ h_{\tilde{\ga}_{n'}} (0) \geq \tilde{\C{B}}(\tilde{\ga}_{m'}).\]
Recall that $c(m)\to +\infty$ as $m\to +\infty$. By \refL{L:h-vs-cocycle}-(ii), there is $m \geq n'$ such that 
\[h_{\tilde{\ga}_{c(m)}} \circ h_{\tilde{\ga}_{c(m)-1}} \circ \dots \circ h_{\tilde{\ga}_{n'}} (0) 
\geq \tilde{\C{B}}(\tilde{\ga}_{c(m)+1})+29.\]
Note that $c(m)\geq m \geq n'$. 
We are going to show that the pair $m+1, n$ satisfies the inequality in \refP{P:Herman-Yoccoz-criterion}.

By \refP{P:brjuno-standard-vs-modified} and \refL{L:fractions-related}-(iv),  
\[ \tilde{\C{B}}(\tilde{\ga}_{c(m)+1})+29 \geq \C{B}(\tilde{\ga}_{c(m)+1})= \C{B}(\ga_{m+1}).\]
Combining with the previous inequality, we obtain
\begin{equation*}
h_{\tilde{\ga}_{c(m)}} \circ h_{\tilde{\ga}_{c(m)-1}} \circ \dots \circ h_{\tilde{\ga}_{n'}} (0) 
\geq \C{B}(\ga_{m+1}).
\end{equation*}
This implies that there is an integer $j$, with $n' \leq j \leq c(m)$ such that 
\begin{equation}\label{E:P:Herman-Yoccoz-criterion-4} 
h^{-1}_{\tilde{\ga}_j} \circ \dots \circ h^{-1}_{\tilde{\ga}_{c(m)}} (\C{B}(\ga_{m+1})) \leq 0.
\end{equation}
Now we consider two cases based on the size of $\C{B}(\ga_{m+1})$. 

First assume that $\C{B}(\ga_{m+1}) < \exp^{\circ 3}(2)$. 
Using $h_r(y)\geq y+1$ several times, we obtain 
\[\exp^{\circ 3}(2) \leq n'-n \leq  m-n \leq h_{\ga_{m}} \circ \dots \circ h_{\ga_n} (0).\]
Thus, we have the desired inequality $\C{B}(\ga_{m+1}) \leq h_{\ga_{m}} \circ \dots \circ h_{\ga_n} (0)$.

Now assume that $\C{B}(\ga_{m+1}) \geq  \exp^{\circ 3}(2)= \exp^{\circ 2}(e^2)$. 
Combining with \refE{E:P:Herman-Yoccoz-criterion-4}, there is an integer $j'$, with $j \leq c(j')-1 \leq c(m)$, such that 
\[h^{-1}_{\tilde{\ga}_{c(j')-1}} \circ\dots\circ h^{-1}_{\tilde{\ga}_{c(m)}} (\C{B}(\ga_{m+1})) 
\in [e^2,\exp^{\circ 2}(e^2)].\]
This is because any such orbit has two consecutive elements in the interval $[e^2,\exp^{\circ 3}(2)]$, 
and the image of the map $i \mapsto c(i)$ covers at least one element of any pair of consecutive integers. 
We may now apply \refL{L:herman-chains}, to conclude that 
$h^{-1}_{\ga_{j'}} \circ \dots \circ h^{-1}_{\ga_{m}}( \C{B}(\ga_{m+1}))$ is defined and 
\begin{align*}
h^{-1}_{\ga_{j'}} \circ \dots \circ h^{-1}_{\ga_{m}}( \C{B}(\ga_{m+1})) 
& \leq 
h^{-1}_{\ga_{j'}} \circ \dots \circ h^{-1}_{\ga_{m}}( \C{B}(\ga_{m+1})) 
- h^{-1}_{\tilde{\ga}_{c(j')-1}} \circ\dots\circ h^{-1}_{\tilde{\ga}_{c(m)}} (\C{B}(\ga_{m+1})) \\
& \qquad + h^{-1}_{\tilde{\ga}_{c(j')-1}} \circ\dots\circ h^{-1}_{\tilde{\ga}_{c(m)}} (\C{B}(\ga_{m+1})) \\
& \leq 2(1+e^{-1}) +\exp^{\circ 3}(2).
\end{align*}
On the other hand, for all $j'\geq 0$, $1+ 2j' \geq c(j')$. 
Then, combining with our choice of $n'$ in \refE{E:P:Herman-Yoccoz-criterion-3}, we have 
\[j'-n \geq (c(j')-1)/2 -n \geq j/2 -n \geq n'/2 -n \geq \exp^{\circ 3}(2) + 2 (1+e^{-1}).\]
Now, using $h_r(y)\geq y+1$ several times, we obtain 
\[2(1+e^{-1}) +\exp^{\circ 3}(2) \leq j'-n \leq h_{\ga_{j'-1}} \circ \dots \circ h_{\ga_n}(0).\]
Combining the above inequalities we obtain 
\[h^{-1}_{\ga_{j'}} \circ \dots \circ h^{-1}_{\ga_{m}}( \C{B}(\ga_{m+1}))\leq h_{\ga_{j'-1}}\circ\dots\circ h_{\ga_n}(0),\] 
which gives the desired relation $\C{B}(\ga_{m+1}) \leq h_{\ga_{m}} \circ \dots \circ h_{\ga_n}(0)$.
This completes the proof of the other side of the proposition. 
\end{proof} 
\section{Elementary properties of the change of coordinates}\label{S:elementary-properties-change-coord}
In this section we establish some basic properties of the changes of coordinates $Y_r$ introduced in 
\refS{S:change-coordinates}. 
The important properties are the relation between $Y_r$ and the function $h_r$ employed in the definition of the 
Herman numbers in \refS{S:arithmetic}, and the relation between $Y_r$ and the functional relation 
for the Brjuno function in \refE{E:Brjuno-functional-equations}. 
In \refP{P:Y_r-vs-h_r^-1} we show that a suitable rescaling of the restriction of $Y_r$ to the 
$y$-axis is uniformly close to $h_r$. 
In \refL{L:Y_n-on-horizontals} we relate the behaviour of $Y_r$ to the main functional relation 
for the Brjuno function. 
In \refL{L:Y_r-distances} we list some geometric properties of the mapping $Y_r: \mathbb{H}' \to \mathbb{H}'$. 

Recall from \refS{SS:standard-fraction} that $h_r$ is a diffeomorphism from $\D{R}$ onto $(0, +\infty)$. 

\begin{propo}\label{P:Y_r-vs-h_r^-1}
For all $r \in (0, 1/2]$ and all $y \geq 1$, 
\[|2 \pi \Im Y_r( i  y/(2\pi)) - h_r^{-1}(y)| \leq \pi.\] 
\end{propo} 

\begin{proof}
The map $h_r^{-1}: (0, +\infty) \to \D{R}$ is given by the formula
\[h_r^{-1}(y)=   
\begin{cases}
r y + \log r^{-1} -1 & \tif  y \geq 1/r,  \\
\log y  & \tif 0 < y \leq 1/r.
\end{cases}\]

We first give some basic estimates needed for the proof. 
For all $t\in [0,1]$ we have $t\leq e^t-1 \leq e t$, for all $t\in [0, 3\pi/2]$ we have $1-e^{-t}\leq t$, and for 
all $t\in \D{R}$ we have $|1-e^{it}|\leq |t|$. 
Then, for all $r \in (0, 1/2]$, we have 
\begin{equation}\label{E:P:Y_r-vs-h_r-0}
|e^{-3\pi r}-e^{\pi r i}| \leq |e^{-3\pi r} -1| + |1-e^{\pi r i}| \leq 3\pi r+\pi r= 4\pi r, 
\end{equation}
and 
\begin{equation}\label{E:P:Y_r-vs-h_r^-1}
|e^{-3\pi r}-e^{\pi r i}| \geq |\Im (e^{-3\pi r}-e^{\pi r i})| =\sin (\pi r) \geq \pi r/2.
\end{equation}

Now we consider two cases. First assume that $1 \leq y \leq 1/r$. We note that 
\begin{align*}
\big| e^{-3\pi r}-e^{-\pi r i}e^{r y} \big| 
&\leq |e^{-3\pi r} -1| + |1- e^{-\pi r i}| + |e^{-\pi r i} - e^{-\pi r i}e^{r y}|  \\
&\leq 3\pi r + \pi r + e r y= 4 \pi r+ e r y,
\end{align*}
and 
\begin{equation*}
\big| e^{-3\pi r}-e^{-\pi r i}e^{r y} \big|  \geq |e^{\pi r i} e^{r y}| - |e^{-3\pi r}| \geq e^{r y} -1 \geq r y.
\end{equation*}
These imply that 
\begin{equation*}
\Big|\frac{e^{-3\pi r}-e^{-\pi r i}e^{r y}}{y(e^{-3\pi r}-e^{\pi r i}) }\Big|
\leq \frac{4\pi r+ e r y}{y \pi r/2}
\leq \frac{4\pi/y + e}{\pi/2} 
\leq \frac{4\pi + e}{\pi/2} \leq e^{\pi}, 
\end{equation*}
and 
\begin{equation*}  
\Big|\frac{e^{-3\pi r}-e^{-\pi r i}e^{r y}}{y(e^{-3\pi r}-e^{\pi r i}) }\Big| 
\geq \frac{r y}{y(4\pi r)} = \frac{1}{4\pi}\geq e^{-\pi}.
\end{equation*}
For $1\leq y \leq 1/r$, we have 
\begin{equation*}
2\pi \Im Y_r ( i  y/(2\pi)) - h_r^{-1}(y) = 
\log \Big| \frac{e^{-3\pi r}-e^{-\pi r i}e^{r y}}{e^{-3\pi r}-e^{\pi r i}}\Big| 
- \log y 
= \log \Big| \frac{e^{-3\pi r}-e^{-\pi r i}e^{r y}}{y(e^{-3\pi r}-e^{\pi r i}) }\Big|.
\end{equation*} 
Combining these, one obtains the inequality in the proposition when $1\leq y \leq 1/r$.

Now assume that $y \geq 1/r$. 
First note that 
\[|e^{-3\pi r}-e^{-\pi r i}e^{r y}| \leq |e^{-3\pi r}| + |e^{-\pi r i}e^{r y}| \leq 1+ e^{ry} 
\leq 2 e^{ry},\]
and 
\begin{equation*}
|e^{-3\pi r}-e^{-\pi r i}e^{r y}| \geq |e^{-\pi r i}e^{r y}| - |e^{-3\pi r}| \geq e^{ry} -1 \geq e^{ry}/2.
\end{equation*}

For $y \geq 1/r$, we have 
\begin{align*}
2\pi \Im Y_r ( i  y/(2\pi)) - h_r^{-1}(y) &= 
\log \Big| \frac{e^{-3\pi r} - e^{-\pi r i}e^{r y}}{e^{-3\pi r}-e^{\pi r i}}\Big| 
- r y - \log r^{-1} + 1\\
&=\log \Big|\frac{r(e^{-3\pi r} - e^{-\pi r i}e^{r y})}
{e^{r y}(e^{-3\pi r}-e^{\pi r i})}\Big| + 1 \\
&=\log \Big| \frac{r}{e^{-3\pi r} - e^{\pi r i}}\Big| 
+ \log \Big| \frac{e^{-3\pi r}-e^{-\pi r i}e^{r y}}{e^{r y}}\Big|+ 1. 
\end{align*}
Using $\log (2/\pi)+ \log 2 + 1 \leq \pi$ and $- \log (4\pi) -\log 2+1 \geq -\pi$, these imply the desired inequality in 
the proposition when $y \geq 1/r$. 
\end{proof}

Compare the item (iii) in the following lemma to the second functional equation in 
\eqref{E:Brjuno-functional-equations}.

\begin{lem}\label{L:Y_n-on-horizontals}
Let $r\in (0, 1/2]$. We have  
\begin{itemize}
\item[(i)] for all $y\geq 0$ and all $x_0 \in [0, 1/r]$,
\[\max_{x\in x_0+\D{Z}} \Im Y_r (x+iy)  = \Im Y_r(x'+iy),\]
where $x' \in (x_0 + \D{Z}) \cap [1/(2r)-1, 1/(2r)]$; 
\item[(ii)] for all $y\geq 0$, and all $x\in[1/(2r)-1, 1/(2r)]$ we have 
\[2\pi ry+ \log (1/r) - 4 \leq 2\pi \Im Y_r (x+iy) \leq 2\pi ry + \log (1/r) + 2.\]  
\end{itemize}
\end{lem}

\begin{proof}
(i): Note that 
\begin{align*}
\max_{x\in x_0+\D{Z}} \Im Y_r (x+iy) 
& =\frac{1}{2\pi}\max_{x \in x_0+\D{Z}}
\log\Big|\frac{e^{-3\pi r}-e^{-\pi ri}e^{-2\pi rix}e^{2\pi r y}}{e^{-3\pi r}-e^{\pi ri}}
\Big|.
\end{align*}
The maximum of $|e^{-3\pi r}-e^{-\pi ri}e^{-2\pi ri(x_0+k)}e^{2\pi r y}|$ occurs when 
$e^{-\pi ri}e^{-2\pi ri(x_0+k)}$ is closest to the negative real axis. 
This happens when $|\pi r i+ 2\pi ri (x_0+k) -\pi i| \leq \pi r$, which implies that $|(x_0+k)-(1/(2r) -1/2)|\leq 1/2$. 

(ii): For $x \in \D{R}$, 
\begin{align*}
2\pi \Im Y_r (x+iy) - 2\pi ry - \log \frac{1}{r}
& =\log\Big|\frac{e^{-3\pi r}-e^{-\pi ri}e^{-2\pi rix}e^{2\pi r y}}{e^{-3\pi r}-e^{\pi ri}} 
\cdot r e^{-2\pi ry} \Big |.
\end{align*}

Assume that $2\pi ry\geq 1$. By \refE{E:P:Y_r-vs-h_r^-1}, 
\begin{align*}
\Big| \frac{e^{-3\pi r}-e^{-\pi ri}e^{-2\pi rix}e^{2\pi r y}}{e^{-3\pi r}-e^{\pi ri}} 
\cdot r e^{-2\pi ry} \Big | 
& \leq \frac{|e^{-3\pi r}|+ |e^{-\pi ri}e^{-2\pi rix}e^{2\pi r y}|}{|e^{-3\pi r}-e^{\pi ri}|} 
\cdot r e^{-2\pi ry} \\
& \leq \frac{2 e^{2\pi ry}}{\pi r/2} \cdot r e^{-2\pi ry} 
= \frac{4}{\pi} \leq e^2. 
\end{align*}
On the other hand, by \refE{E:P:Y_r-vs-h_r-0}, 
\begin{align*}
\Big| \frac{e^{-3\pi r}-e^{-\pi ri}e^{-2\pi rix}e^{2\pi r y}}{e^{-3\pi r}-e^{\pi ri}} 
\cdot r e^{-2\pi ry} \Big| 
& \geq \frac{|e^{-\pi ri}e^{-2\pi rix}e^{2\pi r y}| - |e^{-3\pi r}|}{|e^{-3\pi r}-e^{\pi ri}|} 
\cdot r e^{-2\pi ry} \\
& \geq \frac{e^{2\pi ry} -1}{|e^{-3\pi r}-e^{\pi ri}|} \cdot r e^{-2\pi ry} \\
& \geq \frac{e^{2\pi ry}/2}{4\pi r} \cdot r e^{-2\pi ry} 
= \frac{1}{8\pi} \geq e^{-4}. 
\end{align*}

Now assume that $0 \leq 2\pi ry \leq 1$. 
For all $x\in \D{R}$, by \refE{E:P:Y_r-vs-h_r^-1}, 
\begin{align*}
\Big| \frac{e^{-3\pi r}-e^{-\pi ri}e^{-2\pi rix}e^{2\pi r y}}{e^{-3\pi r}-e^{\pi ri}} 
\cdot r e^{-2\pi ry} \Big | 
&\leq \Big| \frac{e^{-3\pi r}+ e^{2\pi r y}}{e^{-3\pi r}-e^{\pi ri}} \Big | \cdot r e^{-2\pi ry}\\
&\leq \frac{3\pi r+ e 2\pi ry}{\pi r/2} \cdot  r e^{-2\pi ry} \\
&= (6r+ e 4ry) e^{-2\pi ry}.
\end{align*}
Hence, using $y\leq 1/(2\pi r)$ and $r\in (0,1/2]$, we obtain  
\[ (6r+ e 4ry) e^{-2\pi ry}
\leq (3 + \frac{2e}{\pi}) e^{-2\pi ry}
\leq 3 + \frac{2e}{\pi} \leq e^2.\]
On the other hand, for $x\in [1/(2r)-1, 1/(2r)]$, 
we have 
\[|e^{-3\pi r}-e^{-\pi ri}e^{-2\pi rix}e^{2\pi r y}| \geq e^{2\pi ry}.\]
Hence, by \refE{E:P:Y_r-vs-h_r-0}, 
\[\Big| \frac{e^{-3\pi r}-e^{-\pi ri}e^{-2\pi rix}e^{2\pi r y}}{e^{-3\pi r}-e^{\pi ri}} 
\cdot r e^{-2\pi ry} \Big | 
\geq \frac{e^{2\pi ry}}{4\pi r} \cdot r e^{-2\pi ry} = \frac{1}{4\pi} \geq e^{-4}. 
\]
These imply the desired inequality in part (iii). 
\end{proof}

\begin{lem}\label{L:Y_r-distances}
For all $r\in (0, 1/2]$, we have 
\begin{itemize}
\item[(i)] for all $x \in [0,1/r]$ and all $y\geq -1$, 
\[\Im Y_r(x+ iy) \geq \Im Y_r(iy) -1/(2\pi);\] 
\item[(ii)]for all $x\in [0, 1/r]$ and all $y_1 \geq y_2 \geq -1$,
\[\Im Y_r(x+iy_1) - \Im Y_r(x+ iy_2) \leq \Im Y_r(iy_1)- \Im Y_r(iy_2) + 1/(2\pi);\]
\item[(iii)]for all $y_1 \geq y_2 \geq -1$ and $y \geq 0$,
\[\Im Y_r(i y+iy_1)- \Im Y_r(i y+iy_2) \leq  \Im Y_r(iy_1)- \Im Y_r(iy_2)+ 1/(4\pi);\]
\item[(iv)]for all $y_1 \geq y_2\geq -1$ and $y \in [0,5/\pi]$,
\[\Im Y_r(iy_1) - \Im Y_r(iy_2) \leq  \Im Y_r(iy+ iy_1)- \Im Y_r(iy+ iy_2)+ 5/\pi.\]
\end{itemize}
\end{lem}

\begin{proof}
Part (i): Recall from the proof of \refL{L:uniform-contraction-Y_r} that 
\begin{align*}
\frac{\partial Y_r}{\partial s}(w)=\frac{\partial Y_r}{\partial w}(w) + \frac{\partial Y_r}{\partial \ol{w}}(w) 
= r+ \frac{i r \Im \xi}{|\xi-1|^2},
\end{align*}
where $w=s+iy$ and $\xi=e^{-3\pi r} e^{\pi r i } e^{2\pi r  i  w}$. 
We note that when $s \in [0, 1/(2r)-1/2] \cup [1/r-1/2, 1/r]$, $\Im \partial Y_r (s+iy)/\partial s \geq 0$, 
and when $s \in [1/(2r)-1/2, 1/r-1/2]$, $\Im \partial Y_r (s+iy)/\partial s \leq 0$. 
Moreover, when $s \in [1/r-1, 1/r-1/2]$, we have 
\begin{align*}
\Big| \Im \frac{\partial Y_r}{\partial s}(s+iy)\Big|   
= \frac{r |\Im \xi|}{|\xi-1|^2} \leq \frac{r |e^{-3\pi r}e^{2\pi r} \sin (\pi r+ 2\pi rs)|}{(1-e^{-\pi r})^2}
\leq \frac{r e^{-\pi r} \sin {(\pi r)}}{1+ e^{-2\pi r}-2e^{-\pi r}}\leq \frac{1}{\pi}.  
\end{align*}
To see the last inequality, it is enough to show that $g(u)=u \sin u-e^{u}-e^{-u}+2 \leq 0$ for all $u=\pi r\in[0,\pi/2]$. 
By Taylor's remainder theorem, for all $u \in [0, \pi/2]$, there is $u_0 \in [0, u]$ such that 
\[g(u)= g(0)+ u g'(0)+ u^2 g''(0)/2+u_0^3 g^{(3)}(u_0)/6 
= u_0^3 (-3\sin u_0 -u_0 \cos u_0 - e^{u_0}+ e^{-u_0})/6 \leq 0.\]

For $x\in [0,1/r]$, we use the formula 
\begin{align*}
\Im Y_r(x+iy) - \Im Y_r(iy) = \int_0^x \Im \frac{\partial Y_r(s+iy)}{\partial s} \, ds
\end{align*}
to obtain a lower bound. We consider four cases: 

$\bullet$ When $x\in [0, 1/(2r)-1/2]$, the integrand is non-negative, and hence 
\[\int_0^x \Im \frac{\partial Y_r(s+iy)}{\partial s} \, ds\geq 0.\]    

$\bullet$ When $x\in [1/(2r)-1/2, 1/r-1]$, then $\frac{1}{r}-1-x \in [0, 1/(2r)-1/2]$ and by the previous case, 
\begin{align*}
\int_0^x \Im \frac{\partial Y_r(s+iy)}{\partial s} \, ds
& =\int_0^{1/r-1-x} \Im \frac{\partial Y_r(s+iy)}{\partial s} \, ds 
+ \int_{1/r-1-x}^{x} \Im \frac{\partial Y_r(s+iy)}{\partial s} \, ds \\
&\geq \int_{1/r-1-x}^{x} \Im \frac{\partial Y_r(s+iy)}{\partial s} \, ds \\
&= \int_{1/r-1-x}^{x} \frac{r \sin (\pi r+ 2\pi rs)}{1+e^{-6\pi r - 4\pi ry} -2 e^{-3\pi r -2\pi ry} \cos (\pi r+ 2\pi rs)} \,ds\\
&= \int_{2\pi -\pi r-2\pi rx}^{\pi r+ 2\pi rx} \frac{r \sin (\gj)}{1+e^{-6\pi r - 4\pi ry} -2 e^{-3\pi r -2\pi ry} \cos \gj} \, \frac{d\gj}{2\pi r}=0. 
\end{align*}

$\bullet$ When $x\in [1/r-1, 1/r-1/2]$, by the previous case and the above estimate,  
\begin{align*}
\int_0^x \Im \frac{\partial Y_r(s+iy)}{\partial s} \, ds
& =\int_0^{1/r-1} \Im \frac{\partial Y_r(s+iy)}{\partial s} \, ds 
+ \int_{1/r-1}^{1/r-1/2} \Im \frac{\partial Y_r(s+iy)}{\partial s} \, ds \\
&\geq \int_{1/r-1}^{1/r-1/2} \Im \frac{\partial Y_r(s+iy)}{\partial s} \, ds 
\geq \frac{-1}{\pi} \cdot \frac{1}{2}= \frac{-1}{2\pi}.
\end{align*}
 
$\bullet$ When $x \in [1/r-1/2, 1/r]$, by the previous case,  
\begin{align*}
\int_0^x \Im \frac{\partial Y_r(s+iy)}{\partial s} \, ds
& =\int_0^{1/r-1/2} \Im \frac{\partial Y_r(s+iy)}{\partial s} \, ds 
+ \int_{1/r-1/2}^{x} \Im \frac{\partial Y_r(s+iy)}{\partial s} \, ds \\
&\geq  \frac{-1}{2\pi} + \int_{1/r-1/2}^{x} \Im \frac{\partial Y_r(s+iy)}{\partial s} \, ds \geq \frac{-1}{2\pi}.
\end{align*}
This completes the proof of Part (i).  

Part (ii): Let  
\[g(\gj, t_1, t_2)= \frac{1+ e^{2t_1} - 2e^{t_1}\cos \gj}{1+ e^{2t_2} - 2e^{t_2}\cos \gj}, 
\quad \gj \in \D{R}, t_1\geq t_2 >  0.\]
Since $1+ e^{2t_2} - 2e^{t_2}\cos \gj \geq 1+ e^{2t_2} - 2e^{t_2} = (e^{t_2}-1)^2>0$, $g$ is well-defined, 
and is positive.  
Moreover, 
\[\partial_\gj g(\gj, t_1, t_2)= \frac{-2\sin \gj (e^{t_1}-e^{t_2}) (e^{t_1+t_2}-1)}{(1+ e^{2t_2} - 2e^{t_2}\cos \gj)^2}.\]
Thus, $g$ is a decreasing function of $\gj$ on $[0,\pi]+2\pi \D{Z}$, and increasing on $[\pi, 2\pi]+ 2\pi \D{Z}$. 
This implies that $g(\gj, t_1, t_2)\leq g(0, t_1, t_2)$, for all $\gj\in \D{R}$. 
Applying this inequality with $\gj=-\pi r-2\pi rx$, $t_1=3\pi r+2\pi r y_1$, and $t_2=3\pi r+2\pi r y_2$, one obtains  
\[\Big|\frac{e^{-3\pi r}-e^{-\pi r i} e^{-2\pi r ix} e^{2\pi r y_1}}{e^{-3\pi r} - e^{-\pi r i} e^{-2\pi r ix}e^{2\pi r y_2}}\Big|^2 
= g(\gj, t_1, t_2) \leq g(0, t_1, t_2)
= \Big|\frac{e^{-3\pi r} - e^{2\pi r y_1}}{e^{-3\pi r} - e^{2\pi r y_2}} \Big|^2.\]
Therefore, for all $y_1 \geq y_2 \geq -1$, we have 
\begin{align*}
\Im Y_r(x+iy_1) - \Im Y_r(x+ iy_2) 
& = \frac{1}{2\pi}\log\Big|\frac{e^{-3\pi r} - e^{-\pi r i} e^{-2\pi r ix} e^{2\pi r y_1}}
{e^{-3\pi r} - e^{-\pi r i} e^{-2\pi r ix} e^{2\pi r y_2}} \Big |   \\
& \leq \frac{1}{2\pi}\log \Big|\frac{e^{-3\pi r} - e^{2\pi r y_1}}{e^{-3\pi r} - e^{2\pi r y_2}} \Big | \\
&= \Im Y_r(1/r-1/2+iy_1) - \Im Y_r(1/r-1/2+ iy_2). 
\end{align*}
On the other hand, 
\begin{align*}
\Im Y_r(1/r-1/2+iy_1) & = \frac{1}{2\pi} \log \Big|\frac{e^{-3\pi r} - e^{2\pi r y_1}}{e^{-3\pi r} - e^{-\pi ri}} \Big | \\
& \leq \frac{1}{2\pi} \log \Big|\frac{e^{-3\pi r} - e^{-\pi ri} e^{2\pi r y_1}}{e^{-3\pi r} - e^{-\pi ri}} \Big |
=\Im Y_r(iy_1),
\end{align*}
and by the inequality in Part (i),
\[\Im Y_r(1/r-1/2+ iy_2)\geq \Im Y_r(iy_2)-1/(2\pi).\] 
Combining the above inequalities together, we conclude Part (ii) of the lemma. 

Part (iii): 
If $y_1=y_2$, the inequality trivially holds. Below we assume that $y_1> y_2$.  

By the definition of $Y_r$, it is enough to prove that for all $r\in (0, 1/2]$, $y_1\geq y_2\geq 0$, and $y\geq 0$ 
we have 
\[\Big |\frac{e^{-3\pi r}- e^{-\pi r  i }e^{2\pi r (y+y_1)}}{e^{-3\pi r}- e^{-\pi r  i }e^{2\pi r (y+y_2)}}\Big|
\leq \sqrt{2} \Big |\frac{e^{-3\pi r}- e^{-\pi r  i }e^{2\pi r y_1}}{e^{-3\pi r}- e^{-\pi r  i }e^{2\pi r y_2}}\Big|.\]
Let us fix $y_1> y_2$, and consider the M\"obius transformations 
\[M(z)=\frac{e^{-3\pi r}- e^{-\pi ri}e^{2\pi r y_1}z}{e^{-3\pi r}- e^{-\pi ri}e^{2\pi r y_2}z}, \quad 
M_1(z)=\frac{e^{-3\pi r} + e^{2\pi r y_1}z}{e^{-3\pi r} + e^{2\pi r y_2}z}.\]
We have $M(0)=M_1(0)=1$, and $M(\infty)= M_1(\infty)=e^{2\pi r(y_1-y_2)}$. 
Then, there is a M\"obius transformation $M_2$ with $M_2(1)=1$, $M_2(e^{2\pi r(y_1-y_2)})=e^{2\pi r(y_1-y_2)}$,
and $M=M_2 \circ M_1$. The precise form of  $M_2$ is not relevant here. 

An elementary algebraic calculation shows that 
\[M_1(1)\geq (1+ e^{2\pi r(y_1-y_2)})/2.\]
Since $M_1$ maps the real interval $[0, +\infty]$ to $[1, e^{2\pi r (y_1-y_2)}]$, we must have 
\[M_1(e^{2\pi ry}) \geq M_1(1) \geq (1+ e^{2\pi r(y_1-y_2)})/2.\] 
That is, $M_1(1)$ and $M_1(e^{2\pi ry})$ belong to the right hand of the middle point of the interval 
$[1, e^{2\pi r(y_1-y_2)}]$. 

Note that $|M_2'(1)| = |M'(0)/M_1'(0)|=1$. Thus, $M_2$ preserves the line $\Re z= (1+ e^{2\pi r(y_1-y_2)})/2$, 
and maps the real interval $[(1+ e^{2\pi r(y_1-y_2)})/2, e^{2\pi r (y_1-y_2)}]$ to an arc $\gga$ of a circle 
whose center $B$ lies on the vertical line $\Re z= (1+ e^{2\pi r(y_1-y_2)})/2$.
By the above paragraph, $M(e^{2\pi ry})$ and $M(1)$ belong to $\gga$. 
By looking at $\arg M_2'(1)$, one can see that $\gga$ meets the line $\Im z=0$ at $e^{2\pi r(y_1-y_2)}$ 
with asymptotic 
angle $+ \pi r \leq \pi/2$. That is, $\gga$ and $B$ lie in the same component of $\D{C}\setminus \D{R}$.  
See Figure~\ref{F:L:Y_r-distances} for an illustration of the following argument. 

\begin{figure}[ht]
\begin{pspicture}(0,1)(8,4.4) 

\psline[linewidth=.5pt]{->}(0.5,2)(7.5,2)
\psline[linewidth=.5pt]{->}(1,1.2)(1,3.8)
\pscircle[linewidth=.8pt,linestyle=dotted](5,2.8){1}
\psarc[linewidth=.8pt](5,2.8){1}{-53}{90}

\psdots[dotsize=2.5pt](5,3.8)(5,2.8)
\psdots[dotsize=2.5pt](4.4,2)(5.6,2)
\psdot[dotsize=2.5pt](5.98,3)

\psline[linewidth=.5pt](1,2)(5.98,3)
\psline[linewidth=.5pt](5,2.8)(5,3.8)
\psline[linewidth=.5pt](1,2)(5,3.8)

\rput(1.2,1.7){\small $O$}
\rput(4.4,1.7){\small $1$}
\rput(6.1,1.7){\small $e^{2\pi r(y_1-y_2)}$}

\rput(5,4){\small $A$}
\rput(5.1,2.6){\small $B$}

\rput(6.2,3){\small $C$} 

\rput(5.6,3.9){\tiny $M(1)$}
\rput(6.4,3.6){\tiny $M(e^{2\pi ry})$}
\end{pspicture}
\caption{Illustration of the inequality in \refL{L:Y_r-distances}-(iii). The curve $\gga$ is the arc on the dotted circle.}
\label{F:L:Y_r-distances}
\end{figure}
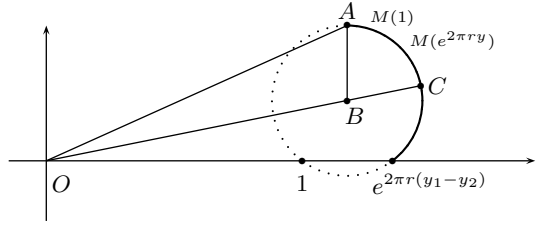

Let $C$ be the point on $\gga$ where $|C|$ achieves its maximum, and let $O$ denote the origin. 
Then, the line from $O$ to $C$ must pass through $B$. 
Let $A$ denote the point where the arc $\gga$ meets the line $\Re z= (1+ e^{2\pi r(y_1-y_2)})/2$. 
In the triangle $OAB$, the angle $\gj$ at the vertex $B$ is at least $\pi/2$. 
By the cosine formula, 
\[|OA|^2=|OB|^2+ |AB|^2-2|OB||AB| \cos \gj \geq |OB|^2 + |AB|^2.\]  
Using the inequality $|OB|^2 + |AB|^2 \geq 2 |OB| |AB|$, we conclude that 
\[2|OA|^2 \geq |OB|^2 + |AB|^2 + 2 |OB||AB|= (|OB|+|AB|)^2.\] 

By the above discussion, $M(1)$ and $M(e^{2\pi ry})$ belong to $\gga$, with $M(1)$ lying between $A$ and 
$M(e^{2\pi ry})$. 
Now, we consider two cases. If $M(1)$ lies between $C$ and $e^{2\pi r (y_1-y_2)}$, then we must have 
$M(e^{2\pi ry})/M(1) \leq 1$.  
If $M(1)$ lies between $A$ and $C$,  by the above inequality, we obtain
\[\frac{|M(e^{2\pi ry})|}{|M(1)|} \leq \frac{|OC|}{|OA|} = \frac{|OB|+ |BA|}{|OA|} \leq \sqrt{2}.\]
This completes the proof of the desired inequality. 

Part (iv): Recall that $Y_r(0)=0$ and hence by \refL{L:uniform-contraction-Y_r}, $\Im Y_r(-i) \geq -9/10$. 
Moreover, $\Im Y_r(iy+ iy_1) \geq \Im Y_r(iy_1)$. 
By rearranging terms and using the inequality in Part (iii), we note that 
\begin{align*}
\Im Y_r(iy_1) - \Im Y_r(iy_2) & - (\Im Y_r(iy+ iy_1) - \Im Y_r(iy+ iy_2)) \\
&=\big(\Im Y_r(iy+ iy_2) -  \Im Y_r(iy_2)\big) + \big( \Im Y_r(iy_1)- \Im Y_r(iy+ iy_1)\big)  \\
& \leq \Im Y_r(iy+ iy_2) -  \Im Y_r(iy_2) \\
& \leq \Im Y_r(i(y-1)) - \Im Y_r(-i) + 1/(4\pi) \\
& \leq \Im Y_r(i (5-\pi)/\pi) + 9/10 + 1/(4\pi).
\end{align*}
On the other hand, by \refE{E:P:Y_r-vs-h_r^-1}, for all $r\in (0, 1/2]$, we have  
\begin{align*}
Y_r(i (5-\pi)/\pi) &=\frac{1}{2\pi} \log \Big| \frac{e^{-3\pi r} - e^{-\pi r i} e^{2r(5-\pi)}}{e^{-3\pi r}- e^{\pi ri}}\Big| \\
&\leq \frac{1}{2\pi} \log \frac{|e^{-3\pi r}-1|+ |1-e^{-\pi r i}|+ |e^{-\pi ri} - e^{-\pi ri}e^{2r(5-\pi)}|}{\pi r/2} \\
&\leq \frac{1}{2\pi} \log \frac{3\pi r+ \pi r + (e^{2r(5-\pi)}-1)}{\pi r/2} \\
&\leq \frac{1}{2\pi} \log \frac{3\pi r+ \pi r + (e^3-8)\pi r/2}{\pi r/2} = \frac{3}{2\pi}.
\end{align*}
Using the inequality $9/10 \leq 3/\pi$, we obtain the desired estimate in Part (iv). 
\end{proof} 
\section{Topology of the sets \texorpdfstring{$\mathbb{A}_\ga$}{A-ga}}\label{S:topology-A}
\subsection{Cantor bouquets and hairy Jordan curves}\label{SS:CB-HJC}
In this section we describe the topology of the sets $\mathbb{A}_\ga$ in terms of the arithmetic properties of $\ga$. 
In particular, here we will prove \refT{T:model-trichotomy-thm}. 
Let us start by presenting the definition of the two main topological objects which appear in the theorem.

A \textbf{Cantor bouquet} is any subset of the plane which is ambiently homeomorphic to a set of the form 
\begin{equation*}\label{E:straight-cantor-bouquet}
\{re^{2\pi i \gj} \in \D{C} \mid 0 \leq \gj \leq 1, 0 \leq r \leq R(\gj) \}
\end{equation*}
where $R: \D{R}/\D{Z} \to [0, 1]$ satisfies the following:
\begin{itemize}
\item[(a)] $R=0$ on a dense subset of $\D{R}/\D{Z}$,  and $R > 0$ on a dense subset of $\D{R}/\D{Z}$, 
\item[(b)] for each $\gj_0\in \D{R}/\D{Z}$ we have 
\[\limsup_{\gj \to \gj_0^+ } R(\gj) = R(\gj_0) = \limsup_{\gj \to \gj_0^-} R(\gj).\]
\end{itemize}

A \textbf{one-sided hairy Jordan curve} is any subset of the plane which is ambiently homeomorphic to a set 
of the form 
\begin{equation*}\label{E:straight-hairy-circle}
\{re^{2\pi i \gj} \in \D{C} \mid 0 \leq \gj \leq 1, 1\leq r \leq 1+ R(\gj) \}
\end{equation*}
where $R: \D{R}/\D{Z} \to [0,1]$ satisfies properties (a) and (b) in the above definition. 

The Cantor bouquet and one-sided hairy Jordan curve enjoy similar topological feature as the standard Cantor set. 
Under an additional mild condition (topological smoothness) they are uniquely characterised by some 
topological axioms, see \cite{AaOv93}.

To study the topology of the sets $\mathbb{M}_\ga$ (and $\mathbb{A}_\ga$), it is convenient to work with the sets 
$I_n^j$ and $I_n$ from \refS{SS:tilings-nest}. That is because each of the sets $I_n^j$ and $I_n$ is the region 
above the graph of a function. Since the sets $I_n^j$, for $j\geq 0$, forms a nest, one is led to an increasing 
collection of functions on a fixed domain. 
Since the nest may shrink to an empty set along a fixed vertical line, we are led to considering functions which attain 
$+\infty$ at some points. 
It turns out that there are two main collections of functions required to capture the topological features of these sets. 
In the next section we build these functions and study their properties. 

\subsection{Height functions}\label{SS:height-functions}
\footnote{The letter $b$ stands for ``base'' and ``p'' for ``pinnacle''; the reason for these will become clear in a moment.}
Recall that the sets $I_n^j$ and $I_n$ consist of closed half-infinite vertical lines. 
For $n\geq -1$, and $j\geq 0$, define $b_n^j:[0, 1/\ga_n] \to [-1, +\infty)$ as 
\begin{equation}\label{E:I_n^j-b_n^j}
b_n^j(x)= \min \{y \mid x+iy \in I_n^j\}.
\end{equation}
Since each $\mathbb{Y}_n$ preserves vertical lines, it follows that 
\[I_n^j= \{w\in \D{C} \mid 0 \leq \Re w \leq 1/\ga_n, \Im w\geq b_n^j(\Re w)\}.\]
By the definition of the sets $I_n^j$ and the functional equations \eqref{E:Y_n-comm-1}--\eqref{E:Y_n-comm-2}, 
one can see that for all $n\geq -1$ and $j\geq 0$, $b_n^j:[0, 1/\ga_n] \to [-1, +\infty)$ is continuous. 
Moreover, by \refE{E:I_n^j-forms-nest}, we must have $b_n^{j+1} \geq b_n^j$ on $[0, 1/\ga_n]$. 
For $n\geq -1$, we define $b_n:[0, 1/\ga_n] \to [-1,+\infty]$ as
\[b_n(x)= \lim_{j\to + \infty} b_n^j(x)= \sup_{j\geq 1} b_n^j(x).\]
Note that $b_n$ is allowed to take $+\infty$. The function $b_n$ describes the set $I_n$, that is, 
\begin{equation}\label{E:I_n-b_n}
I_n= \{w\in \D{C} \mid 0 \leq \Re w \leq 1/\ga_n, \Im w \geq b_n(\Re w)\}.
\end{equation}

By \refL{L:model-almost-periodic}, $b_n^j(0)=b_n^j(1/\ga_n)$, and $b_n^j(x+1)= b_n^j(x)$ for all $x\in [0, 1/\ga_n-1]$. 
Taking limits as $j\to +\infty$, we obtain 
\begin{equation}\label{E:b_n^j-cont-periodic}
b_n(0)=b_n(1/\ga_n), \qquad b_n(x+1)=b_n(x), \text{ for all } x\in [0, 1/\ga_n-1] \text{ and } n\geq -1.   
\end{equation}

Only when $\ga$ is a Brjuno number, for $n\geq -1$ and $j\geq 0$ we inductively define the functions 
\[p_n^j: [0,1/\ga_n] \to [-1, +\infty).\] 
For all $n\geq -1$, we set $p_n^0\equiv (\C{B}(\ga_{n+1})+5\pi)/(2\pi)$. 
Assume that $p_n^j$ is defined for some $j\geq 0$ and all $n\geq -1$. 
We define $p_n^{j+1}$ on $[0, 1/\ga_n]$ as follows. 
For $x_n \in [0,1/\ga_n]$, we find $x_{n+1} \in [0, 1/\ga_{n+1}]$ and $l_{n}\in \D{Z}$ such that 
$-\gep_{n+1} \ga_{n+1} x_{n+1}= x_n - l_n$, and define  
\[p_n^{j+1}(x_n)=\Im \mathbb{Y}_{n+1}\big(x_{n+1} + i p_{n+1}^j (x_{n+1})\big).\]
In other words, the graph of $p_n^{j+1}$ is obtained from applying $\mathbb{Y}_{n+1}$ to the graph of $p_{n+1}^j$, 
and then applying suitable translations by integers. 

Evidently, for all $n\geq -1$ and all $j\geq 0$, we have 
$p_n^{j+1}(x+1)= p_n^{j+1}(x)$, for $x \in [0, 1/\ga_n-1]$.
Moreover, it follows from \eqref{E:Y_n-comm-1} and \eqref{E:Y_n-comm-2} that each 
$p_n^j: [0,1/\ga_n]\to \D{R}$ is continuous, and $p_n^j(0)=p_n^j(1/\ga_n)$. 

Note that for every $n\geq -1$, by \refL{L:Y_n-on-horizontals}-(i)-(ii) and \refE{E:Brjuno-functional-equations}, 
we have 
\begin{align*}
p_n^1&  \leq \max_{x \in [0, 1/\ga_{n+1}]} \Im \mathbb{Y}_{n+1}\big (x+ i (\C{B}(\ga_{n+2})+5\pi)/(2\pi)\big ) \\
& \leq \ga_{n+1} \frac{\C{B}(\ga_{n+2})+5\pi}{2\pi} + \frac{1}{2\pi} \log \frac{1}{\ga_{n+1}}+\frac{1}{\pi} \\
&=\frac{1}{2\pi}\big(\ga_{n+1}\C{B}(\ga_{n+2})+\log\frac{1}{\ga_{n+1}}\big)
+\frac{\ga_{n+1}5\pi}{2\pi}+\frac{1}{\pi}\\
& \leq  \frac{1}{2\pi} \C{B}(\ga_{n+1}) + \frac{5}{4} + \frac{1}{\pi} 
\leq  p_n^0. 
\end{align*}
Using an induction argument, starting with $p_n^ 1 \leq p_n^0$, for every $n\geq -1$, one may show that for all 
$n\geq -1$ and $j\geq 0$ we have 
\[p_n^{j+1}(x) \leq p_n^j(x), \quad \forall x\in [0, 1/\ga_n].\] 
Therefore, we may define the functions 
\[p_n(x)= \lim_{j \to +\infty}  p_n^j(x), \quad \forall x\in [0, 1/\ga_n].\]
It follows that 
\begin{equation}\label{E:p_n-peiodic}
p_n(0)=p_n(1/\ga_n), \quad p_n(x)=p_{n}(x+1), \forall x\in [0, 1/\ga_n-1].
\end{equation}

On the other hand, by definition, $p_n^0 \geq b_n^0$, for all $n\geq -1$. Since the graphs of $b_{n+1}^0$ 
and $p_{n+1}^0$ are mapped to the graphs of $b_n^1$ and $p_n^1$, respectively, by $\mathbb{Y}_{n+1}$ 
and its integer translations, we must have $p_{n}^1 \geq b_n^1$, for all $n\geq -1$. 
By induction, this implies that for all $n\geq -1$ and all $j\geq 0$,  
\begin{equation}\label{E:p_n>=b_n}
p_n^j(x) \geq b_n^j(x), \quad \forall x\in [0, 1/\ga_n].
\end{equation}
In particular, $p_n \geq b_n$ on $[0, 1/\ga_n]$. 

\subsection{Accumulation of the hairs}\label{SS:hairs-accumulate}

\begin{propo}\label{P:b_n-liminfs}
For all $n\geq -1$, we have \footnote{To be clear, the notation $s \to x^+$ means that $s \to x$ and $s>x$. Similarly, 
$s\to x^-$ means that $s\to x$ and $s< x$.} 
\begin{itemize}
\item[(i)] for all $x\in [0, 1/\ga_n)$, $\liminf_{s\to x^+} b_n(s)= b_n(x)$;
\item[(ii)] for all $x\in (0, 1/\ga_n]$, $\liminf_{s\to x^-} b_n(s)= b_n(x)$.
\end{itemize}
\end{propo}

\begin{proof}
By \refL{L:I_n^j-basic-features}, the sets $I_n^j$ are closed. Thus, their intersection $I_n$ is also a closed set. 
Since $I_n$ is the set above the graph of $b_n$, this implies that  
\begin{gather*}
\forall x\in [0, 1/\ga_n), \liminf_{s\to x^+} b_n(s)\geq b_n(x), \quad \tand \quad 
\forall x\in (0, 1/\ga_n], \liminf_{s\to x^-} b_n(s) \geq b_n(x). 
\end{gather*}
So we need to show that the equality holds in both cases. If $b_n(x)=+\infty$, then we automatically have the 
equalities. Indeed, $\liminf$ can be replaced by $\lim$ in that case. 
Below we assume that $b_n(x)\neq \infty$. 

Fix $x_n \in [0, 1/\ga_n]$, and let $e_n \in \{+1, -1\}$ be arbitrary, except that $e_n=+1$ if $x_n=0$, and $e_n=-1$ 
if $x_n=1/\ga_n$. 
Define the sequence $\{e_m\}_{m\geq n}$ according to $e_{m+1}= - \gep_{m+1} e_m$. 

We use $e_n$ to deal with both statements at once. 
That is, to prove the desired equalities, it is enough to show that for every $\delta>0$ there is $x_n'$ strictly 
between $x_n$ and $x_n+ \delta e_n$ such that $b_n(x_n')\leq b_n(x_n)+\delta$. 
The idea of the proof is rather elementary. 
We map $x_n+i b_n(x_n) \in I_n$ to $z_m \in I_m$ using the maps $\mathbb{Y}_j^{-1}$, for large $m\geq n$. 
Then, using \refE{E:b_n^j-cont-periodic}, we find $z_m'=z_m+e_m \in I_m$, which may be mapped to  
$z'_n \in I_n$ using $\mathbb{Y}_j$. 
It follows that $\sign (\Re z'_n - x_n)=e_n$, and by the uniform contraction of $\mathbb{Y}_j$,  
$z'_n$ is close to $x_n +i b_n(x_n)$. 
However, there are some technical difficulties due to $\gep_j=\pm 1$ and $z'_m=z_m+e_m \notin I_m$. 
We present the details in several steps. 

\smallskip

{\em Step 1.} There is an infinite sequence $\{(l_m, x_{m+1})\}_{m\geq n}$, with $l_m \in \D{Z}$ and 
$x_{m+1} \in [0, 1/\ga_{m+1}]$ such that $\Re \mathbb{Y}_{m+1}(x_{m+1}) + l_m= x_m$ for all $m\geq n$. 
Moreover, if $x_m=0$ for some $m\geq n$, then $e_m=+1$. 

\smallskip 

We define the sequence inductively. Assume that $x_m \in [0,1/\ga_m]$ is defined for some $m\geq n$. 
To define $l_m$ and $x_{m+1}$ we proceed as follows:
\begin{itemize}
\item[(a)]if $x_m \in [0, 1/\ga_m] \cap \D{Z}$, we let $l_m= x_m+  (e_m+\gep_{m+1})/2$ and define 
\[x_{m+1}=\frac{1}{-\gep_{m+1}\ga_{m+1}} \cdot (x_m-l_m) =  
\frac{1+ \gep_{m+1}e_m}{2\ga_{m+1}}.\] 
\item[(b)] if $x_m  \in [0, 1/\ga_m] \setminus \D{Z}$ we choose $l_m \in \D{Z}$ such that 
$x_m-l_m \in ((-1-\gep_{m+1})/2,(1-\gep_{m+1})/2)$, and define 
\[x_{m+1} = \frac{x_m- l_m}{-\gep_{m+1}\ga_{m+1}}.\] 
\end{itemize} 

In part (a), $x_{m+1}\in \{0, 1/\ga_{m+1}\}$, depending on the sign of $\gep_{m+1} e_m$. 
In part (b), if $\gep_{m+1}=-1$, then $x_m-l_m \in (0,1)$ and therefore $x_{m+1} \in (0, 1/\ga_{m+1})$. 
If $\gep_{m+1}=+1$, then $x_m - l_m \in (-1,0)$ and therefore $x_{m+1} \in (0, 1/\ga_{m+1})$. 
Thus, in both cases $x_{m+1}$ belongs to $[0, 1/\ga_{m+1}]$. 

By \refE{E:Y_n}, $\Re \mathbb{Y}_{m+1}(x_{m+1})= -\gep_{m+1} \ga_{m+1}x_{m+1}$. 
Thus, $\Re \mathbb{Y}_{m+1}(x_{m+1})+l_m=x_m$.

To prove the latter part of Step 1, note that if $j=n$ then $e_{j}=+1$ by the definition of $e_n$ at the beginning of 
the proof. If $j > n$, by the definition of the sequence $(x_m, l_m)$, $x_j$ must be generated in part (a), since part (b) 
produces values in $(0, 1/\ga_{m})$. 
But, part (a) produces $x_{j}=0$ only if $\gep_{j} e_{j-1}=-1$. Then, $e_{j}= - \gep_{j} e_{j-1}=+1$. 

For all $m\geq n$ we have 
\begin{equation}\label{E:P:bounds-on-l_m}
(1+\gep_{m+1})/2 \leq l_m \leq  a_m + \gep_{m+1}.
\end{equation}
That is because, if $l_m$ is produced in (b) and $\gep_{m+1}=-1$ then the largest integer in 
$[0, 1/\ga_{m}]$ is $a_m-1$ and there is $0 \leq l_m \leq a_m-1$ with $x_m-l_m\in (0,1)$. 
If $l_m$ is produced in (b) and $\gep_{m+1}=+1$ then the largest integer in $[0,1/\ga_m]$ is $a_m$ and there is 
$1 \leq l_m \leq a_m+1$ with $x_m-l_m \in (-1,0)$. 
If $l_m$ is produced in (a), then $x_m \leq a_m + (\gep_{m+1}-1)/2$ and therefore 
$l_m \leq a_m+\gep_{m+1} +e_m/2-1/2 \leq a_m+\gep_{m+1}$. On the other hand, in (a) 
$l_m = x_m + (e_m + \gep_{m+1})/2 \geq (1+\gep_{m+1})/2$, since, by the above paragraph, if $x_m=0$ then 
$e_m=+1$. 
We are done with Step 1. 

Let us say that some level $m \geq n$ is \textbf{eligible}, if one of the following four cases occurs: 
\begin{itemize}
\item[$(E_1)$]  $(e_m, \gep_{m+1})= (+1, -1)$ and $x_m$ belongs to 
\[[0, 1/\ga_m-2] \cup [a_m-2-\ga_{m+1}, a_m-2] \cup [a_m-1-\ga_{m+1}, a_m-1];\]
\item[$(E_2)$] $(e_m, \gep_{m+1})= (+1, +1)$ and $x_m$ belongs to 
\[[0, 1/\ga_m-2] \cup [a_m-1-\ga_{m+1}, a_m-1] \cup [a_m-\ga_{m+1}, a_m];\]
\item[$(E_3)$] $(e_m, \gep_{m+1})= (-1, -1)$ and $x_m$ belongs to 
\[[1, 1/\ga_m-1] \cup [1-\ga_{m+1},1] \cup [a_m-1, a_m-1+\ga_{m+1}];\]
\item[$(E_4)$] $(e_m, \gep_{m+1})= (-1, +1)$ and $x_m$ belongs to 
\[[1, 1/\ga_m-1] \cup [\ga_{m+1}, 2\ga_{m+1}] \cup [a_m, a_m+\ga_{m+1}].\]
\end{itemize}  
Recall the numbers $\gb_j$ introduced in \refS{SS:modified-fractions}. 

\smallskip 

{\em Step 2.} Let $k \geq j \geq n$ be integers. Either there is an eligible $m \in [j,k]$, or 
\begin{equation}\label{E:P:tiny-end-intervals}
x_j \in 
\begin{cases}
[0,\gb_k/\gb_j] \cup [1/\ga_j-1, 1/\ga_j-1+ \gb_k/\gb_j] & \tif e_j=-1 \\
[1/\ga_j-1-\gb_k/\gb_j, 1/\ga_j-1] \cup [1/\ga_j-\gb_k/\gb_j, 1/\ga_j]  & \tif e_j=+1. 
\end{cases}
\end{equation}

\smallskip 

We prove this by induction on $k-j$. 
Assume that $k-j=0$.  
If $e_j=+1$ then either $x_j \in [0, 1/\ga_j-2]$ and $j$ is eligible as in ($E_1$) and $(E_2)$, 
or $x_j \in [1/\ga_j-2, 1/\ga_j-1] \cup [1/\ga_j-1, 1/\ga_j]$ and \eqref{E:P:tiny-end-intervals} holds.
If $e_j=-1$ then either $x_j \in [1, 1/\ga_j-1]$ and $j$ is eligible as in ($E_3$) and $(E_4)$, or 
$x_j \in [0,1] \cup [1/\ga_j-1, 1/\ga_j]$ and \eqref{E:P:tiny-end-intervals} holds. 

Now assume that the statement holds for integers $k$ and $j$ with $k-j=t \geq 0$. 
We aim to prove it for integers $k$ and $j$ with $k-j=t+1$.
By the induction hypotheses applied to the pair $j+1$ and $k$ we conclude that either there is an eligible $m\in [j+1, k]$, 
or 
\begin{equation}\label{E:P:tiny-end-intervals-2}
x_{j+1} \in 
\begin{cases}
[0,\gb_k/\gb_{j+1}] \cup [1/\ga_{j+1}-1, 1/\ga_{j+1}-1+\gb_k/\gb_{j+1}] & \tif e_{j+1}=-1 \\
[1/\ga_{j+1}-1-\gb_k/\gb_{j+1}, 1/\ga_{j+1}-1] \cup [1/\ga_{j+1}-\gb_k/\gb_{j+1}, 1/\ga_{j+1}]  & \tif e_{j+1}=+1. 
\end{cases}
\end{equation}
If there is an eligible $m \in [j+1, k] \subset [j,k]$ then we are done. We show that if 
\eqref{E:P:tiny-end-intervals-2} holds, either $j$ is eligible or \eqref{E:P:tiny-end-intervals} holds. 
To prove this, we consider four cases based on the values of $e_j$ and $\gep_{j+1}$.

\smallskip

I) If $(e_j, \gep_{j+1})=(+1,-1)$, $e_{j+1}= -\gep_{j+1} e_j=+1$, and by \refE{E:P:tiny-end-intervals-2}, 
\[x_{j+1}\in [1/\ga_{j+1}-1-\gb_k/\gb_{j+1}, 1/\ga_{j+1}-1] \cup [1/\ga_{j+1}-\gb_k/\gb_{j+1}, 1/\ga_{j+1}].\] 
Since $x_m - l_m=-\gep_{j+1}\ga_{j+1}x_{j+1}= \ga_{j+1}x_{j+1}$, this implies that 
\[x_j \in \big([1-\ga_{j+1} - \gb_k/\gb_j, 1-\ga_{j+1}] \cup [1 - \gb_k/\gb_j, 1]\big)  +\D{Z}.\] 
If $x_j \in [0, 1/\ga_j-2]$ then $j$ is eligible through ($E_1$). 
If $x_j \in [1/\ga_j-2, 1/\ga_j]$, using $1/\ga_j=a_j-\ga_{j+1}$, $x_j$ must belong to one of the intervals 
\begin{gather*}
[a_j-2-\gb_k/\gb_j, a_j-2], \; [a_j-1- \gb_k/\gb_j, a_j-1] \\
[1/\ga_j-1-\gb_k/\gb_j, 1/\ga_j-1], \;  [1/\ga_j - \gb_k/\gb_j, 1/\ga_j].
\end{gather*}
If $x_j$ belongs to one of the bottom two intervals then \eqref{E:P:tiny-end-intervals} holds. 
If $x_j$ belongs to one of the top two interval, then $j$ is eligible though ($E_1$), since $[a_j-2-\gb_k/\gb_j, a_j-2]$ 
is contained in $[a_j-2-\ga_{j+1}, a_j-2]$ and $[a_j-1- \gb_k/\gb_j, a_j-1]$ is contained in $[a_j-1- \ga_{j+1}, a_j-1]$. 
Here we use $\gb_k/\gb_j\leq \ga_{j+1}$, which is valid due to $k\geq j+1$. 

\smallskip 

II) If $(e_j, \gep_{j+1})=(+1,+1)$, $e_{j+1}=-\gep_{j+1} e_j= -1$, and by \refE{E:P:tiny-end-intervals-2},
\[x_{j+1}\in [0,\gb_k/\gb_{j+1}] \cup [1/\ga_{j+1}-1, 1/\ga_{j+1}-1+\gb_k/\gb_{j+1}] .\] 
Since $x_j-l_j= -\gep_{j+1}\ga_{j+1} x_{j+1}=-\ga_{j+1} x_{j+1}$, this implies that 
\[x_j \in \big([-\gb_k/\gb_j,0] \cup [-1+\ga_{j+1} - \gb_k/\gb_j, -1+\ga_{j+1}]\big)+\D{Z}.\] 
If $x_j \in [0, 1/\ga_j-2]$ then $j$ is eligible though ($E_2$). 
If $x_j \in [1/\ga_j-2, 1/\ga_j]$, using $1/\ga_j= a_j + \ga_{j+1}$, $x_j$ must belong to one of the intervals 
\begin{gather*}
[a_j- \gb_k/\gb_j, a_j],  \; [a_j-1 - \gb_k/\gb_j, a_j-1] \\
[1/\ga_j-1 - \gb_k/\gb_j, 1/\ga_j-1], \; [1/\ga_j-\gb_k/\gb_j, 1/\ga_j].
\end{gather*}
If $x_j$ belongs to one of the bottom two intervals, then \eqref{E:P:tiny-end-intervals} holds. 
If $x_j$ belongs to one of the top two intervals then $j$ is eligible as in ($E_2$), since $[a_j- \gb_k/\gb_j, a_j]$ 
is contained in $[a_j- \ga_{j+1}, a_j]$ and $[a_j-1 - \gb_k/\gb_j, a_j-1]$ is contained in $[a_j-1 - \ga_{j+1}, a_j-1]$. 

\smallskip 

III) If $(e_j, \gep_{j+1})=(-1,-1)$, $e_{j+1}=-\gep_{j+1} e_j=-1$, and by \refE{E:P:tiny-end-intervals-2},
\[x_{j+1}\in [0,\gb_k/\gb_{j+1}] \cup [1/\ga_{j+1}-1, 1/\ga_{j+1}-1+\gb_k/\gb_{j+1}] .\] 
Since $x_j-l_j= -\gep_{j+1}\ga_{j+1} x_{j+1}=\ga_{j+1} x_{j+1}$, this implies that 
\[x_j \in \big([0,\gb_k/\gb_j] \cup [1-\ga_{j+1}, 1-\ga_{j+1} +\gb_k/\gb_j]\big)  +\D{Z}.\] 
If $x_j \in [1, 1/\ga_j-1]$ then $j$ is eligible through ($E_3$).  
If $x_j \in [0,1] \cup [1/\ga_j-1, 1/\ga_j]$, using $1/\ga_j=a_j-\ga_{j+1}$, $x_j$ must belong to one of the intervals 
\begin{gather*}
[0, \gb_k/\gb_j], [1/\ga_j-1, 1/\ga_j-1 + \gb_k/\gb_j] \\
[1-\ga_{j+1},1-\ga_{j+1} + \gb_k/\gb_j], [a_j-1, a_j-1+ \gb_k/\gb_j]. 
\end{gather*}
If $x_j$ belongs to one of the top two intervals, then \eqref{E:P:tiny-end-intervals} holds. 
If $x_j$ belongs to one of the bottom two interval then $j$ is eligible through ($E_3$) since 
$[1-\ga_{j+1},1-\ga_{j+1}+\gb_k/\gb_j]$ is contained in $[1-\ga_{j+1},1]$ 
and $[a_j-1, a_j-1+ \gb_k/\gb_j]$ is contained in $[a_j-1, a_j-1+ \ga_{j+1}]$. 

\smallskip 

IV) If $(e_j, \gep_{j+1})=(-1,+1)$, $e_{j+1}= - \gep_{j+1} e_j= +1$, and by \refE{E:P:tiny-end-intervals-2},
\[x_{j+1}\in [1/\ga_{j+1}-1-\gb_k/\gb_{j+1}, 1/\ga_{j+1}-1] \cup [1/\ga_{j+1}-\gb_k/\gb_{j+1}, 1/\ga_{j+1}].\] 
Since $x_j-l_j= -\gep_{j+1}\ga_{j+1} x_{j+1}=-\ga_{j+1} x_{j+1}$, this implies that 
\[x_j \in \big([-1+\ga_{j+1}, -1+\ga_j + \gb_k/\gb_j] \cup [-1, -1 + \gb_k/\gb_j]\big)  +\D{Z}.\] 
If $x_j \in [1, 1/\ga_j-1]$ then $j$ is eligible through ($E_4$). 
If $x_j \in [0,1] \cup [1/\ga_j-1, 1/\ga_j]$, using $1/\ga_j=a_j + \ga_{j+1}$, $x_j$ must belong to one of the intervals  
\begin{gather*}
[\ga_{j+1}, \ga_{j+1}+ \gb_k/\gb_j], \; [a_j,a_j+ \gb_k/\gb_j]\\
[0, \gb_k/\gb_j], \; [1/\ga_j-1, 1/\ga_j-1+ \gb_k/\gb_j]
\end{gather*}
If $x_j$ belongs to one of the bottom two intervals, then \eqref{E:P:tiny-end-intervals} holds. 
If $x_j$ belongs to one of the top two intervals then $j$ is eligible through ($E_4$), 
since $[\ga_{j+1}, \ga_{j+1}+ \gb_k/\gb_j]$ is contained in $[\ga_{j+1}, 2\ga_{j+1}]$ and $[a_j,a_j+ \gb_k/\gb_j]$ 
is contained in $[a_j, a_j+ \ga_{j+1}]$.
 
\smallskip

{\em Step 3.} Either there is $j \geq n$ such that for all $m \geq j$ we have $x_m=1/\ga_m-1$, or there are arbitrarily 
large eligible $m \geq n$.

\smallskip 

If there is $j \geq n$ such that $x_j \in \{1/\ga_j-1, 1/\ga_j\}$, then for all $m \geq j+1$ we have $x_m=1/\ga_m-1$. 
That is because, by the definition of the sequence $\{(x_m, l_m)\}$, if $x_j \in \{1/\ga_j-1, 1/\ga_j\}$ then 
$x_{j+1}\in \{1/\ga_{j+1}-1\}$. 
Therefore, for all $m \geq j+1$ we have $x_m \in \{1/\ga_m-1\}$.  

Assume that there is $j\geq n$ such that $x_j=0$. 
Recall from Step 1 that whenever $x_m=0$, $e_m=+1$. 
By the definition of the sequence $(x_m, l_m)$, if some $x_m=0$ then either 
$\gep_{m+1}= -1$ and hence $x_{m+1}=0$, or $\gep_{m+1}=+1$ and hence $x_{m+1}=1/\ga_{m+1}$. 
By the above paragraph, it follows that either eventually $x_m=0$, or eventually $x_m= 1/\ga_m-1$. 
When eventually $x_m=0$, all sufficiently large $m$ becomes eligible through ($E_1$) and ($E_2$). 

By the above paragraphs, if there is $j \geq n$ with $x_j \in \{0, 1/\ga_j-1, 1/\ga_j\}$, we are done. 
Below we assume that there are no such $j$. 

Fix an arbitrary $j \geq n$. Since $\gb_k/\gb_j \to 0$ as $k\to \infty$, there is $k \geq j$ such that 
$x_j \notin [0, \gb_k/\gb_j] \cup [1/\ga_j-1- \gb_k/\gb_j, 1/\ga_j-1 + \gb_k/\gb_j] \cup [1/\ga_j - \gb_k/\gb_j]$. 
It follows from Step 2 that there must be an eligible $j' \in [j, k]$. 
This proves that there are arbitrarily large eligible $m\geq n$. 

\smallskip 

{\em Step 4.} For every $\delta>0$ there is $x_n'$ strictly between $x_n$ and $x_n+ \delta e_n$ such that 
$b_n(x_n')\leq b_n(x_n)+\delta$.

\smallskip 

For $m\geq n$, let $z_m=x_m+ i b_m(x_m) \in I_m$, which is the lowest point on $I_m$ with real part 
equal to $x_m$. 
By Step 1, for $m\geq n$, we have $\Re \mathbb{Y}_{m+1}(x_{m+1}) + l_m= x_m$. 
It follows from the definition of the sets $I_n^j$ and $I_n$ that $\mathbb{Y}_{m+1}(z_{m+1})+l_m=z_m$. 

For $m\geq 0$, let $K_m= \cap _{j\geq 1} K_m^j$ and $J_m= \cap_{j\geq 1} J_m^j$.

Assume that $m\geq n$, and either $x_m=1/\ga_m-1$ or $m$ is eligible. 
We claim that there is $z_m'\in I_m$ satisfying the following properties: 
\begin{itemize}
\item[(i)] $\sign (\Re z'_m - \Re z_m)=e_m$, 
\item[(ii)] $|z'_m-z_m|\leq 2$, 
\item[(iii)] either both $z'_m$ and $z_m$ belong to $K_m$, or both $z'_m$ and $z_m$ belong to $J_m$. 
\end{itemize}

If $x_m = 1/\ga_m-1$ we simply let $z'_m=z_m+e_m$. By \refE{E:b_n^j-cont-periodic}, $z'_m \in I_m$. 
Moreover, $\{z_m, z'_m\} \subset K_m$ if $e_m=-1$ and $\{z_m, z'_m\} \subset J_m$ if $e_m=+1$. 
If $m$ is eligible, we show this by looking at cases $(E_1)$ through $(E_4)$ in the definition of eligibility. 
The main tool for each of those cases is the periodic property of $I_m$ and $I_{m+1}$ in \refE{E:b_n^j-cont-periodic}. 

$(E_1)$: If $x_m \in [0, 1/\ga_m-2]$ we let $z'_m=z_m+1=z_m+e_m$. Here $\{z_m ,z'_m\} \subset K_m$.

If $x_m \in [a_m-2-\ga_{m+1},a_m-2]= [1/\ga_m-2, a_m-2]$, then $x_{m+1}\in [1/\ga_{m+1}-1,1/\ga_{m+1}]$. 
Choose an integer $k_{m+1}$ with $x_{m+1}-k_{m+1} \in [0,1]$ and define 
$z'_m= \mathbb{Y}_{m+1}(z_{m+1}-k_{m+1})+ a_m-2$. 
Note that $\Re z_m \leq a_m-2 \leq \Re z'_m\leq a_m-2+\ga_{m+1} \leq 1/\ga_m-1$. 
This implies that $\sign (\Re z_m' -\Re z_m)=+1=e_m$, and $\{z_m, z'_m\} \subset K_m$. 
Here, $a_m = 1/\ga_m+\ga_{m+1}\geq 2+ \ga_{m+1}$, so $a_m\geq 3$. 
Using \refE{E:Y_n-comm-1} with $e_m=+1$ and \refE{E:uniform-contraction-Y},  
\begin{align*}
|z_m-z'_m| &= |(\mathbb{Y}_{m+1}(z_{m+1})+a_m-3) - (\mathbb{Y}_{m+1}(z_{m+1} -k_{m+1}) + a_m-2)| \\
&= |\mathbb{Y}_{m+1}(z_{m+1}) -1  - \mathbb{Y}_{m+1}(z_{m+1} -k_{m+1})| \\
&= |\mathbb{Y}_{m+1}(z_{m+1}-1/\ga_{m+1}) - \mathbb{Y}_{m+1}(z_{m+1} -k_{m+1})| \\
&\leq 0.9 |(z_{m+1} -1/\ga_{m+1}) - (z_{m+1}-k_{m+1})| \\ 
&= 0.9 |1/\ga_{m+1} -k_{m+1}| \leq 0.9 \cdot 2 \leq 2.  
\end{align*}

If $x_m \in [a_m-1-\ga_{m+1},a_m-1]= [1/\ga_m-1, a_m-1]$ we let $z'_m= \mathbb{Y}_{m+1}(z_{m+1}-k_{m+1})+ a_m-1$. 
In this case, $\Re z_m \leq a_m-1 \leq \Re z'_m \leq a_m-1 +\ga_{m+1} \leq 1/\ga_m$. 
Hence, $\sign (\Re z_m' -\Re z_m)=+1=e_m$, and $\{z_m, z'_m\} \subset J_m$. 
As in the previous case, $|z_m-z'_m|\leq 2$. 

$(E_2)$: If $x_m \in [0, 1/\ga_m-2]$ we let $z'_m=z_m+1=z_m+e_m$.
If $x_m$ belongs to $[a_m-1-\ga_{m+1}, a_m-1]$, then $x_{m+1} \in [0,1]$. 
Choose an integer $k_{m+1}$ with $x_{m+1}+k_{m+1} \in [1/\ga_{m+1}-1, 1/\ga_{m+1}]$ and define 
$z'_m = \mathbb{Y}_{m+1}(z_{m+1}+ k_{m+1})+ a_m$. 
If $x_m$ belongs to $[a_m-\ga_{m+1}, a_m]$, we let $z'_m = \mathbb{Y}_{m+1}(z_{m+1}+ k_{m+1})+ a_m+1$. 
As in the previous case, one can see that $z'_m$ enjoys the desired properties. 

$(E_3)$: If $x_m$ belongs to $[1, 1/\ga_m-1]$, we let $z'_m=z_m-1=z_m+e_m$.   
If $x_m$  belongs to $[1-\ga_{m+1},1]$, then $x_{m+1} \in [1/\ga_{m+1}-1, 1/\ga_{m+1}]$. 
We define $z'_m= \mathbb{Y}_{m+1}(z_{m+1}-1)$. 
If $x_m$ belongs to $[a_m-1, a_m-1+\ga_{m+1}]$ then $x_{m+1}\in [0,1]$. 
We choose an integer $k_{m+1}$ with $x_{m+1}+k_{m+1}\in [1/\ga_{m+1}-1, 1/\ga_{m+1}]$ and define 
$z'_m= \mathbb{Y}_{m+1}(z_{m+1}+k_{m+1})+ a_m-2$. 
One can see that $z'_m$ enjoys the desired properties. 

$(E_4)$: If $x_m$ belongs to $[1, 1/\ga_m-1]$, we let $z'_m=z_m-1=z_m+e_m$.   
If $x_m$ belongs to $[\ga_{m+1}, 2\ga_{m+1}]$  then $x_{m+1}\in [1/\ga_{m+1}-2,1/\ga_{m+1}-1]$. 
We define $z'_m=\mathbb{Y}_{m+1}(z_{m+1}+1)+1$. 
If $x_m$ belongs to $[a_m, a_m+\ga_{m+1}]=[a_m, 1/\ga_m]$, then $x_{m+1} \in [1/\ga_{m+1}-1, 1/\ga_{m+1}]$. 
Choose an integer $k_{m+1}$ with $x_{m+1}-k_{m+1}\in [0,1]$ and define 
$z'_m= \mathbb{Y}_{m+1}(z_{m+1}-k_{m+1})+ a_m$. 
One can see that $z'_m$ enjoys the desired properties. 
This completes the proof of the existence of $z'_m$. 

By the definition of the domains $I_n^j$, and \refE{E:P:bounds-on-l_m}, for all $m\geq n$ 
\begin{itemize}
\item either $\mathbb{Y}_{m+1}(K_{m+1})+l_m \subset K_m$ or $\mathbb{Y}_{m+1}(K_{m+1})+l_m \subset J_m$,
\item either $\mathbb{Y}_{m+1}(J_{m+1})+l_m \subset K_m$ or $\mathbb{Y}_{m+1}(J_{m+1})+l_m \subset J_m$. 
\end{itemize}
It follows from the above properties that any composition of the form $(\mathbb{Y}_{n+1}+l_n) \circ \dots \circ (\mathbb{Y}_{m+1}+l_m)$ 
is defined and continuous on both $K_m$ and $J_m$.

Fix an arbitrary $\delta>0$ and choose $n' \geq n$ such that $2 \cdot 0.9^{(n'-n)}< \delta$. 
By Step 3 there is $m \geq n'$ such that either $x_m=1/\ga_m-1$ or $m$ is eligible. 
Then, by the above argument, there is $z'_m$ satisfying the three items listed above. 
Let us define $z_n^m= (\mathbb{Y}_{n+1}+l_n) \circ \dots \circ (\mathbb{Y}_{m+1}+l_m)(z'_m) \in I_n$. 
By \refE{E:uniform-contraction-Y}, we have $|z_n^m -z_n|\leq 2 \cdot 0.9^{m-n}< \delta$. 
Let $x'_n =x_n^m= \Re z_n^m$. We have $|x'_n - x_n| = |\Re z_n^m - \Re z_n| < \delta$ and 
$b_n(x_n^m)- b_n(x_n) \leq \Im z_n^m-\Im z_n < \delta$. 
Moreover, the relations $\sign (\Re z'_m-\Re z_m)=e_m$, $e_{m-1}= -\gep_{m}e_{m}$, and 
$\Re \mathbb{Y}_{m}(x)= -\gep_{m} \ga_{m} x$ imply that $\sign (x_n^m-x_n)= \sign (\Re x_n^m - \Re z_n)=e_n$. 
\end{proof}

\subsection{The Brjuno condition in the renormalisation tower}\label{SS:Brjuno-Tower}
\begin{propo}\label{P:b_n-sup-B}
For all $\ga \in \E{B}$ and all $n\geq -1$ we have   
\[\Big|2 \pi  \sup_{x\in [0, 1/\ga_n]} b_n(x) - \C{B}(\ga_{n+1})\Big | \leq  5 \pi.\] 
\end{propo}

\begin{proof}
For $n\geq -1$ and $j\geq 0$ we define
\[D_n^j= \max \{b_n^j(x) \mid x\in [0, 1/\ga_n]\}.\]
We first show that the numbers $D_n^j$ nearly satisfy the recursive relation for the 
Brjuno function, see \refE{E:Brjuno-functional-equations}. That is, for all $n\geq 0$ and all $j\geq 1$, we have 
\begin{equation}\label{E:L:max-relations}
\big |2\pi D_{n-1}^{j} - 2\pi \ga_n D_{n}^{j-1} - \log (1/\ga_n) \big | \leq 4.
\end{equation}
Since $b_{n-1}^j$ and $b_n^{j-1}$ are periodic of period $+1$, we may choose 
$x_{n-1} \in [1/(2\ga_{n-1})-1,1/(2\ga_{n-1})]$ and $x_n \in [1/(2\ga_n)-1,1/(2\ga_n)]$ such that 
$b_{n-1}^j(x_{n-1})=D_{n-1}^j$ and $b_{n}^{j-1}(x_{n})=D_{n}^{j-1}$. 
Choose $x_n' \in [0, 1/\ga_n]$ such that $-\gep_n \ga_n x'_n \in x_{n-1}+\D{Z}$. 
By \refL{L:Y_n-on-horizontals}-(i), we must have $x_n' \in [1/(2\ga_n)-1,1/(2\ga_n)]$.
We apply \ref{L:Y_n-on-horizontals}-(ii) with $y=b_n^{j-1}(x_n')$ and $x=x_n'$, to obtain 
\begin{align*}
2\pi \ga_n D_{n}^{j-1}+ \log 1/\ga_n = 2\pi \ga_n b_{n}^{j-1}(x_n) + \log 1/\ga_n 
& \geq  2\pi \ga_n b_{n}^{j-1}(x'_n) + \log 1/\ga_n \\
& \geq 2\pi \Im \mathbb{Y}_{n}(x_n'+i b_n^{j-1}(x_n')) - 2\\
&= 2\pi b_{n-1}^j(x_{n-1}) - 2= 2\pi D_{n-1}^j -2. 
\end{align*}  
Similarly, 
\begin{align*}
2\pi \ga_n D_n^{j-1} + \log 1/\ga_n  = 2\pi \ga_n b_{n}^{j-1}(x_n) + \log 1/\ga_n 
& \leq 2\pi \Im(x_n+ ib_n^{j-1}(x_n)) + 4 \\
& \leq 2\pi b_{n-1}^j(x_{n-1}) + 4 = 2\pi D_{n-1}^j+ 4. 
\end{align*}  
This completes the proof of inequality \eqref{E:L:max-relations}. 

Fix an arbitrary $n\geq -1$ and $j\geq 1$. Let us define $\gb_{n+i}(\ga_{n+1}) = \prod_{l=1}^{i} \ga_{n+l}$ for $i\geq 1$, 
and $\gb_{n}(\ga_{n+1})=1$. 
Then for integers $k\in [0, j]$ define the numbers 
\[X_k= 2\pi \gb_{n+k}(\ga_{n+1}) D_{n+k}^{j-k} + \sum_{i=1}^{k} \gb_{n+i-1}(\ga_{n+1}) \log 1/\ga_{n+i}.\]
We have $X_{0}= 2 \pi D_{n}^{j}$. 
With this notation, we form a telescoping sum 
\[2\pi D_n^j = \sum_{k=0}^{j-1} (X_{k} - X_{k+1}) + X_j.\]
By \refE{E:L:max-relations}, 
\[|X_k-X_{k+1}|
=\big |\gb_{n+k}(\ga_{n+1})\big (2\pi D_{n+k}^{j-k}-2\pi \ga_{n+k+1} D_{n+k+1}^{j-k-1}-\log (1/\ga_{n+k+1})\big)\big| 
\leq \gb_{n+k}(\ga_{n+1}) 4. \]
On the other hand, since $D_{n+j}^0=-1$, we have 
\[\big |X_j - \sum_{i=1}^{j} \gb_{n+i-1}(\ga_{n+1}) \log 1/\ga_{n+i} \big | 
= |2\pi \gb_{n+j}(\ga_{n+1})D_{n+j}^0| \leq 2\pi.\]
Combining the above inequalities, and using \refE{E:rotations-rest}, we conclude that 
\begin{align*}
\Big| 2\pi D_n^j- \sum_{i=1}^{j} \gb_{n+i-1}(\ga_{n+1}) \log 1/\ga_{n+i} \Big | 
& \leq \sum_{k=0}^{j-1} \gb_{n+k}(\ga_{n+1})4 + 2\pi \\
& \leq \sum_{k=0}^{j-1} 2^{-k} 4 + 2\pi \leq 8 +2\pi.
\end{align*}

By \refE{E:I_n^j-forms-nest}, $b_n^j \geq b_n^{j-1}$, which implies that $D_{n}^{j}\geq D_{n}^{j-1}$. 
Hence, for each $n\geq -1$, $D_{n}^{j}$ forms an increasing sequence. Combining with the above inequality,  
\[\Big | 2\pi \lim_{j\to +\infty} D_n^j - \C{B}(\ga_{n+1})\Big | \leq 8 +2\pi.\]
Note that because $b_n^j \leq b_n$ and $b_n^j \to b_n$ point-wise, we must have 
$\sup_{x\in [0, 1/\ga_n]} b_n=\lim_{j\to +\infty} D_n^j$. 
\end{proof}
 
\subsection{The Herman condition in the renormalisation tower}\label{SS:Herman-tower}
In this section we establish an equivalent criterion for the arithmetic class $\E{H}$ in terms of the maps $\mathbb{Y}_n$. 
The key idea here is that the equivalent criterion in \refP{P:Herman-Yoccoz-criterion} is stable under uniform changes 
to the maps $h_{\ga_n}$. That is, if one replaces $h_{\ga_n}$ by uniformly nearby maps, say $\mathbb{Y}_n^{-1}$, 
the corresponding set of rotation numbers stays the same. 
See \refP{P:Y_r-vs-h_r^-1}. 

Note that for arbitrary $m>n\geq 0$ and $y\geq 0$, the compositions 
$h_{\ga_n}^{-1} \circ  \dots  \circ  h_{\ga_m}^{-1}(y)$ may not be defined. This happens when an intermediate 
iterate falls into $(-\infty, 0]$. 

\begin{lem}\label{L:herman-Y-iterates-close}
Assume that for some integers $m > n\geq 0$, and $y \in (1, +\infty)$, the composition
$h_{\ga_n}^{-1} \circ  \dots  \circ  h_{\ga_m}^{-1}(y)$ is defined and is positive. Then, 
\[\big| 2\pi \Im \left( \mathbb{Y}_{n} \circ \dots \circ \mathbb{Y}_m( i  y/(2\pi))\right) - h_{\ga_n}^{-1} \circ  \dots  \circ  h_{\ga_m}^{-1}(y)\big| 
\leq 10 \pi.\]
\end{lem}

\begin{proof}
For integers $j$ with $n \leq j \leq m-1$, we may introduce 
\begin{gather*}
G_{m,j-1}(y)=h_{\ga_j}^{-1} \circ \dots \circ  h_{\ga_{m}}^{-1}(y), G_{m,m-1}(y) = h_{\ga_m}^{-1}(y), G_{m, m}(y)= y. 
\end{gather*}
By the assumptions in the lemma, all the above values are positive. 
Also, for integers $j$ with $n+1 \leq j \leq m$, and $t\geq -1$, we may introduce the maps 
\begin{gather*}
\gY_{j,n-1}(t)= 2\pi \Im \mathbb{Y}_{n} \circ \dots \circ \mathbb{Y}_{j}(i t/(2\pi)),
\gY_{n,n-1}(t)=2\pi \Im \mathbb{Y}_n(i t/(2\pi)), \gY_{n-1, n-1}(t)=t.
\end{gather*}
With the above notations, we may form a telescoping sum, to obtain 
\begin{align*}
\big|h_{\ga_n}^{-1} \circ \dots  \circ  h_{\ga_m}^{-1}(y) & - 2\pi \Im \mathbb{Y}_{n} \circ \dots \circ \mathbb{Y}_m(i y/(2\pi))\big| \\
&= \big|G_{m,n-1}(y)- \gY_{m,n-1}(y) \big | \\ 
&= \Big | \sum_{j=n}^{m} \gY_{j-1,n-1}(G_{m,j-1}(y)) - \gY_{j, n-1}(G_{m,j}(y)) \Big| 
\end{align*}
By \refE{E:uniform-contraction-Y}, for all $s$ and $t$ in $(0, +\infty)$, 
$|\gY_{j-1,n-1}(s) - \gY_{j-1, n-1}(t)\big| \leq 0.9^{(j-n)} \cdot |s-t|$. 
Thus, 
\begin{align*}
\big |\gY_{j-1,n-1}(G_{m,j-1}(y))-&\gY_{j, n-1}(G_{m,j}(y))\big|  \\
&= \big |\gY_{j-1,n-1}(G_{m,j-1}(y)) - \gY_{j-1, n-1}(-2\pi i \mathbb{Y}_j (i G_{m,j}(y)/(2\pi)))\big| \\
& \leq 0.9^{(j-n)} \big |G_{m,j-1}(y) + 2\pi i \mathbb{Y}_j (i G_{m,j}(y)/(2\pi)) \big| 
\end{align*}
On the other hand, by \refP{P:Y_r-vs-h_r^-1} and \refE{E:Y_n}, for all $j \geq 0$ and all $t \geq 1$, we have 
\[|h_{\ga_j}^{-1}(t) - 2 \pi \Im \mathbb{Y}_j(i t/(2\pi))| \leq  \pi.\] 
Also, note that since $G_{m,n-1}(y)>0$ and $h_{\ga_n}(0)=1$, $G_{m, j}(y)>1$, for $n\leq j \leq m$. 
Therefore, 
\begin{align*}
\big |G_{m,j-1}(y) + 2\pi i \mathbb{Y}_j (i G_{m,j}(y)/(2\pi)) \big| 
=\big |h_{\ga_j}^{-1} (G_{m,j}(y)) - 2\pi\Im \mathbb{Y}_j (i G_{m,j}(y)/(2\pi))\big| \leq \pi. 
\end{align*}
Combining the above inequalities together, we obtain
\begin{align*}
\big|h_{\ga_n}^{-1} \circ \dots  \circ  h_{\ga_m}^{-1}(y) & - 2\pi \Im \mathbb{Y}_{n} \circ \dots \circ \mathbb{Y}_m(i y/(2\pi))\big| 
\leq  \sum_{j=n}^m 0.9^{(j-n)} \pi  
\leq 10\pi.\qedhere
\end{align*}
\end{proof}

\begin{propo}\label{P:Herman-Y-iterates}
An irrational number $\ga$ belongs to $\E{H}$, if and only if, for all $x>0$ there is $m \geq 1$ such that 
\[\Im \mathbb{Y}_0 \circ  \dots  \circ \mathbb{Y}_{m-1}(i \C{B}(\ga_{m})/(2\pi)) \leq x.\]
\end{propo}

\begin{proof}
Fix an arbitrary $\ga \in \E{H}$ and $x >0$. 
Choose $n\geq 0$ such that $0.9^n(10 \pi + 1)/(2\pi) \leq x$. 
By the criterion in \refP{P:Herman-Yoccoz-criterion}, there is $m\geq n$ such that  
\begin{equation}\label{E:P-Herman-Y-iterates}
h_{\ga_{m-1}}\circ \cdots \circ h_{\ga_n} (0) \geq \C{B}(\ga_m).
\end{equation}
Now we consider two cases. First assume that $\C{B}(\ga_m)\leq 1$. 
Recall that $\C{B}(\ga_m)>0$. By \eqref{E:invariant-imaginary-line} and \eqref{E:uniform-contraction-Y}, 
\[\Im \mathbb{Y}_0 \circ  \dots  \circ \mathbb{Y}_{m-1}(i \C{B}(\ga_{m})/(2\pi)) \leq 0.9^{m} \cdot 1/ (2\pi) \leq 0.9^{n}/(2\pi) \leq x.\] 
Thus, we have the desired inequality in the proposition in this case. 

Now assume that $\C{B}(\ga_m) > 1$. The composition of the maps in \refE{E:P-Herman-Y-iterates} is understood 
as the identity map when $m=n$, and as the map $h_{\ga_n}$ when $m=n+1$. Also recall that $h_{\ga_n}(0)=1$. 
Then, in this case, we must have $n< m-1$. 
Let $k \in [n, m-1]$ be the smallest integer such that 
\[h_{\ga_k}^{-1} \circ  \dots  \circ  h_{\ga_{m-1}}^{-1}(\C{B}(\ga_{m}))\]
is defined and is positive. By \refE{E:P-Herman-Y-iterates}, $k$ exists and $k>n$. Moreover, by the minimality of 
$k$, we must have $h_{\ga_k}^{-1} \circ  \dots  \circ  h_{\ga_{m-1}}^{-1}(\C{B}(\ga_{m})) \leq +1$.
Now, we may use \refL{L:herman-Y-iterates-close} with $y=\C{B}(\ga_m)$ and 
$h_{\ga_k}^{-1} \circ  \dots  \circ  h_{\ga_{m-1}}^{-1}(\C{B}(\ga_{m}))$ to conclude that 
\begin{align*}
|2\pi \Im & \, \mathbb{Y}_k \circ  \dots  \circ  \mathbb{Y}_{m-1}(i \C{B}(\ga_{m})/(2\pi))|  \\
& \qquad \leq |2\pi \Im \mathbb{Y}_k \circ  \dots  \circ  \mathbb{Y}_{m-1}(i \C{B}(\ga_{m})/(2\pi)) 
- h_{\ga_k}^{-1} \circ  \dots  \circ  h_{\ga_{m-1}}^{-1}(\C{B}(\ga_{m}))| \\
& \qquad \qquad + |h_{\ga_k}^{-1} \circ  \dots  \circ  h_{\ga_{m-1}}^{-1}(\C{B}(\ga_{m}))| \\
& \qquad \leq 10\pi+1.
\end{align*}
Then, by \eqref{E:invariant-imaginary-line} and \eqref{E:uniform-contraction-Y}, we obtain 
\begin{align*}
\Im \mathbb{Y}_0 \circ  \dots  \circ \mathbb{Y}_{k-1} \circ \mathbb{Y}_k \circ \cdots \circ \mathbb{Y}_{m-1}(i \C{B}(\ga_{m})/(2\pi))
& \leq \mathbb{Y}_0 \circ  \dots  \circ \mathbb{Y}_{k-1} ((10\pi+1)/(2\pi)) \\
&\leq 0.9^k  \frac{10 \pi+1}{2\pi} \leq 0.9^n \frac{10 \pi +1}{2\pi} \leq x.
\end{align*}
Thus, the desired inequality in the proposition also holds in this case. 

To prove the other direction of the proposition, fix an arbitrary $n\geq 0$. 
We shall prove that there is $m\geq n$ satisfying the inequality in \refP{P:Herman-Yoccoz-criterion}. 

First note that for all $j\geq 0$ and $y\geq 0$, $h_{\ga_j}^{-1}(y)\leq y-1$. This implies that there are $k > l \geq n$ 
such that 
\begin{equation}\label{E:herm-equiv-k}
h_{\ga_l}^{-1} \circ \dots \circ h_{\ga_k}^{-1}(12 \pi)\leq 0.
\end{equation}
In particular, the composition in the above equation is defined. Note that in general one cannot choose $l=n$. 
Now, choose $x>0$ such that  
\begin{equation}\label{E:herm-equiv-1}
x < \Im \mathbb{Y}_0 \circ \dots \circ \mathbb{Y}_k(i),
\end{equation} 
and 
\begin{equation}\label{E:herm-equiv-2}
x < \min \big\{\Im \mathbb{Y}_0 \circ \dots \circ \mathbb{Y}_j(i \C{B}(\ga_{j+1})/(2\pi))\mid \forall j\in [0, k-1] \cap \mathbb{Z} \big\},
\end{equation} 
By the hypothesis in the proposition, there is $m\geq 1$ such that 
\begin{equation}\label{E:P:herm-equiv-second-side}
\Im \mathbb{Y}_0 \circ  \dots  \circ \mathbb{Y}_{m-1}(i \C{B}(\ga_{m})/(2\pi)) \leq x.
\end{equation}
Note that by \refE{E:herm-equiv-2} we must have $m\geq k+1$. In particular, $m\geq n$.  

By \refE{E:herm-equiv-1}, we must have 
\begin{equation}\label{E:P:herman-equiv-second-side-truncate}
\Im \mathbb{Y}_{k+1} \circ  \dots  \circ \mathbb{Y}_{m-1}(i \C{B}(\ga_{m})/(2\pi)) \leq 1.
\end{equation}
Otherwise, by the injectivity of the maps $\mathbb{Y}_j$ and \refE{E:invariant-imaginary-line}, 
\begin{align*}
\Im \mathbb{Y}_0 \circ  \dots  \circ \mathbb{Y}_{m-1}(i \C{B}(\ga_{m})/(2\pi)) 
&=\Im \mathbb{Y}_0 \circ  \dots  \circ \mathbb{Y}_k (\mathbb{Y}_{k+1} \circ \dots \circ \mathbb{Y}_{m-1}(i \C{B}(\ga_{m})/(2\pi)) \\
&> \Im \mathbb{Y}_0 \circ  \dots  \circ \mathbb{Y}_k (i) > x,
\end{align*}
which contradicts \refE{E:P:herm-equiv-second-side}.

Now we consider two cases. 
First assume that $\C{B}(\ga_m)> 1$ and $h_{\ga_{k+1}}^{-1} \circ\dots\circ h_{\ga_{m-1}}^{-1}(\C{B}(\ga_{m}))$
is defined. Here, we may apply \refL{L:herman-Y-iterates-close} to this composition with $y=\C{B}(\ga_m)$, and use 
\refE{E:P:herman-equiv-second-side-truncate}, to get 
\begin{align*}
h_{\ga_{k+1}}^{-1} \circ  \dots &  \circ h_{\ga_{m-1}}^{-1}(\C{B}(\ga_{m}))  \\
& \leq |h_{\ga_{k+1}}^{-1} \circ  \dots  \circ h_{\ga_{m-1}}^{-1}(\C{B}(\ga_{m})) 
- 2\pi \Im \mathbb{Y}_{k+1} \circ  \dots  \circ \mathbb{Y}_{m-1}(i \C{B}(\ga_{m})/(2\pi)) |  \\
& \qquad \qquad + |2\pi \Im \mathbb{Y}_{k+1} \circ  \dots  \circ \mathbb{Y}_{m-1}(i \C{B}(\ga_{m})/(2\pi))| \\
&\leq 10 \pi + 2\pi.
\end{align*}
Combining this with \refE{E:herm-equiv-k}, and using the monotonicity of the maps $h_{\ga_j}$, we conclude that 
there is $l' \in [l, k]$ such that
\begin{align*}
h_{\ga_{l'}}^{-1} \circ  \dots  \circ h_{\ga_{m-1}}^{-1}(\C{B}(\ga_{m}))
& = h_{\ga_{l'}}^{-1} \circ  \dots \circ h_{\ga_k}^{-1} 
(h_{\ga_{k+1}}^{-1}\circ\cdots \circ h_{\ga_{m-1}}^{-1}(\C{B}(\ga_{m}))) \\
&\leq h_{\ga_{l'}}^{-1} \circ  \dots \circ h_{\ga_k}^{-1} (12 \pi) < 0.
\end{align*}
Here we need to choose $l'\geq l$ so that the compositions in the above equation are defined. 
The above inequality implies that $\C{B}(\ga_m)< h_{\ga_{m-1}} \circ \dots \circ h_{\ga_{l'}}(0)$. 
Note that $l'\geq l \geq n$. 
On the other hand, since $h_{\ga_j}(y)\geq y+1$ for all $j$ and $y > 0$, we must have 
$h_{\ga_{l'-1}} \circ \cdots \circ h_{\ga_n}(0)>0$. Therefore, 
\[\C{B}(\ga_m)< h_{\ga_{m-1}} \circ \dots \circ h_{\ga_{l'}}(0)
< h_{\ga_{m-1}} \circ \dots \circ h_{\ga_n}(0).\] 
This completes the argument in this case.  

Now assume that either $\C{B}(\ga_m)\leq 1$, or 
$h_{\ga_{k+1}}^{-1} \circ\dots\circ h_{\ga_{m-1}}^{-1}(\C{B}(\ga_{m}))$ is not defined. 
These imply that there is $j$ in $[k+2, m-1]$ such that 
$h_{\ga_{j}}^{-1} \circ\dots\circ h_{\ga_{m-1}}^{-1}(\C{B}(\ga_{m}))\leq 0$, and hence  
$\C{B}(\ga_m) \leq h_{\ga_{m-1}} \circ \dots \circ h_{\ga_{j}}(0)$. 
Note that $j \geq k+2 \geq n+2$. 
As in the previous case, $h_{\ga_{j-1}} \circ \cdots \circ h_{\ga_n}(0)>0$. Therefore, 
\[\C{B}(\ga_m)< h_{\ga_{m-1}} \circ \dots \circ h_{\ga_j}(0)
< h_{\ga_{m-1}} \circ \dots \circ h_{\ga_n}(0).\] 
This completes the argument in this case.  
\end{proof}
 
\begin{propo}\label{P:optimal-herman-p_n-b_n}
Assume that $\ga \in \E{B}$. Then, $\ga\in \E{H}$ if and only if $p_n(0)=b_n(0)=0$ for all $n\geq -1$. 
\end{propo} 
 
\begin{proof}
Recall that $p_n(0)=p_n(1/\ga_n)$ and $b_n(0)=b_n(1/\ga_n)$, for all $n\geq -1$. Also, $\mathbb{Y}_{n+1}$ maps the graphs of 
$p_{n+1}$ and $b_{n+1}$ to the graphs of $p_n$ and $b_n$, respectively. These imply that $p_{n+1}(0)=b_{n+1}(0)$
if and only if $p_n(0)=b_n(0)$, for all $n\geq -1$. 
Therefore, to prove the proposition, it is enough to show that $\ga\in \E{H}$ if and only if $p_{-1}(0)=b_{-1}(0)=0$. 

Assume that $\ga\in \E{H}$. Fix an arbitrary $\gep>0$. Choose $m_0 \geq 1$ satisfying $5 \cdot 0.9^{m_0}\leq \gep$. 
Let $\gep' \leq \gep/2$ be a positive constant. We may apply \refP{P:Herman-Y-iterates} with $x=\gep'$ and obtain 
$m\geq 1$ satisfying the inequality in that proposition. 
By making $\gep'$ small enough, we may make $m \geq m_0$. 
Now, using the uniform contraction of the maps $\mathbb{Y}_n$ in \refE{E:uniform-contraction-Y}, 
\begin{align*}
p_{-1}(0) &  \leq p_{-1}^m(0) \\
& = \Im \mathbb{Y}_0 \circ \dots \circ \mathbb{Y}_{m-1}(ip_{m-1}^0(0)) \\
&= \Im \mathbb{Y}_0 \circ \dots \circ \mathbb{Y}_{m-1}(i\C{B}(\ga_m)/(2\pi)+5/2)
- \Im \mathbb{Y}_0 \circ \dots \circ \mathbb{Y}_{m-1}(i\C{B}(\ga_m)/(2\pi)) \\
& \qquad \qquad + \Im \mathbb{Y}_0 \circ \dots \circ \mathbb{Y}_{m-1}(i\C{B}(\ga_m)/(2\pi)) \\
&\leq 0.9^m 5/2 + \gep' \leq \gep.   
\end{align*}
That is, $b_{-1}(0)=0\leq p_{-1}(0) \leq \gep$, for all $\gep > 0$. This implies that $b_{-1}(0)=p_{-1}(0)$. 

Now assume that $b_{-1}(0)=p_{-1}(0)=0$. Fix $x >0$. Since $p_{-1}^m(0) \to p_{-1}(0)$, as  $m\to +\infty$, there 
is $m\geq 1$ such that $p_{-1}^m(0)< x$. 
Then, by the monotonicity of the maps $t \mapsto \Im \mathbb{Y}_l(it)$,  
\[\Im \mathbb{Y}_0 \circ \dots \circ \mathbb{Y}_{m-1}(i \C{B}(\ga_m)/(2\pi)) 
\leq \Im \mathbb{Y}_0 \circ \dots \circ \mathbb{Y}_{m-1}(i\C{B}(\ga_m)/(2\pi)+5/2)
= p_{-1}^m(0) \leq x.    
\]
By \refP{P:Herman-Y-iterates}, this implies that $\ga \in \E{H}$. 
\end{proof} 
 
\subsection{Hairs, or no hairs} 
 
\begin{propo}\label{P:b_n-p_n-dense-touches}
For every $n\geq -1$, the following properties hold.
\begin{itemize}
\item[(i)] If $\ga \in \E{B}\setminus \E{H}$, $b_n=p_n$ on a dense subset of $[0, 1/\ga_n]$ and
$b_n< p_n$ on a dense subset of $[0, 1/\ga_n]$;
\item[(ii)] If $\ga \notin \E{B}$, $b_n=+\infty$ on a dense subset of $[0, 1/\ga_n]$ 
and $b_n < +\infty$ on a dense subset of $[0,1/\ga_n]$.
\end{itemize}
\end{propo}

\begin{proof}
Each of the strict inequalities in items (i) and (ii) hold at least at one point. That is because when 
$\ga \in \E{B} \setminus \E{H}$, by \refP{P:optimal-herman-p_n-b_n}, we have $p_n(0) \neq b_n(0)$, and
when $\ga \notin \E{B}$, $b_n(0) = 0 < +\infty$. 
The main task is to show that each of the equalities in (i) and (ii) hold at least at one point. 
To that end we show that there is $x_n \in [0,1/\ga_n]$ such that if $\ga \in \E{B}\setminus \E{H}$ we have 
$p_n(x_n)=b_n(x_n)$ and if $\ga \notin \E{B}$ we have $b_n(x_n)=+\infty$. 
There is an algorithm to identify the points $x_n$, as we explain below. 

There are $x_n\in [0, 1/\ga_n]$ and $l_n \in \D{Z}$, for $n\geq -1$, such that 
\[x_{n+1} \in [1/(2\ga_{n+1})-1, 1/(2\ga_{n+1})],  \quad -\gep_{n+1}\ga_{n+1} x_{n+1} \in x_n - l_n.\]

By \refL{L:Y_n-on-horizontals}-(ii), for all $n\geq -1$ and all $j\geq 0$ we have 
\[| 2\pi b_{n-1}^j(x_{n-1}) - 2 \pi \ga_n b_n^{j-1}(x_n) - \log (1/\ga_n)| \leq 4.\]
One may literally repeat the latter part of the proof of \refP{P:b_n-sup-B} with $D_n^j=b_n^j(x_n)$ to conclude that 
for all $n\geq -1$, 
\begin{equation}\label{E:P:b_n-p_n-dense-touches-1}
\big |2\pi b_{n}(x_n) - \C{B}(\ga_{n+1})\big | \leq 8+2\pi.
\end{equation}

Now, if $\ga \notin \E{B}$ we must have $b_n(x_n)=+\infty$. 
Assume that $\ga \in \E{B}$. We have
\begin{align*}
p_n^{j}(x_n) &= \Im (\mathbb{Y}_n+l_n) \circ \dots \circ (\mathbb{Y}_{n+j}(x_{n+j}+ ip_{n+j}^0(x_{n+j}))+l_{n+j-1}) \\
&= \Im (\mathbb{Y}_n+l_n) \circ \dots \circ (\mathbb{Y}_{n+j}(x_{n+j}+ i\C{B}(\ga_{n+j+1})/(2\pi)+5/2)+l_{n+j-1}), 
\end{align*} 
and 
\begin{align*}
b_n(x_n) = \Im (\mathbb{Y}_n+l_n) \circ \dots \circ (\mathbb{Y}_{n+j}(x_{n+j}+ i b_{n+j}(x_{n+j}))+l_{n+j-1}). 
\end{align*} 
Therefore, by \eqref{E:uniform-contraction-Y} and \eqref{E:P:b_n-p_n-dense-touches-1}, 
\begin{align*}
p_n^{j}(& x_n) - b_n(x_n) \leq 0.9^{j} \Big(\frac{5}{2}+ \frac{8+2\pi}{2\pi}\Big). 
\end{align*} 
Since $p_n \geq b_n$ on $[0,1/\ga_n]$, we conclude that $p_n(x_n)=\lim_{j\to +\infty} p_n^j(x_n)=b_n(x_n)$. 

To discuss items (i) and (ii) in the proposition at once, let us define $p_n \equiv +\infty$, for all $n\geq -1$, 
when $\ga \notin \E{B}$. 

Recall that the graphs of $p_n$ and $b_n$ are obtained from the graphs of $p_{n+1}$ and $b_{n+1}$, 
respectively, using the map $\mathbb{Y}_{n+1}$ and its integer translations. 
Thus, if $p_n=b_n$ at some $y \in [0, 1/\ga_n]$, then $p_{n-1}=b_{n-1}$ at $-\gep_n\ga_n y + (1+ \gep_n)/2$. 
Similarly, if $p_n\neq b_n$ at some $y\in [0,1/\ga_n]$, then $p_{n-1}\neq b_{n-1}$ at $-\gep_n\ga_n y+(1+\gep_n)/2$. 
On the other hand, $p_n$ and $b_n$ are periodic of period $+1$, for all $n\geq -1$. 
One infers from these properties, and the first part of the proof, that $p_n=b_n$ on a dense subset of $[0, 1/\ga_n]$, 
and $p_n\neq b_n$ on a dense subset of $[0, 1/\ga_n]$, for all $n\geq -1$.
\end{proof} 

\begin{propo}  \label{P:p_n-cont} \label{P:b_n-p_n-identical-or-not}
Assume that $\ga \in \E{B}$. For all $n\geq -1$, $p_n: [0, 1/\ga_n] \to [1, +\infty)$ is continuous. 
Moreover, if $\ga\in \E{H}$ then $p_n=b_n$ on $[0, 1/\ga_n]$, for all $n\geq -1$. 
\end{propo}

\begin{proof}
We aim to show that $p_n^j$ uniformly converges to $p_n$ on $[0, 1/\ga_n]$. By the continuity of the maps $p_n^j$, 
this implies the first part of the proposition. At the same time, we show that $p_n^j$ uniformly converges to $b_n$, when 
$\ga \in \E{H}$, which implies the latter part of the proposition. We present the details in several steps. 

\smallskip

{\em Step 1.} For all $n\geq -1$, all $j\geq 0$, and all $x\in [0,1/\ga_n]$, 
\[p_n^j(x)\geq p_n^j(0)-5/\pi, \quad b_n^j(x) \geq b_n^j(0)-5/\pi.\] 

\smallskip 

We prove this by induction on $j$. 
When $j=0$, $p_n^0\equiv (\C{B}(\ga_{n+1})+5\pi)/(2\pi)$ and $b_n^0 \equiv -1$. 
Therefore, $p_n^0 (x) \geq p_n^0(0)-5/\pi$ and $b_n^0(x)\geq b_n^0(0)-5/\pi$. 
Now assume that both inequalities hold for some $j-1 \geq 0$ and all $n\geq -1$. 
Below we prove them for $j$ and all $n\geq -1$.

Fix an arbitrary $n\geq 0$, and let $x_{n-1} \in [0, 1/\ga_{n-1}]$ be arbitrary. Choose $x_{n}\in [0, 1/\ga_{n}]$ 
and $l_{n-1}\in \D{Z}$ with $-\gep_{n}\ga_{n} x_{n}= x_{n-1}+l_{n-1}$. 

If $p_n^{j-1}(x_n) \geq p_n^{j-1}(0)$, by the monotonicity of $y\mapsto \Im \mathbb{Y}_n(x_n+iy)$, we have 
\[\Im \mathbb{Y}_{n}(x_{n}+ip_{n}^{j-1}(x_{n})) - \Im \mathbb{Y}_{n}(x_{n}+ip_{n}^{j-1}(0)) \geq 0\geq -9/(2\pi).\]
If $p_n^{j-1}(x_n) \leq p_n^{j-1}(0)$, by the uniform contraction of $\mathbb{Y}_{n}$ in \refE{E:uniform-contraction-Y} 
and the induction hypothesis, we have
\begin{align*}
\Im \mathbb{Y}_{n}(x_{n}+ip_{n}^{j-1}(x_{n})) - \Im \mathbb{Y}_{n}(x_{n}+ip_{n}^{j-1}(0)) 
& \geq 0.9 (p_{n}^{j-1}(x_{n}) - p_{n}^{j-1}(0)) \\
& \geq 0.9(-5/\pi)= -9/(2\pi).
\end{align*}
On the other hand, by \refL{L:Y_r-distances}-(i), 
\[\Im \mathbb{Y}_{n}(x_{n}+ip_{n}^{j-1}(0)) \geq  \Im \mathbb{Y}_{n}(ip_{n}^{j-1}(0)) - 1/(2\pi) =p_{n-1}^{j}(0) -1/(2\pi).\]
Therefore, by the above inequalities and the definition of $p_{n-1}^{j}$, we have 
\begin{align*}
p_{n-1}^{j}(x_{n-1})& =\Im \mathbb{Y}_{n}(x_{n}+ ip_{n}^{j-1}(x_{n})) \\
&=\big(\Im \mathbb{Y}_{n}(x_{n}+ip_{n}^{j-1}(x_{n}))-\Im \mathbb{Y}_{n}(x_{n}+ip_{n}^{j-1}(0))\big)
+\Im \mathbb{Y}_{n}(x_{n}+ip_{n}^{j-1}(0))\\
& \geq  -9/(2\pi) + p_{n-1}^{j}(0) -1/(2\pi) = p_{n-1}^{j}(0) -5/\pi.  
\end{align*}
The same argument applies to the map $b_{n-1}^j$.  

As $j\to + \infty$, $p_n^j\to p_n$ and $b_n^j \to b_n$ point-wise on $[0, 1/\ga_n]$. 
These lead to 
\begin{equation}\label{E:P:p_n-cont-1}
p_n(x)\geq p_n(0)-5/\pi, \quad b_n(x) \geq b_n(0)-5/\pi. 
\end{equation} 

\smallskip

{\em Step 2.} For all $n\geq -1$ and $j\geq 0$ we have 
\[p_n^j(x) -p_n(x) \leq p_n^j(0)-p_n(0) +55/\pi, \q \tand \q p_n^j(x)-b_n^j(x) \leq p_n^j(0)-b_n^j(0) + 55/\pi.\]

\smallskip 

We shall prove these by induction on $j$. If $j=0$, by \refE{E:P:p_n-cont-1},
\[p_n^0(x)- p_n(x)= p_n^0(0)- p_n(x) \leq p_n^0(0) - p_n(0)+5/\pi, \]
and 
\[p_n^0(x)- b_n^0(x)= p_n^0(0) - b_n^0(0) \leq p_n^0(0) - b_n^0(0)+5/\pi.\]
Now assume that both inequalities in Step 2 hold for some $j\geq 0$ and all $n\geq -1$. 
We aim to prove it for $j+1$ and all $n\geq -1$. 
Fix an arbitrary $x_{n-1} \in [0, 1/\ga_{n-1}]$. Choose $x_{n}\in [0, 1/\ga_{n}]$ and $l_{n-1} \in \D{Z}$ such that 
$- \gep_{n} \ga_{n} x_{n}= x_{n-1} +l_{n-1}$. 
Then, using \refL{L:Y_r-distances}-(ii),  
\begin{align*}
p_{n-1}^{j+1}(x_{n-1}) - p_{n-1}(x_{n-1}) &= \Im \mathbb{Y}_{n}(x_{n} + i p_{n}^j(x_n)) - \Im \mathbb{Y}_{n}(x_{n} + i p_{n}(x_n)) \\
& \leq  \Im \mathbb{Y}_{n}(i p_{n}^j(x_n)) - \Im \mathbb{Y}_{n}(i p_{n}(x_n)) + 1/(2\pi).
\end{align*}
We consider two cases. 

(1) Assume that $p_n(x_n) \geq p_n(0)$. Using \refL{L:Y_r-distances}-(iii), the induction hypothesis, and 
\refE{E:uniform-contraction-Y}, respectively, we obtain  
\begin{align*}
\Im \mathbb{Y}_{n}(i p_{n}^j(x_n))&  - \Im \mathbb{Y}_{n}(i p_{n}(x_n)) \\
& \leq  \Im \mathbb{Y}_{n}(i (p_{n}^j(x_n) - p_n(x_n) +p_n(0)) - \Im \mathbb{Y}_{n}(i p_{n}(0)) + 1/(4\pi) \\
& \leq  \Im \mathbb{Y}_{n}(i(p_{n}^j(0) + 55/\pi) - p_{n-1}(0) + 1/(4\pi) \\
& \leq \Im \mathbb{Y}_{n}(ip_{n}^j(0)) + (9/10) \cdot (55/\pi)-  p_{n-1}(0) + 1/(4\pi) \\
& =p_{n-1}^{j+1}(0) - p_{n-1}(0) + 199/(4\pi).
\end{align*} 

(2) Assume that $p_n(x_n) < p_n(0)$. By \refE{E:P:p_n-cont-1}, we must have $p_n(0)-p_n(x_n)\in [0, 5/\pi]$. 
Then we may apply \refL{L:Y_r-distances}-(iv), and the induction hypothesis, to obtain 
\begin{align*}
\Im \mathbb{Y}_{n}(i p_{n}^j(x_n)) & - \Im \mathbb{Y}_{n}(i p_{n}(x_n)) \\
& \leq  \Im \mathbb{Y}_{n}(i(p_{n}^j(x_n) - p_n(x_n)+ p_n(0)) - \Im \mathbb{Y}_{n}(ip_{n}(0))+ 5/\pi \\
& \leq  \Im \mathbb{Y}_{n}(ip_{n}^j(0) + 55/\pi) - p_{n-1}(0) + 5/\pi \\
& \leq \Im \mathbb{Y}_{n}(i p_{n}^j(0)) + 0.9 \cdot 55/\pi - p_{n-1}(0) + 5/\pi \\
& =p_{n-1}^{j+1}(0) - p_{n-1}(0) + 109/(2\pi).
\end{align*}
Combining the above inequalities, we obtain the first inequality in Step 2 for $j+1$. 
The same argument applies to the difference $p_{n-1}^{j+1}(x_{n-1}) - b_{n-1}^{j+1}(x_{n-1})$. 

\smallskip

{\em Step 3.} For every $n\geq -1$, $p_n^j$ uniformly converges to $p_n$ on $[0, 1/\ga_n]$, as $j\to +\infty$. 

\smallskip

Fix $n\geq -1$, and let $\gep>0$ be arbitrary. Choose $m \geq n$ such that $0.9 ^{m-n} (1+55/\pi) < \gep$. 
Since $p_m^j(0) \to p_m(0)$, as $j\to + \infty$, there is $j_0 >0$ such that for all $j\geq j_0$ we have 
$|p_m^j(0)-p_m(0)| < 1$. By Step 2, this implies that for all $x' \in [0, 1/\ga_m]$ we have 
$|p_m^j(x')-p_m(x')| \leq 1+55/\pi$.
By the uniform contraction of $\mathbb{Y}_l$ in \refE{E:uniform-contraction-Y}, and since $\mathbb{Y}_l$ maps the graph of $p_l$ to 
$p_{l-1}$ and the graph of $p_l^k$ to the graph of $p_{l-1}^{k+1}$, we conclude that 
$|p_n^{j+m-n}(x)- p_n(x)| \leq 0.9^{m-n} (1+55/\pi) < \gep$.  

\smallskip

{\em Step 4.} If $\ga\in \E{H}$, for every $n\geq -1$, $p_n^j \to b_n$ on $[0, 1/\ga_n]$. 

\smallskip 

By \refP{P:optimal-herman-p_n-b_n}, if $\ga \in \E{H}$, we have $b_n(0)=p_n(0)$ for all $n\geq -1$. 
Taking limits as $j\to +\infty$ in the second inequality in Step 2, we conclude that for all $n\geq -1$ we have 
$|p_n(x)-b_n(x)| \leq 55/\pi$. 
By the uniform contraction of $\mathbb{Y}_l$, and equivariant property of the graphs of these functions, 
we conclude that $p_n \equiv b_n$, for all $n\geq -1$. 
\end{proof}

 
\subsection{Proof of Theorem \ref{T:model-trichotomy-thm}}\label{SS:proof-thm-thrichotomy}
In this section we combine the results from the previous sections to characterise the topology of the model 
$\mathbb{A}_\ga$. 

Note that $\mathbb{A}_\ga = \partial \mathbb{M}_\ga$ satisfies the relations 
\begin{equation}\label{E:A_ga}
\begin{gathered}
\mathbb{A}_\ga = \left\{s(e^{2\pi i w}) \mid w \in I_{-1}, \Im w \leq p_{-1}(\Re w) \right \},\quad  \tif \ga \in \E{B}, \\
\mathbb{A}_\ga= \left \{s(e^{2\pi i w}) \mid w \in I_{-1} \right\} \cup \left\{0 \right\}, \quad   \tif \ga \notin \E{B}.
\end{gathered}
\end{equation}

\begin{proof}[Proof of Theorem~\ref{T:model-trichotomy-thm}]
Recall that $I_{-1}$ is the set above the graph of $b_{-1}$, and $\mathbb{A}_\ga$ is obtained from $I_{-1}$ 
through a projection, see \eqref{E:I_n-b_n} and \eqref{E:A_ga}. 
By \refL{L:model-almost-periodic} and \refE{E:b_n^j-cont-periodic} we may extend $I_{-1}$ and $b_{-1}$ 
periodically with period $+1$. That is, $b_{-1}$ is defined on $\D{R}$ and is +1-periodic, and $I_{-1}$ becomes 
$I_{-1}+\D{Z}$.
When $\ga$ is a Brjuno number, we also have the function $p_{-1}$, which is $+1$-periodic by \refE{E:p_n-peiodic}. 
We may extend this function $+1$-periodically onto $\D{R}$ as well. 

{\em Part (i)}
By \refP{P:b_n-p_n-identical-or-not}, $b_{-1}\equiv p_{-1}$ on $[0,1]$, and hence $b_{-1}$ is a continuous function. 
It follows that $\partial I_{-1}$ is equal to the graph of $b_{-1}$. 
The map $e^{2\pi ix} \mapsto e^{-2\pi ix} e^{-2\pi b_{-1}(x)}$ from the unit circle to $\mathbb{A}_\ga$ is continuous 
and injective. 

{\em Part (ii)}
By \refP{P:b_n-sup-B}, $b_{-1}(x)$ is finite for every $x\in \D{R}$, and by \refE{E:p_n>=b_n}, $b_{-1} \leq p_{-1}$. 
Consider the function $R_\ga:\D{R}/\D{Z} \to [1, +\infty)$, $R_\ga(x)=e^{2\pi (p_{-1}(x)-b_{-1}(x))}$, and then the set 
\[A'_\ga=\{r e^{2\pi i x}  \mid 1\leq r \leq R_\ga(x)\}.\]
From \refP{P:b_n-p_n-dense-touches} we infer that each of $R_\ga(x)=1$ and $R_\ga(x) \neq 1$ hold on 
a dense subset of $\D{R}/\D{Z}$.  
Using the continuity of $p_{-1}$ in \refP{P:p_n-cont}, and \refP{P:b_n-liminfs}, we note that 
for every $x\in \D{R}/\D{Z}$, we have 
\[\limsup_{s\to x^+} R_\ga(s)= \exp \big(2\pi p_{-1} (x) - 2\pi \liminf_{s \to x^+} b_{-1}(s)\big) = R_\ga(x).\]
Similarly, $\limsup_{s\to x^-} R_\ga(s)= R_\ga(x)$.
Therefore, by the definition in the introduction, $A'_\ga$ is a one-sided hairy Jordan curve. 

Since $I_{-1}$ is a closed set, $\partial I_{-1} \subset I_{-1}$.  
On the other hand, $p_{-1}$ is continuous, and $p_{-1}=b_{-1}$ on a dense subset of $\D{R}$. 
It follows that 
\[\partial I_{-1}= \{x+iy \mid x\in \D{R}, b_{-1}(x) \leq y\leq p_{-1}(x)\}.\]
Hence, 
\[\mathbb{A}_\ga=\{ e^{2\pi ix} e^{-2\pi y} \mid x\in \D{R}/\D{Z}, b_{-1}(x) \leq y\leq p_{-1}(x) \}.\]
The map $r e^{2\pi ix}  \mapsto r e^{-2\pi i x} e^{-2\pi p_{-1}(x)}$ from $\mathbb{A}_\ga'$ to $\mathbb{A}_\ga$ is a homeomorphism.

{\em Part (iii)}
Consider the set 
\[A'_\ga = \{r e^{2\pi ix}  \mid 0 \leq r \leq e^{-2\pi b_{-1}(x)}\}.\]
By virtue of propositions \ref{P:b_n-liminfs} and \ref{P:b_n-p_n-dense-touches}, $\mathbb{A}_\ga'$ is a Cantor bouquet. 
Here, $I_{-1}$ is the set above the graph of $b_{-1}$, and $b_{-1}$ takes $+\infty$ on a dense subset of $\D{R}$, 
see \refP{P:b_n-p_n-dense-touches}. 
It follows that 
\[I_{-1}= \{x+iy \mid x \in \D{R}, b_{-1}(x) \leq y\},\] 
and hence, 
\[\mathbb{A}_\ga=\{r e^{2\pi ix} \mid 0 \leq r \leq e^{-2\pi b_{-1}(x)} \}.\]
Here, $r e^{2\pi ix}  \mapsto r e^{-2\pi i x}$ provides a homeomorphism from $\mathbb{A}_\ga'$ to $\mathbb{A}_\ga$. 
Thus, $\mathbb{A}_\ga$ is a Cantor bouquet. 
\end{proof}

\begin{cor}
For every $\ga \in \E{B}$, we have the following properties: 
\begin{itemize}
\item $\mathbb{M}_\ga$ contains the ball of radius $e^{-\mathcal{B}(\ga)-5\pi}$ about $0$, 
\item $\mathbb{M}_\ga$ does not contain any ball of radius more than $e^{-\mathcal{B}(\ga)+5\pi}$ about $0$. 
\end{itemize}
\end{cor}

\begin{proof}
Recall that any point $w$ with $\Re w \in [0,1]$ and $\Im w \geq p_{-1}(\Re w)$ belongs to 
$I_{-1}$. 
Recall that $P_{-1}^j$ is a decreasing sequence of functions converging to $p_{-1}$, 
and $p_{-1}^0 \equiv (\mathcal{B}(\ga)+5\pi)/(2\pi)$. 
Through projection $w \mapsto s(e^{2\pi i w})$, and adding $0$, we obtain the desired ball. 
On the other hand, by \refP{P:b_n-sup-B}, $ \sup_{x\in [0,1]} b_{-1}(x)$ is at least $(\mathcal{B}(\ga)-5\pi)/(2\pi)$. 
With the project, we obtain a point with modulus at most $e^{-\mathcal{B}(\ga)+ 5\pi}$ outside $\mathbb{M}_\ga$. 
\end{proof}  
\section{Dynamics of  \texorpdfstring{$\mathbb{T}_\ga$}{T-ga} on \texorpdfstring{$\mathbb{A}_\ga$}{A-ga}}\label{SS:dynamics on the model}
Here, we study the dynamics of $\mathbb{T}_\ga$ on $\mathbb{A}_\ga$ and 
classify the closed invariant subset of this model. 
In particular, we prove all parts of \refT{T:dynamics-on-model} in this section. 

\subsection{Topological recurrence}\label{SS:recurrence}
Recall that a map $f: X \to X$, of a topological space $X$, is called \textbf{topologically recurrent}, 
if for every $x\in X$ there is a strictly increasing sequence of positive integers $(m_i)_{i\geq 0}$ such 
that $f\co{m_i}(x) \to x$ as $i \to +\infty$. 

\begin{propo}\label{P:model-recurrent}
For every $\ga \in \D{R} \setminus \D{Q}$, $\mathbb{T}_\ga : \mathbb{A}_\ga \to \mathbb{A}_\ga$ is 
topologically recurrent. 
\end{propo}

\begin{proof}
Evidently, it is enough to show that $\tilde{T}_\ga: I_{-1}/\D{Z} \to I_{-1}/\D{Z}$ is topologically recurrent. 
To that end, fix an arbitrary $w_{-1} \in I_{-1}$, and assume that $(w_i ; l_i)_{i\geq -1}$ denotes the trajectory 
of $w_{-1}$. 
We consider two cases: 
\begin{itemize}
\item[(i)] there are arbitrarily large integers $m$ with $w_m \in K_m$, 
\item[(ii)] there is an integer $m \geq -1$ such that for all $i \geq m$ we have $w_i \in I_i \setminus K_i$. 
\end{itemize}
 
Let us first assume that (i) holds. Fix an arbitrary $\gep>0$. There is $m \geq 4$ such that $w_m \in K_m$ and 
$(0.9)^m \sqrt{2} \leq \gep$. 
If $w_m+1 \in K_m$, then we note that 
\[(\mathbb{Y}_0+(\gep_0+1)/2) \circ (\mathbb{Y}_1 + l_0)  \circ \dots \circ (\mathbb{Y}_m + l_{m-1})(w_m+1)\]
is defined and belongs to $I_{-1}$. 
It follows from an induction argument that there is an integer $n\geq 0$ such that the above point is equal to 
$\tilde{T}_\ga\co{n}(w_{-1})$. 
Then, by \refE{E:trajectory-condition-1} and the uniform contraction of $\mathbb{Y}_j$ in 
\refL{L:uniform-contraction-Y_r}, 
we have $|\tilde{T}_\ga\co{n}(w_{-1}) -w_{-1}| \leq (0.9)^{m+1} \cdot 1 \leq \gep$. 

If $w_m+1 \notin K_m$, we may not directly apply the above argument, since $\mathbb{Y}_m+l_{m-1}(w_m+1)$ 
may not belong to $I_{m-1}$. 
However, there is an integer $l' \geq (\gep_m+1)/2$ such that 
\[|\Re (\mathbb{Y}_m+l')(w_m+1) - \Re w_{m-1}| \leq 1,\] 
and either both $(\mathbb{Y}_m+l')(w_m+1)$ and $w_{m-1}$ belong to $K_{m-1}$, or both 
$(\mathbb{Y}_m+l')(w_m+1)$ and $w_{m-1}$ belong to $I_{m-1} \setminus K_{m-1}$. 
Note that, by \refL{L:uniform-contraction-Y_r}, we have 
\begin{align*}
| \Im (\mathbb{Y}_m + l')(w_m+1) - \Im w_{m-1}| 
& = |\Im (\mathbb{Y}_m + l')(w_m+1) - \Im (\mathbb{Y}_m + l_{m-1})(w_m) | \\ 
& \leq | \mathbb{Y}_m (w_m+1) - \mathbb{Y}_m(w_m) | \\
& \leq 0.9 \cdot 1 \leq 1. 
\end{align*}
Combining with the upper bound on the difference of the real parts, we obtain 
\[|(\mathbb{Y}_m + l')(w_m+1) - w_{m-1}| \leq \sqrt{2}.\] 
Now we consider the point 
\[(\mathbb{Y}_0+(\gep_0+1)/2) \circ (\mathbb{Y}_1 + l_0)  \circ \dots 
\circ (\mathbb{Y}_{m-1}+ l_{m-2}) \circ (\mathbb{Y}_m + l')(w_m+1),\]
which is defined and belongs to $I_{-1}$. 
In the same fashion, there is an integer $n\geq 0$ such that the above point is equal to $\tilde{T}_\ga \co{n}(w_{-1})$. 
Using the uniform contraction of $\mathbb{Y}_j$ again, and the above bound on 
$|(\mathbb{Y}_m + l')(w_m+1) - w_{m-1}|$, we conclude that 
$|\tilde{T}_\ga\co{n}(w_{-1}) - w_{-1}| \leq (0.9)^{m} \cdot \sqrt{2} \leq \gep$. 

Now assume that case (ii) holds. 
Given $\gep>0$, we may choose $m \geq 4$ such that for all $i \geq m$ we have $w_i \in I_i \setminus K_i$, and 
$0.9 ^{m+1} \cdot 31 \leq \gep$. 

There is $w'_{-1}$ in $I_{-1}$ whose trajectory $(w'_i; l_i')_{i\geq -1}$ satisfies $w_i \in I_i \setminus K_i$, 
for $0 \leq i \leq m$, and $w'_m = w_m$. It follows that there is an integer $n_1 \geq 0$ such that 
$\tilde{T}_\ga \co{n_1}(w_{-1})= w'_{-1}$. 

The point $w'_{-1}$ belongs to $V^\infty$, and $\tilde{T}_\ga$ at $w'_{-1}$ is defined using the second case 
in the definition of $\tilde{T}_\ga$. Recall the map $E_n$ defined in \refE{E:E_n-defn}. 
Let us choose the integer $l \geq (\gep_{m+1}+1)/2$ such that $E_m(w_m)+l \in I_m\setminus K_m$. 
It follows that $|\Re E_m(w_m) + l - \Re w_m | \leq 1$. 
By \refE{E:L:model-map-lift-2-3}, we have $|\Im E_m(w_m) + l - \Im w_m | \leq 30$.
Thus, $|E_m(w_m) +l - w_m | \leq 31$. 

Consider the point 
\[(\mathbb{Y}_0+(\gep_0+1)/2) \circ (\mathbb{Y}_1 + l_0)  \circ \dots 
\circ (\mathbb{Y}_{m-1}+ l_{m-2}) \circ (\mathbb{Y}_m + l)(E_m(w_m)),\]
which is defined and belongs to $I_{-1}$. 
The above point is equal to $\tilde{T}_\ga\co {n_2} (\tilde{T}_\ga(w'_{-1}))$. 
By the uniform contraction of $\mathbb{Y}_j$, we conclude that 
$|\tilde{T}_\ga\co {n_2} (\tilde{T}_\ga(w'_{-1})) - \tilde{T}_\ga (w_{-1}')| \leq 0.9^{m+1} \cdot 31$. 
By the above paragraph, this implies that 
$|\tilde{T}_\ga\co {n_2+1+n_1}(w_{-1}) - \tilde{T}_\ga^{n_1+1} (w_{-1}) | \leq 0.9^{m+1} \cdot 31 \leq \gep$. 
As $n_1$ is independent of $\gep$, this completes the proof in case (ii).  
\end{proof}

Recall that a set $K \subset \mathbb{A}_\ga$ is called \textbf{forward invariant} under $\mathbb{T}_\ga$, 
if $\mathbb{T}_\ga (K) \subseteq K$.
The set $K$ is called \textbf{invariant}, or \textbf{fully invariant}, under $\mathbb{T}_\ga$, 
if $\mathbb{T}_\ga ^{-1}(K)=K$.

\begin{propo}\label{P:invariant-fully-invariant}
Let $\ga \in \D{R} \setminus \D{Q}$. If $K \subset \mathbb{A}_\ga$ is closed and forward invariant under 
$\mathbb{T}_\ga$, then $K$ is fully invariant under $\mathbb{T}_\ga$. 
\end{propo}

\begin{proof}
Fix an arbitrary $z\in K$. 
By \refP{P:model-recurrent}, there is an increasing sequence of positive integers $m_i$ 
such that $\mathbb{T}_\ga\co {m_i}(z) \to z$ as $i \to \infty$. 
Since, $K$ is closed, the sequence $\mathbb{T}_\ga \co {m_i-1}(z)$ has a convergence 
subsequence, which converges to some $z'$ in $K$. Evidently, $\mathbb{T}_\ga(z')=z$, and hence 
$z'= \mathbb{T}_\ga ^{-1}(z) \in K$. 
\end{proof}

\subsection{Closed invariant subsets}\label{SS:invariants-model}
In this section we build a family of closed invariant sets for $\mathbb{T}_\ga: \mathbb{A}_\ga \to \mathbb{A}_\ga$, 
parametrised on a closed interval in $\D{R}$. 
The process is in analogy with how the set $\mathbb{A}_\ga$ is built in \refS{SS:tilings-nest}. 

Fix an arbitrary $y \geq 0$, and inductively define $y_n \geq 0$, for $n\geq -1$, according to 
\begin{equation}\label{E:invariants-imaginary-traces}
y_{-1}= y, \qquad y_{n+1}= \Im \mathbb{Y}_{n+1}^{-1}(i y_n).
\end{equation}
For $n\geq 0$, let 
\begin{equation}\label{E:I_n-J_n-K_n-general}
\begin{gathered}
\prescript{y}{}{I}_n^0 = \{w \in \D{C}  \mid \Re w\in [0, 1/\ga_n], \Im w \geq y_n -1\}, \\
\prescript{y}{}{J}_n^0 =  \{w \in \prescript{y}{}{I}_n^0 \mid \Re w \in [1/\ga_n-1, 1/\ga_n]\}, \q
\prescript{y}{}{K}_n^0 =  \{w \in \prescript{y}{}{I}_n^0 \mid \Re w \in [0, 1/\ga_n-1] \}.
\end{gathered}
\end{equation}
As in \refS{SS:tilings-nest}, we inductively defined the sets $\prescript{y}{}{I}_n^j$, $\prescript{y}{}{J}_n^j$, and 
$\prescript{y}{}{K}_n^j$, for $j \geq 1$ and $n\geq 0$. 
Assume that $\prescript{y}{}{I}_n^j$, $\prescript{y}{}{J}_n^j$, and $\prescript{y}{}{K}_n^j$ are defined for some 
$j$ and all $n \geq 0$. 
When $\gep_{n+1}=-1$, we let 
\begin{equation}\label{E:I_n^j--1-general}
\prescript{y}{}{I}_n^{j+1} 
= \bigcup_{l=0}^{a_n-2} \big( \mathbb{Y}_{n+1} (\prescript{y}{}{I}_{n+1}^j)+ l \big)  
\bigcup \big( \mathbb{Y}_{n+1}(\prescript{y}{}{K}_{n+1}^j)+ a_n-1\big).
\end{equation}
When $\gep_{n+1}=+1$, we let 
\begin{equation}\label{E:I_n^j-+1-general}
\prescript{y}{}{I}_n^{j+1} 
= \bigcup_{l=1}^{a_n} \big( \mathbb{Y}_{n+1} (\prescript{y}{}{I}_{n+1}^j)+ l \big) 
 \bigcup \big(\mathbb{Y}_{n+1}(\prescript{y}{}{J}_{n+1}^j)+ a_n+1\big ).
\end{equation}
Then, define 
\[\prescript{y}{}{J}_n^{j+1} =  \{w \in \prescript{y}{}{I}_n^{j+1} \mid \Re w \in [1/\ga_n-1, 1/\ga_n]\}, \; 
\prescript{y}{}{K}_n^{j+1} =  \{w\in \prescript{y}{}{I}_n^{j+1} \mid \Re w \in [0, 1/\ga_n-1] \}.\]
Let 
\[\prescript{y}{}{I}_{-1}^0=\{w \in \D{C}  \mid \Re w \in [0, 1/\ga_{-1}], \Im w \geq y_{-1}-1\},\] 
and for $j\geq 1$, consider the sets 
\[\prescript{y}{}{I}_{-1}^j= \mathbb{Y}_0(\prescript{y}{}{I}_0^{j-1}) + (\gep_0+1)/2.\]
By the latter part of \refL{L:Y-domain}, $\prescript{y}{}{I}_n^1 \subset \prescript{y}{}{I}_n^0$, for $n\geq -1$. 
This implies that for all $n\geq -1$ and all $j\geq 0$, 
\begin{equation}\label{E:I_n^j-forms-nest-general}
\prescript{y}{}{I}_n^{j+1} \subset \prescript{y}{}{I}_n^j.
\end{equation}
For $n\geq -1$, we define the closed sets 
\[\prescript{y}{}{I}_{n}= \bigcap_{j\geq 0} \prescript{y}{}{I}_{n}^j.\] 
Evidently, when $y=0$, we have $y_n=0$, for all $n\geq 0$, and hence, $\prescript{0}{}I_n=I_n$, for all $n\geq -1$. 

Note that, $iy \in \prescript{y}{}I_{-1}$, for any $y\geq 0$, and  
\begin{equation}\label{E:invariant-sets-lift-nest}
\prescript{x}{}I_{-1} \subsetneq \prescript{y}{}I_{-1}, \qquad \tif x> y \geq 0. 
\end{equation}
Moreover, by the uniform contraction of $\mathbb{Y}_n$ in \refL{L:uniform-contraction-Y_r}, 
\begin{equation}\label{E:invariant-set-end}
iy' \notin \prescript{y}{}{I}_{-1}, \quad \tif  y' < y.  
\end{equation}

Recall that for all $\ga \in \D{R} \setminus \D{Q}$, $\max (\mathbb{A}_\ga \cap \D{R})=+1$.  
We define $r_\ga \geq 0$ according to 
\[[r_\ga, 1]= \mathbb{A}_\ga \cap [0, +\infty).\]
If $\ga \notin \E{B}$, $r_\ga=0$, and if $\ga \in \E{B}$, $r_\ga=e^{-2\pi p_{-1}(0)}$. 
When $\ga \notin \E{B}$ and $t \in (0,1]$, choose $y \geq 0$ so that $t=e^{-2\pi y}$ and define 
\[\prescript{t}{}{\mathbb{A}}_\ga=\left \{s(e^{2\pi i w}) \mid w\in\prescript{y}{}{I}_{-1} \right\} \cup \left\{0 \right\}.\]
We extend this notation by setting $\prescript{0}{}{\mathbb{A}_\ga}= \{0\}$. 
When $\ga \in \E{B}$ and $t \in [r_\ga,1]$, choose $y \geq 0$ so that $t=e^{-2\pi y}$, and define 
\[\prescript{t}{}{\mathbb{A}}_\ga= \left\{s(e^{2\pi i w}) \mid w\in\prescript{y}{}{I}_{-1}, 
\Im w\leq p_{-1}(\Re w) \right \}.\]
For all $\ga \in \D{R} \setminus \D{Q}$, $\prescript{1}{}{\mathbb{A}}_\ga= \mathbb{A}_\ga$. 
By \eqref{E:invariant-sets-lift-nest} and \eqref{E:invariant-set-end}, for every $r_\ga \leq s < t \leq 1$, 
we have 
\begin{equation}\label{E:invariants-model-nested}
\prescript{s}{}{\mathbb{A}_\ga} \subsetneq \prescript{t}{}{\mathbb{A}_\ga}, \quad \tand \quad t 
\notin \prescript{s}{}{\mathbb{A}_\ga}. 
\end{equation}

\begin{propo}\label{P:^tA_ga-invariant}
For any $\ga \in \D{R} \setminus \D{Q}$ and any $t \in [r_\ga, 1]$, $\prescript{t}{}{\mathbb{A}}_\ga$ is fully 
invariant under $\mathbb{T}_\ga: \mathbb{A}_\ga \to \mathbb{A}_\ga$. 
\end{propo}

\begin{proof}
By \refP{P:invariant-fully-invariant}, it is enough to show that $\prescript{t}{}{\mathbb{A}}_\ga$ is forward invariant. 
By the definition of $\prescript{t}{}{\mathbb{A}}_\ga$, it is enough to show that for all $y\geq 0$, 
$\prescript{y}{}{I}_{-1}$ is forward invariant under $\tilde{T}_\ga$. Let us fix $y \geq 0$. 

Recall the decomposition $I_{-1}=\cup_{n\geq 0} W^n \cup V^\infty$ in \refE{E:I_-1-decomposed}. 
Let $w \in \prescript{y}{}I_{-1} \cap W^n$, for some $n\geq 0$.  
Let $(w_j, l_j)_{j\geq -1}$ denote the trajectory of $w$. 
By the definition of $W^n$, $w_{n} \in K_n$, and by the definition of 
$\prescript{y}{}I_{-1}$, $w_n \in \prescript{y}{}I_{n}$. 
Since, $\prescript{y}{}I_{n}$ is translation invariant, $w_n +1 \in \prescript{y}{}I_{n}$. 
Then, it follows from the definitions of $\tilde{T}_\ga$ and $\prescript{y}{}I_{-1}$ that 
$\tilde{T}_{\ga}(w) \in \prescript{y}{}I_{-1}$. 

By the above paragraph, $\tilde{T}_\ga$ maps $\cup_{n\geq 0} W^n \cap \prescript{y}{}I_{-1}$ into 
$\prescript{y}{}I_{-1}$. 
Since $\tilde{T}_\ga: I_{-1}/\D{Z} \to I_{-1}/\D{Z}$ is continuous and $\prescript{y}{}I_{-1}/\D{Z}$ is closed, 
$\tilde{T}_{\ga}$ maps the closure of $(\cup_{n\geq 0} W^n \cap \prescript{y}{}I_{-1})/\D{Z}$ into 
$\prescript{y}{}I_{-1}/\D{Z}$.
The closure of $(\cup_{n\geq 0} W^n \cap \prescript{y}{}I_{-1})/\D{Z}$ is equal to $\prescript{y}{}I_{-1}/\D{Z}$. 
\end{proof}

\subsection{Closures of orbits}

\begin{lem}\label{L:dense-orbits-model}
For every $\ga \in \D{R} \setminus \D{Q}$, the following hold: 
\begin{itemize}
\item[(i)] if $\ga \notin \E{B}$, then for all $y \geq 0$, the orbit of $iy$ under $\tilde{T}_\ga$ 
is dense in $\prescript{y}{}{I}_{-1}$; 
\item[(ii)] if $\ga \in \E{B}$, then for all $y$ with $0 \leq y \leq p_{-1}(0)$, the orbit of $iy$ 
under $\tilde{T}_\ga$ is dense in 
\[\{ w\in \prescript{y}{}{I}_{-1}\mid \Im w\leq p_{-1}(\Re w)\}.\] 
\end{itemize}
\end{lem}

\begin{proof}
To simplify the notations, let us first consider the orbit of $y=0$. 

When $\ga \in \E{H}$, the set in item (ii) becomes the graph of the function $b_{-1}= p_{-1}$. 
Then, the statement follows from the continuity of $b_{-1}$ in \refP{P:p_n-cont}. 
The non-trivial case is to prove the statement when $\ga \notin \E{H}$.  

Let $\langle x \rangle$ denote the fractional part of $x \in \D{R}$, that is, $\langle x \rangle \in [0,1)$ and 
$ x \in \langle x \rangle + \D{Z}$.  
By \refE{E:A_ga}, in order to prove the proposition, it is enough to show that the set
\[\Big \{\langle - m \gep_0 \ga_0  + (1+\gep_0)/2 \rangle +  i  b_{-1}(\langle -m \gep_0 \ga_0 + (1+\gep_0)/2 \rangle) 
\Bigm \vert m \in \D{N}\Big\}\] 
is dense in $\partial I_{-1}$. This is because $s(e^{2 \pi i (-m \gep_0 \ga_0)})= e^{2\pi i m\ga}$. 

It is possible to prove both items in the proposition at once. 
Assume that $z$ is an arbitrary point, such that either $\ga \notin \E{B}$ and $z \in I_{-1}$, 
or $\ga \in \E{B}$ and $z \in I_{-1}$ with $\Im z \leq p_{-1}(\Re z)$. 
Also, fix an arbitrary $\gd>0$. 
We aim to identify an element of the orbit of $y$ in the $\delta$ neighbourhood of $z$. 

Recall that $I_{-1}= \cap_{j\geq 1} I_{-1}^j$. 
Choose $j_0 \geq 1$ such that $(9/10)^{j_0} \sqrt{2} \leq \gd/2$. 
There are $j \geq j_0$ and $z' \in \partial I_{-1}^{j+1}$ such that $|z'-z|\leq \gd/2$ and $\Re z'\in (0,1)$.  
By the definition of $I_{-1}^{j+1}$, there must be $w_{j} \in \partial I_{j}^0 $ and integers $l_k \in \D{Z}$, 
for $-1\leq k\leq j$, such that 
\[z'=(\mathbb{Y}_0+l_{-1})\circ (\mathbb{Y}_1+l_0) \circ \dots \circ (\mathbb{Y}_{j}+l_{j-1})(w_j).\] 
There is an integer $l_j$ with $0 \leq l_j \leq 1/\ga_j$ such that $|l_j - w_j| \leq \sqrt{2}$. 
Let 
\[z''=(\mathbb{Y}_0+l_{-1}) \circ (\mathbb{Y}_1+l_1) \circ \dots \circ (\mathbb{Y}_j+l_{j-1})(l_j).\] 
By the uniform contraction of the maps $\mathbb{Y}_k$ in \refE{E:uniform-contraction-Y}, 
$|z'-z''| \leq (9/10)^{j+1}\sqrt{2}\leq \gd/2$. 
In particular, $|z''-z|\leq \gd$. 

On the other hand, since  $\Re \mathbb{Y}_k(x) = - \gep_k \ga_k \Re x$, one may verify that 
\[\Re z''= \sum_{k=0}^{j} \Big(l_k \prod_{n=0}^{k} (-\gep_n \ga_n)\Big)+ l_{-1}.\]
Any value of the above form is equal to $- m\gep_0\ga_0 + (1+\gep_0)/2$, for some $m\in \D{N}$, modulo $\D{Z}$. 
This may be proved by induction on $j$. 
From the definition of $I_{-1}$ we note that $l_{-1}= (1+\gep_0)/2$. This implies the statement for $j=0$. 
Now assume that the statement is true for all integers less than $j$. To prove it for $j$, one uses the relation 
$\gep_j \ga_j= 1/\ga_{j-1}-a_{j-1}$ to reduce the statement to $j-1$. 

The proof for non-zero values of $y$ is identical to the above one; one only needs to employ the translation 
invariance of the sets $\prescript{y}{}{I}_n^j$.
\end{proof}

\begin{propo}\label{P:invariants-ordered-model}
For any $\ga \in \D{R} \setminus \D{Q}$, and any $t \in [r_\ga,1]$, $\gw(t)= \prescript{t}{}{\mathbb{A}}_\ga$. 
In particular, 
\begin{itemize}
\item if $s > t$, $s \notin \gw(t)$, 
\item the obit of $+1$ is dense in $\mathbb{A}_\ga$. 
\end{itemize}
\end{propo}

\begin{proof}
This is an immediate corollary of \refL{L:dense-orbits-model} and the property in \refE{E:invariant-set-end}.
\end{proof}

\begin{propo}\label{P:invariants-classified-model}
Let $\ga \in \D{R} \setminus \D{Q}$. For every non-empty closed invariant set $X$ of 
$\mathbb{T}_\ga: \mathbb{A}_\ga \to \mathbb{A}_\ga$, there is $t \in [r_\ga,1]$ such that 
$X=\prescript{t}{}{\mathbb{A}}_\ga$.
\end{propo}

\begin{proof}
If $\ga \notin \E{B}$ and $X=\{0\}$, we let $t=0$. 
Otherwise, let $X_{-1} \subset I_{-1}$ be the lift of $X \setminus \{0\}$, which is a closed set in $\mathbb{C}$. 
We inductively define the closed sets $X_n$, for $n\geq 0$, according to 
\[X_{n+1}=\{w \in I_{n+1} \mid \mathbb{Y}_{n+1}(w)+(\gep_{n+1}+1)/2 \in X_n\}.\]
We consider two cases below. 

{\em Case (I)} There is $N\geq -1$ such that for all $n\geq N$, $\min \Im X_n \geq 4$.

We define the functions $h_n^j:[0, 1/\ga_n] \to [-1, +\infty)$, for $n\geq -1$ and $j\geq N+1$, as follows. 
For $n\geq N$, let 
\[h_n^0(x) = \min \Im X_n -5.\]
Then, define $h_n^{j+1}$ as the lift of $h_{n+1}^j$ by $\mathbb{Y}_{n+1}$.
That is, for $x\in [0,1]$, let 
\[h_n^{j+1}(x)= \Im \mathbb{Y}_{n+1} (- \gep_{n+1}x / \ga_{n+1}+ i h_{n+1}^j(- \gep_{n+1}x / \ga_{n+1})),\] 
and then extend $h_n^{j+1}$ over $[0, 1/\ga_n]$ using $h_n^{j+1}(x)= h_n^{j+1}(x+1)$.  

We claim that for all $n \geq -1$ and all $j\geq N+1$, $h_n^{j+1} \geq h_n^j$. 
We show this by induction on $j$. 
Let $n \geq N$. 
Since $X_{-1}$ is invariant under $\tilde{T}_\ga$, $X_{n+1}$ is $+1$-periodic. 
This implies that every point on the graph of $h_{n+1}^0$ is at distance at most $(5+ 1/2)$ from some element of 
$X_{n+1}$. 
Then, by the uniform contraction of $\mathbb{Y}_{n+1}$ in \refL{L:uniform-contraction-Y_r}, any point on the 
graph of $h_n^1$ is at distance at most $(1/2+5) \cdot 0.9 \leq 5$ from some element of $X_n$. 
This implies that $h_n^1 \geq \min \Im X_n -5 = h_n^0$, for all $n \geq N$. 
Now repeatedly applying the maps $\mathbb{Y}_l$, we obtain the desired property. 

By the above paragraph, $h_n^j$ converges to some function $h_n: [0, 1/\ga_n] \to [-1, +\infty]$, as $j \to +\infty$. 
Moreover, by the relations in \eqref{E:Y_n-comm-1} and \eqref{E:Y_n-comm-2}, for any $n\geq -1$ we have 
\[h_n(0)=h_n(1/\ga_n), \qquad h_n(x)=h_n(x+1), \quad \tfor n\geq -1, x\in [0, 1/\ga_n-1].\]
One may repeat the argument in Step 1 of the proof of \refP{P:p_n-cont} (replace $p_n^j$ with $h_n^j$) to 
conclude that for all $n \geq N$ and all $x\in [0, 1/\ga_n]$, $h_n(x) \geq h_n(0) -5/\pi$. 
Since $X_N$ is not empty, and lies above the graph of $h_N$, there must be $x \in [0,1/\ga_N]$ such that 
$h_N(x) < +\infty$. 
In particular, $h_N(0)\neq +\infty$, and hence $h_{-1}(0) \neq +\infty$. 
Let us introduce $y=h_{-1}(0)$. 

We claim that $iy \in X_{-1}$. 
To see this, fix an arbitrary $n\geq N$. 
There is a net of points on $X_{n}-5 i$, at distance one from one another, which lie on the graph of $h_n^0$. 
Lifting these points using $\mathbb{Y}_j+l_j$, we obtain a net of points on the graph of $h_{-1}^n$, 
which by the uniform contraction of $\mathbb{Y}_j$ in \refL{L:uniform-contraction-Y_r} must be at distances 
at most $1 \cdot 0.9^{n+1}$ from one another. 
Moreover, those points on the graph of $h_{-1}^n$ lie at distances at most $5 \cdot 0.9^{n+1}$ from some elements 
of $X_{-1}$. 
As $h_{-1}^n \to h_{-1}(0)$, one may conclude that there is a sequence of points in $X_{-1}$ which converges 
to $i h_{-1}(0)= iy$. 
Since $X_{-1}$ is closed, $iy \in X_{-1}$. 

Let us consider the set $\prescript{y}{}I_{-1}$ and the associated function $\prescript{y}{}b_n$. 
We claim that $h_n \equiv \prescript{y}{}{b}_n$, on $[0, 1/\ga_n]$, for all $n\geq -1$. 
Recall the values $y_n \geq 0$ from \refE{E:invariants-imaginary-traces}, and note that $h_n(0)=y_n$, 
for all $n\geq -1$. 
Also, recall the functions $\prescript{y}{}b_n^j: [0, 1/\ga_n] \to [-1, +\infty]$. 
Since $X_n$ lies above the graph of $h_n$, for $n\geq N$ we have 
\[h_n^0(x)=\min \Im X_n -5 \geq \min_{x\in [0, 1/\ga_n]} h_n(x) -5 \geq h_n(0)-5/\pi-5.\] 
On the other hand, $h_n(0) \geq h_n^0(0)= h_n^0(x)$, for all $x \in [0, 1/\ga_n]$. 
Recall that $\prescript{y}{}b^0_n(x) \equiv y_n-1=h_n(0)-1$, for $x\in [0, 1/\ga_n]$. 
Therefore, for all $n \geq N$ and all $x\in [0, 1/\ga_n]$, $|h_n^0(x)-\prescript{y}{}{b}_n^0(x)| \leq 6$. 
By the uniform contraction of $\mathbb{Y}_l$ in \refL{L:uniform-contraction-Y_r}, for all $n \geq -1$, 
$x \in [0, 1/\ga_n]$ and $j \geq N+1$, we must have $|h_n^j(x)- \prescript{y}{}b_n^j(x)| \leq 6 \cdot 0.9^j$. 
Taking limits as $j \to +\infty$, we obtain $h_n(x) \equiv \prescript{y}{}{b}_n(x)$. 

Since $X_{-1}$ lies above the graph of $h_{-1}$, and $h_{-1}= \prescript{y}{}b_{-1}$, we must have 
\[X_{-1} \subseteq \prescript{y}{}{I}_{-1}, \; \tif \ga \notin \E{B}, \q \tand \q 
X_{-1} \subseteq \{ w\in \prescript{y}{}{I}_{-1} \mid \Im w \leq p_{-1}(\Re w) \}, \; \tif \ga \in \E{B}.\]

On the other hand, as $iy\in X_{-1}$, $X_{-1}$ is closed, and $X_{-1}$ is invariant under $\tilde{T}_\ga$, 
the closure of the orbit of $iy$ under $\tilde{T}_\ga$ must be contained in $X_{-1}$. 
By virtue of \refL{L:dense-orbits-model}, we conclude that the inclusions in the above equation are equalities. 

{\em Case (II)} There are arbitrarily large $n$ with $\min \Im X_n \leq 4$.  

Here, we may introduce $h_n^0(x)\equiv -1$, for all $n \geq -1$, and define the functions $h_n^j$ as in case (I). 
We note that $h_n^j$ is an increasing sequence of functions, which converges to some 
$h_n: [0, 1/\ga_n] \to [-1, +\infty]$. 
Moreover, these functions also enjoy the functional relations $h_n(x+1)=h_n(x)$, for $x\in [0, 1/\ga_n]$, and 
$h_n(0)=h_n(1/\ga_n)$. Then, as in the above paragraphs, we must have $h_n(0) \leq h_n(x) +5/\pi$, for all $n\geq 0$. 
Since $X_n$ lies above the graph of $h_n$, combining with the hypothesis in this case, there must be arbitrarily 
large $n$ with $h_n(0) \leq 4+ 5/\pi$. 
By the uniform contraction of $\mathbb{Y}_l$, we conclude that $h_{-1}(0)=0$. 
Thus, $0 \in X_{-1}$. 
Since $X_{-1}$ is closed, and invariant under $\tilde{T}_\ga$, the closure of the orbit of $0$ must be in $X_{-1}$. 
Using \refL{L:dense-orbits-model}, we complete the proof in this case. 
\end{proof}

\subsection{Topology of the closed invariant subsets}\label{SS:top-invariant-sets}
In order to explain the topological properties of the sets $\prescript{t}{}{\mathbb{A}}_\ga$, we use height functions 
as in \refS{SS:height-functions} to study the structure of the sets $\prescript{y}{}I_{-1}$. 
Since each $\mathbb{Y}_n$ preserves vertical lines, each of $I_n^j$ and $I_n$, for $n\geq -1$ and $j\geq 0$, 
consists of closed half-infinite vertical lines. 
For $n\geq -1$ and $j\geq 0$, define $\prescript{y}{}b_n^j:[0, 1/\ga_n] \to [-1, +\infty)$ according to 
\begin{equation}\label{E:I_n^j-b_n^j-general}
\prescript{y}{}I_n^j= \{w\in \D{C} \mid 0 \leq \Re w \leq 1/\ga_n, \Im w \geq \prescript{y}{}b_n^j(\Re w)\}.
\end{equation}
By the relations in \eqref{E:Y_n-comm-1}--\eqref{E:Y_n-comm-2}, for all $n\geq -1$ and $j\geq 0$, 
$\prescript{y}{}b_n^j:[0, 1/\ga_n] \to [-1, +\infty)$ is continuous. 
Moreover, by \eqref{E:I_n^j-forms-nest-general} and \eqref{E:I_n^j-b_n^j-general}, 
$\prescript{y}{}b_n^{j+1} \geq \prescript{y}{}b_n^j$ on $[0, 1/\ga_n]$. 
Thus, for $n\geq -1$, we may define $\prescript{y}{}b_n:[0, 1/\ga_n] \to [-1,+\infty]$ as
\[\prescript{y}{}b_n(x)= \lim_{j\to + \infty} \prescript{y}{}b_n^j(x)= \sup_{j\geq 0} \prescript{y}{}b_n^j(x).\]
The function $\prescript{y}{}b_n$ describes the set $I_n$, that is, 
\begin{equation}\label{E:I_n-b_n-general}
\prescript{y}{}I_n= \{w\in \D{C} \mid 0 \leq \Re w \leq 1/\ga_n, \Im w \geq \prescript{y}{}b_n(\Re w)\}.
\end{equation}
By the definition of the sets $\prescript{y}{}I_n^j$, $\prescript{y}{}b_n^j(0)=\prescript{y}{}b_n^j(1/\ga_n)$, and 
$\prescript{y}{}b_n^j(x+1)= \prescript{y}{}b_n^j(x)$ for all $x\in [0, 1/\ga_n-1]$. 
Taking limits as $j\to +\infty$, we note that for all $n\geq -1$,
\begin{equation}\label{E:b_n^j-cont-periodic-general}
\prescript{y}{}b_n(0)=\prescript{y}{}b_n(1/\ga_n), \quad 
\prescript{y}{}b_n(x+1)=\prescript{y}{}b_n(x), \quad \forall x \in [0, 1/\ga_n-1].
\end{equation}
These are the key functional relations required to explain the topology of $\prescript{y}{}I_{-1}$. 

From \refS{SS:height-functions}, recall the height functions $p_n : [0, 1/\ga_n] \to [-1, +\infty)$, for $n\geq -1$. 

\begin{propo}\label{P:invariants-topology-model}
For every $\ga \in \D{R} \setminus \D{Q}$ the following hold: 
\begin{itemize}
\item[(i)] if $\ga \notin \E{B}$, for every $t \in (r_\ga, 1]$, $\prescript{t}{}{\mathbb{A}}_\ga$ is a Cantor bouquet;
\item[(ii)] if $\ga \in \E{B} \setminus \E{H}$, for every $t \in (r_\ga, 1]$, $\prescript{t}{}{\mathbb{A}}_\ga$ 
is a one-sided hairy Jordan curve.
\end{itemize}
\end{propo}

\begin{proof}
One may repeat the proof of \refP{P:b_n-liminfs}, replacing $b_n$ with $\prescript{y}{}b_n$, and using the 
relations in \refE{E:b_n^j-cont-periodic-general} instead of the corresponding ones in \refE{E:b_n^j-cont-periodic}.
(Indeed, one only needs to verify the first and the last paragraphs in the proof of \refP{P:b_n-liminfs}.) 
That leads to the properties 
\begin{itemize}
\item[(a)] for all $x\in [0, 1/\ga_n)$, $\liminf_{s\to x^+} \prescript{y}{}{b}_n(s)= \prescript{y}{}{b}_n(x)$;
\item[(b)] for all $x\in (0, 1/\ga_n]$, $\liminf_{s\to x^-} \prescript{y}{}{b}_n(s)= \prescript{y}{}{b}_n(x)$; 
\end{itemize}
for each $n\geq -1$ and $y\geq 0$. 
In the same way, as in the proof of \refP{P:b_n-p_n-dense-touches}, we also note that for every $n\geq -1$, 
\begin{itemize}
\item[(c)] if $\ga \notin \E{B}$, each of $\prescript{y}{}b_n=+\infty$ and $\prescript{y}{}b_n < +\infty$ 
hold on a dense subset of 
$[0,1/\ga_n]$;
\item[(d)] if $\ga \in \E{B}\setminus \E{H}$ and $y < p_{-1}(0)$, 
each of $\prescript{y}{}b_n=p_n$ and $\prescript{y}{}b_n < p_n$ hold on a dense subset of $[0, 1/\ga_n]$.
\end{itemize}
Then, one may repeat the proof of \refT{T:model-trichotomy-thm}, using the above properties (a)--(d) 
to establish the desired result. 
\end{proof}

\subsection{Dependence on the parameter}

\begin{propo}\label{P:invariants-cont-depen}
For every $\ga \in \D{R} \setminus \D{Q}$, the map $t \mapsto \prescript{t}{}{\mathbb{A}}_\ga$, $t \in [r_\ga, 1]$, 
is continuous with respect to the Hausdorff metric on compact subsets of $\D{C}$.
\end{propo}

\begin{proof}
By the definition of $\prescript{t}{}{\mathbb{A}}_\ga$, it is enough to show that the map 
$y \mapsto \prescript{y}{}I_{-1}$, for $y\geq 0$, is continuous with respect to the Hausdorff distance. 
To that end, we need to show that for any $y \geq 0$ and any $\gep >0$, there is $\delta>0$ such that 
if $|x-y| \leq \delta$ and $x\geq 0$, then $\prescript{x}{}I_{-1} \subset B_\gep (\prescript{y}{}I_{-1})$
and $\prescript{y}{}I_{-1} \subset B_\gep (\prescript{x}{}I_{-1})$. 
Here, $B_\gep(\prescript{z}{}I_{-1})$ denotes the $\gep$-neighbourhood of $\prescript{z}{}I_{-1}$. 
Fix an arbitrary $y \geq 0$ and $\gep >0$.

We claim that 
\begin{itemize}
\item[(i)] $\prescript{y}{}I_{-1}= \overline{\cup_{x>y} \prescript{x}{}I_{-1}}$, 
\item[(ii)] if $y> 0$, $\prescript{y}{}I_{-1} = \cap_{x<y} \prescript{x}{}I_{-1}$. 
\end{itemize}
To prove these, it is enough to show that for all $t \in [0, 1]$, $\prescript{x}{}b_{-1}(t) \to \prescript{y}{}b_{-1}(t)$, 
as $x \to y$. 
Let $(x_n)_{n \geq -1}$ and $(y_n)_{n\geq -1}$ denote the sequences defined according to 
\refE{E:invariants-imaginary-traces} for the values $x$ and $y$ respectively. 
If $|x-y|$ is small enough, $|x_n - y_n| \leq 1$. 
The inequality $|\prescript{x}{}b_n^0 - \prescript{y}{}b_n^0| \leq 1$ and the uniform contraction in 
\refL{L:uniform-contraction-Y_r} imply that 
$|\prescript{x}{}b_{-1}^{n+1} - \prescript{y}{}b_{-1}^{n+1}| \leq  0.9 ^{n+1}$. 
Since $\prescript{x}{}b_{-1}^n(t) \to \prescript{x}{}b_{-1}(t)$ and 
$\prescript{y}{}b_{-1}^n(t) \to \prescript{y}{}b_{-1}(t)$, as $n\to \infty$, one infers that 
$\prescript{x}{}b_{-1}(t) \to \prescript{y}{}b_{-1}(t)$, as $x \to y$.

By property (i) above, $\cup_{x>y} B_\gep (\prescript{x}{}I_{-1})$ provides an open cover of $\prescript{y}{}I_{-1}$. 
Note that for any $x>y$, $B_\gep (\prescript{x}{}I_{-1})$ covers all points in $\prescript{0}{}I_{-1}$ 
above some imaginary value. 
It follows (choose a finite cover) that there is $x_0 > y$ such that if $y < x \leq x_0$, 
$\prescript{y}{}I_{-1} \subset B_\gep (\prescript{x}{}I_{-1})$. 
On the other hand, if $y>0$, property (ii) implies that there is $x_1 < y$ such that for all $x_1 < x < y$, 
we have $\prescript{x}{}I_{-1} \subset B_\gep (\prescript{y}{}I_{-1})$. 

Let $\delta = \min \{|x_0-y|, |x_1-y|\}$. 
Below assume that $|x-y| \leq \delta$, for some $x \geq 0$. 

If $x< y$, by \refE{E:invariant-sets-lift-nest}, $\prescript{y}{}I_{-1} \subset \prescript{x}{}I_{-1}$ and hence 
$\prescript{y}{}I_{-1} \subset B_\gep(\prescript{x}{}I_{-1})$. 
If $x>y$, by the above paragraph, we have $\prescript{y}{}I_{-1} \subset B_\gep(\prescript{x}{}I_{-1})$. 

Similarly, if $x > y$, by \refE{E:invariant-sets-lift-nest}, $\prescript{x}{}I_{-1} \subset \prescript{y}{}I_{-1}$ and hence 
$\prescript{x}{}I_{-1} \subset B_\gep(\prescript{y}{}I_{-1})$. 
If $x< y$, by the above paragraphs, $\prescript{x}{}I_{-1} \subset B_\gep (\prescript{y}{}I_{-1})$. 
\end{proof}

It would be interesting to know the modulus of continuity of the map 
$t \mapsto \prescript{t}{}{\mathbb{A}}_\ga$, $t \in [r_\ga, 1]$.

This completes the proof of \refT{T:dynamics-on-model}. 
More precisely, part (i) is proved in \refP{P:model-recurrent}. 
Part (ii) follows from Propositions \ref{P:invariants-ordered-model} and \ref{P:invariants-cont-depen}. 
Part (iii) follows from \refP{P:invariants-classified-model} and \refE{E:invariants-model-nested}. 
Parts (iv)  and (v) follow from Propositions \ref{P:invariants-classified-model} and \ref{P:invariants-topology-model}. 
Here, when $\ga \in \E{H}$, we have $r_\ga=1$ and there is nothing to prove. 

\subsection*{Acknowledgement}
I would like to thank Xavier Buff and Arnaud Cheritat for sharing with me the manuscript \cite{BC09},
and also for many useful discussions around the topic during the author's visit to Univeriste Paul Sabatier 
in 2012. 
I would also like to thank the anonymous referees for helpful comments to improve the 
presentation of the paper.
I am grateful to EPSRC (UK) for the fellowship EP/M01746X/1, which allowed carrying 
out this project.  
\bibliographystyle{amsalpha}
\bibliography{/Users/dcheragh/Work/Inscriptions/Data}
\end{document}